\documentclass[10pt,reqno,a4paper]{amsart}

\title[Lagrangian scheme for nonlinear diffusion]{A Lagrangian scheme
  for the solution of nonlinear diffusion equations using moving simplex meshes}

\author{Jos\'{e} A. Carrillo}\address{Jos\'{e} A. Carrillo, Department of Mathematics, Imperial College London, London, SW7 2AZ, UK}
\author{Bertram D\"{u}ring}
\address{Bertram D\"{u}ring, Department of Mathematics, University of
  Sussex, Pevensey II, Brighton, BN1 9QH, UK}
\author{Daniel Matthes}
\address{Daniel Matthes \\ Technische Universit\"at M\"unchen \\ Zentrum Mathematik \\ Boltzmannstra\ss e 3 \\ D-85747 Garching}
\author{David S. McCormick}\address{David S. McCormick, Department of Mathematics, University of
  Sussex, Pevensey II, Brighton, BN1 9QH, UK}
\thanks{DM was supported by the DFG 
  Collaborative Research Center TRR 109, ``Discretization in Geometry and Dynamics''. 
  BD and DSMcC were supported by the Leverhulme Trust 
  research project grant ``Novel discretisations for higher-order nonlinear PDE'' (RPG-2015-69). }

\date{\today}

\usepackage{geometry}
\usepackage{amssymb}
\usepackage{mathrsfs} 
\usepackage{graphicx}
\usepackage{dsfont}
\usepackage{esint}
\usepackage{placeins}

\usepackage{color}

\graphicspath{{./figures/}}

\newcommand{\ee}{\mathrm{e}}

\newcommand{\R}{\mathbb{R}}
\newcommand{\Rnn}{\mathbb{R}_{\ge0}}
\newcommand{\Rp}{\mathbb{R}_{>0}}

\newcommand{\myset}[2]{\left\{#1\,;\,#2\right\}}

\newcommand{\eps}{\varepsilon}

\newcommand{\weakstarto}{\overset{*}{\rightharpoonup}}

\newcommand{\pd}{\partial}

\newcommand{\df}{\mathrm{D}}
\newcommand{\dn}{\mathrm{d}}
\newcommand{\dd}{\,\mathrm{d}}
\newcommand{\Grad}{\nabla}
\newcommand{\Div}{\nabla \cdot}

\newcommand{\Laplace}{\Delta}

\newcommand{\nml}{\nu}
\newcommand{\stdtri}{\mathord{\bigtriangleup}}
\newcommand{\triang}{\mathscr{T}}
\newcommand{\vecG}{\vec{G}}

\newcommand{\prb}{\mathcal{P}}
\newcommand{\dens}{\mathcal{P}_2^\text{ac}(\R^d)}

\newcommand{\ansatz}{\mathcal{A}_\mathscr{T}}
\newcommand{\ansatzeps}{\mathcal{A}_{\mathscr{T}_\eps}}
\newcommand{\Gansatz}{(\R^d)_\mathscr{T}^L}
\newcommand{\test}{\mathcal{D}}

\newcommand{\refrho}{\overline\rho}
\newcommand{\intrho}{\widetilde\rho}
\newcommand{\velo}{\mathbf{v}}
\newcommand{\bigvelo}{\mathbf{V}}
\newcommand{\intvelo}{\widetilde{\velo}}

\newcommand{\plan}{\gamma}
\newcommand{\intplan}{\widetilde{\gamma}}
\newcommand{\intM}{\widetilde{M}}

\newcommand{\id}{\mathrm{id}}
\newcommand{\idm}{\mathds{1}}
\newcommand{\bigO}{\mathcal{O}}
\newcommand{\jm}{\mathbb{J}}
\newcommand{\disc}{\boxplus}
\newcommand{\aint}{\fint}
\newcommand{\zs}{0}
\newcommand{\BB}{\mathcal{B}}
\newcommand{\impulse}{\mathbf{J}}
\newcommand{\momentum}{\mathbf{p}}
\newcommand{\trap}{\top\!}

\newcommand{\I}{\mathrm{I}}
\newcommand{\II}{\mathrm{II}}

\newcommand{\ent}{\mathbf{E}}
\newcommand{\aent}{\mathcal{E}}
\newcommand{\entdens}{\mathbb{H}}

\newcommand{\distdens}{\mathbb{L}}
\newcommand{\potdens}{\mathbb{V}}

\newcommand{\wass}{\mathrm{W}_2}

\DeclareMathOperator*{\argmin}{argmin}
\DeclareMathOperator{\tr}{tr}

\theoremstyle{plain}
\newtheorem{thm}{Theorem}[section]

\newtheorem{lemma}[thm]{Lemma}
\newtheorem{rmk}[thm]{Remark}
\newtheorem{cor}[thm]{Corollary}

\numberwithin{equation}{section}

\makeatletter
\newtheorem*{rep@theorem}{\rep@title}
\newcommand{\newreptheorem}[2]{\newenvironment{rep#1}[1]{\def\rep@title{#2 \ref{##1}}\begin{rep@theorem}}{\end{rep@theorem}}}
\makeatother

\newreptheorem{thm}{Theorem}

\begin{document}

\begin{abstract}
  A Lagragian numerical scheme for solving nonlinear degenerate Fokker-Planck equations
  in space dimensions $d\ge2$ is presented. It applies to a large class of nonlinear
diffusion equations, whose dynamics are driven by internal energies
and given external potentials, e.g.\ the porous medium equation and
the fast diffusion equation.
  The key ingredient in our approach is the gradient flow structure of the dynamics.
  For discretization of the Lagrangian map,
  we use a finite subspace of linear maps in space
  and a variational form of the implicit Euler method in time.
  Thanks to that time discretisation,
  the fully discrete solution inherits energy estimates from the original gradient flow,
  and these lead to weak compactness of the trajectories in the continuous limit.
  Consistency is analyzed in the planar situation, $d=2$.
  A variety of numerical experiments for the porous medium equation indicates
  that the scheme is well-adapted to track the growth of the solution's support.
\end{abstract}

\maketitle

\section{Introduction}
We study a Lagrangian discretization of the following type of initial value problem
for a nonlinear Fokker--Planck equation:
\begin{subequations}
\label{eq:NFPsystem}
\begin{align}
  \label{eq:NFP}
  \pd_{t} \rho &= \Laplace P(\rho) + \Div (\rho\,\nabla V) && \text{on $\Rp\times\R^d$}, \\
  \label{eq:NFPic}
  \rho(\cdot,0) &= \rho^0 && \text{on $\R^d$}.
\end{align}
\end{subequations}
This problem is posed for the time-dependent probability density function $\rho \colon \Rnn\times\R^d\to \Rnn$,
with initial condition $\rho^0\in\dens$.
We assume that the pressure $P \colon \Rnn\to\Rnn$ can be written in the form
\begin{align}
  \label{eq:Ph}
  P(r) = rh'(r)-h(r) \quad\text{for all $r \ge 0$},
\end{align}
for some non-negative and convex $h \in C^1(\Rnn)\cap
C^\infty(\Rp)$ and that $V\in
C^2(\R^d)$ is a non-negative potential without loss of generality. 

Problem~\eqref{eq:NFPsystem} encompasses a large class of
diffusion equations, such as the heat equation ($P(r)=r, V=0$),
the porous medium equation ($P(r)=r^m/(m-1), m>1, V=0$) and the fast
diffusion equation ($P(r)=r^m/(m-1), m<1, V=0$), and extends to related
problems with given external potentials $V$. To motivate our discretization, we first briefly recall the Lagrangian form of the dynamics:
rewriting~\eqref{eq:NFPsystem} as a transport equation, we obtain
\begin{subequations}
\label{eq:LagrSystem}
\begin{align}
  \label{eq:trapo}
  \pd_t\rho + \Div\big(\rho\,\velo[\rho]\big) = 0,
\intertext{with a velocity field $\velo$ that depends on the solution $\rho$ itself,}
  \label{eq:velo}
  \velo[\rho] = -\nabla\big(h' (\rho)+V\big).
\end{align}
\end{subequations}
More general systems can be written in this form. For instance, interaction potentials leading to aggregation equations, relativistic heat equations, $p$-Laplacian equations and Keller-Segel type models can also be included; see Carrillo and Moll~\cite{art:CaMo09} and the references therein for a good account of models enjoying this structural form. Here, we reduce to models with nonlinear degenerate diffusion, i.e. $h(0)=h'(0)=0$, and confining potentials in order to explore our new discretization.

The system~\eqref{eq:LagrSystem} naturally induces a Lagrangian representation of the dynamics, which can be summarized as follows. Below, we use the notation $G_\#\refrho$ for the \emph{push-forward} of $\refrho$ under a map $G \colon \R^d\to\R^d$;
the definition is recalled in~\eqref{eq:pushdens}.
\begin{lemma}
  \label{lem:Lagrange}
  Assume that $\rho \colon [0,T]\times\R^d\to\Rnn$ is a smooth positive solution of~\eqref{eq:NFPsystem}.
  Let $\refrho \colon \R^d\to\Rnn$ be a given reference density,
  and let $G^0 \colon \R^d\to\R^d$ be a given map such that $G^0_\#\refrho=\rho^0$.
  Further, let $G \colon [0,T]\times \R^d\to\R^d$ be the flow map associated to~\eqref{eq:velo},
  satisfying
  \begin{align}
    \label{eq:Lagrange}
    \pd_tG_t = \velo[\rho_t]\circ G_t, \quad G(0,\cdot)=G^0,
  \end{align}
  where $\rho_t:=\rho(t,\cdot)$ and $G_t:=G(t,\cdot) \colon \R^d \to\R^d$.
  Then
  \begin{equation}
    \label{eq:push}
    \rho_t = (G_t)_{\#} \refrho 
  \end{equation}
  at any $t\in[0,T]$.
\end{lemma}
In short, the solution $G$ to~\eqref{eq:Lagrange} is a Lagrangian map for the solution $\rho$ to~\eqref{eq:NFPsystem}.
This fact is an immediate consequence of~\eqref{eq:trapo};
for convenience of the reader, we recall the proof in Appendix~\ref{sct:Lagrange}.
Notice that~\eqref{eq:push} can be substituted for $\rho$ in the expression~\eqref{eq:velo} for the velocity,
which makes~\eqref{eq:Lagrange} an autonomous evolution equation for $G$:
\begin{align}
  \label{eq:GG}
  \pd_tG_t = -\nabla\left[h'\left(\frac{\refrho}{\det\df G_t}\right)\right]\circ G_t-\nabla V\circ G_t.
\end{align}
A more explicit form of~\eqref{eq:GG} is derived in~\eqref{eq:Geq}.

There is a striking structural relation between~\eqref{eq:NFPsystem} and~\eqref{eq:GG}:
it is well-known (see Otto~\cite{art:Otto} or Ambrosio, Gigli and Savar\'e~\cite{book:AGS}) that~\eqref{eq:NFPsystem} is a gradient flow
for the relative Renyi entropy functional
\begin{align}
  \label{eq:aent}
  \aent(\rho ) = \int_{\R^d}\big[ h(\rho(x))+V(x)\rho(x)\big]\dd x,
\end{align}
with respect to the $L^2$-Wasserstein metric on the space $\dens$
of probability densities on ${\R^d}$. It appears to be less well known
(see Evans et~al.~\cite{art:EGS}, Carrillo and Moll~\cite{art:CaMo09}, or Carrillo and Lisini~\cite{art:CaLi10}) that also~\eqref{eq:GG} is a gradient
flow, namely for the functional
\begin{align}
  \label{eq:ent}
  \ent(G|\refrho) := \aent(G_\#\refrho)
  = \int_K \left[\widetilde h\left(\frac{\det\df G}{\refrho}\right) + V\circ G\right]\refrho\dd\omega,
  \quad \widetilde h(s):=s\,h(s^{-1}),
\end{align}
on the Hilbert space $L^2(K\to{\R^d};\refrho)$ of maps from $K$ to ${\R^d}$, where $\refrho$ is a reference measure supported on $K\subset \R^d$.
We shall discuss these gradient flow structures in more detail in Section~\ref{sec:GF} below.

The particular spatio-temporal discretization of the initial value
problem~\eqref{eq:NFPsystem} that we study in this
paper is based on these facts. Instead of numerically integrating~\eqref{eq:NFP}
to obtain the density $\rho$ directly, we
approximate the Lagrangian map $G$ --- using a finite subspace of linear maps in space --- which then allows us to recover $\rho$ a
posteriori via~\eqref{eq:push}. Our approximation for $G$ is
constructed by solving a succession of minimization
problems that are naturally derived from the gradient flow
structure behind~\eqref{eq:Lagrange}. Via variational methods,
this provides compactness estimates on the discrete solutions. 

The use of a finite subspace of linear maps for the Lagrangian maps has a geometric interpretation: the induced densities are piecewise constant on triangles whose vertices move in time. 
In the sequel, we further assume that $\lim_{s\to\infty}sh''(s)=+\infty$ in order to prevent the collapse of the images of our Lagrangian maps $G_t$, as proven below. However, we do not yet know how to prevent (globally) the possible intersection of the images of the approximated Lagrangian maps.

This approach is alternative to the one developed by Carrillo et~al.~\cite{art:CRW,art:CaMo09}, where $G$ is obtained by directly solving
the PDE~\eqref{eq:GG} numerically with finite differences or Galerkin approximation via finite element methods. In fact, the main difference between these two strategies can be summarized as follows: while
Carrillo et~al.~\cite{art:CRW,art:CaMo09} follows the strategy \emph{minimize first
then discretize}, our present approach is to \emph{discretize first
then minimize}. In other words, in Carrillo et~al.~\cite{art:CRW,art:CaMo09} one
minimizes first, obtaining as Euler-Lagrange equations the 
implicit Euler discretization of~\eqref{eq:GG} in \cite{art:CRW}
approximated by the explicit Euler method in \cite{art:CaMo09}, and then
discretized in space. In the present approach, we discretize first
approximating the space of Lagrangian maps by a suitable finite
subspace of linear maps, and then we minimize obtaining a
nonlinear system of equations to find the approximated
Lagrangian map within that set of linear maps.

Let us mention that other numerical methods have been developed to conserve particular properties of solutions of the gradient flow~\eqref{eq:NFPsystem}.
Finite volume methods preserving the decay of energy at the semi-discrete level, along with other important properties like non-negativity and mass conservation, were proposed in the papers~\cite{Filbet,CCH2}.
Particle methods based on suitable regularisations of the flux of the continuity equation~\eqref{eq:NFPsystem} have been proposed in the papers~\cite{DM,LM,MGallic,Russo}. A particle method based on the steepest descent of a regularized internal part of the energy~$\aent$ in~\eqref{eq:aent} by substituting particles by non-overlapping blobs was proposed and analysed in Carrillo et~al.~\cite{CHPW,CPSW}.
Moreover, the numerical approximation of the JKO variational scheme has already been tackled by different methods using pseudo-inverse distributions in one dimension (see~\cite{Blanchet,CG,GT,westdickenberg2010variational}) or solving for the optimal map in a JKO step (see~\cite{art:BCMO,art:JMO}). Finally, note that gradient-flow-based Lagrangian methods in one dimension for higher-order, drift diffusion and Fokker-Planck equations have recently been proposed in the papers~\cite{during2010gradient,MO3,MO1,MO2}.

There are two main arguments in favour of our taking this indirect
approach of solving~\eqref{eq:GG} instead of solving~\eqref{eq:NFPsystem}.
The first is our interest in \emph{structure
preserving discretizations}: the scheme that we present builds on
the non-obvious ``secondary'' gradient flow representation of~\eqref{eq:NFPsystem}
in terms of Lagrangian maps. The benefits are
monotonicity of the transformed entropy functional~$\ent$ and an
$L^2$-control on the metric velocity for our fully discrete
solutions, that eventually lead to weak compactness of the
trajectories in the continuous limit.
We remark that our long-term goal is to design a numerical scheme that makes full use
of the much richer ``primary'' variational structure of~\eqref{eq:NFPsystem} in the Wasserstein distance,
that is reviewed in Section~\ref{sec:GF} below.
However, despite significant effort in the recent past --- see, e.g., the references \cite{art:BCMO,art:BCC,CPSW, art:CRW,during2010gradient,art:GT,art:JMO,art:MO1,art:Peyre,westdickenberg2010variational} ---
it has not been possible so far to preserve features like metric contractivity of the flow under the discretization,
except in the rather special situation of one space dimension (see Matthes and Osberger~\cite{art:MO1}).
This is mainly due to the non-existence of finite-dimensional submanifolds of $\dens$
that are complete with respect to generalized geodesics.

The second motivation is that Lagrangian schemes are a natural choice for \emph{numerical front tracking},
see, e.g., Budd~\cite{art:Budd} for first results
on the numerical approximation of self-similar solutions to the porous medium equation.
We recall that due to the assumed degeneracy $P'(0)=0$ of the diffusion in~\eqref{eq:NFPsystem},
solutions that are compactly supported initially remain compactly supported at all times.
A numerically accurate calculation of the moving edge of support is challenging,
since the solution can have a very complex behavior near that edge,
like the waiting time phenomenon (see Vazquez~\cite{book:Vazquez}).
Our simulation results for $\partial_t\rho=\Delta(\rho^3)$
--- that possess an analytically known, compactly supported, self-similar Barenblatt solution ---
indicated that our discretization is indeed able to track the edge of support quite accurately.

This work is organized as follows. In Section~\ref{sec:GF} we present an
overview of previous results in gradient flows pertaining our
work. Section~\ref{sec:SchemeDefn} is devoted to the introduction of the linear set
of Lagragian maps and the derivation of the numerical scheme.
Section~\ref{sec:LimitTraj} shows the compactness of the approximated sequences of
discretizations and we give conditions leading to the eventual
convergence of the scheme towards~\eqref{eq:NFPsystem}. Section~\ref{sec:Consistency2D} deals
with the consistency of the scheme in two dimensions while Section~\ref{sec:NumSim2D} gives several numerical tests showing the performance of this
scheme.

\section{Gradient flow structures}
\label{sec:GF}
\subsection{Notations from probability theory}
$\prb(X)$ is the space of probability measures on a given base set
$X$. We say that a sequence $(\mu_n)$ of measures in $\prb(X)$
\emph{converges narrowly} to a limit $\mu$ in that space if
\[
\int_Xf(x)\dd\mu_n(x)\to\int_Xf(x)\dd\mu(x)
\]
for all bounded and continuous functions $f\in C^0_b(X)$. The
\emph{push-forward} $T_\#\mu$ of a measure $\mu\in\prb(X)$ under a
measurable map $T \colon X\to Y$ is the uniquely determined measure
$\nu\in\prb(Y)$ such that, for all $g\in C^0_b(Y)$,
\begin{align*}
  \int_Xg\circ T(x)\dd\mu(x) = \int_Yg(y)\dd\nu(y).
\end{align*}
With a slight abuse of notation
--- identifying absolutely continuous measures with their densities ---
we denote the space of probability densities on ${\R^d}$ of finite
second moment by
\begin{align*}
  \dens = \myset{\rho\in L^1({\R^d})}{\rho\ge0,\,\int_{\R^d}\rho(x)\dd x=1,\,\int_{\R^d}\|x\|^2\rho(x)\dd x<\infty}.
\end{align*}
Clearly, the reference density $\refrho$, which is supported on
the compact set $K\subset\R^d$, belongs to $\dens$. If
$G \colon K\to\R^d$ is a diffeomorphism onto its image (which is again
compact), then the push-forward of $\refrho$'s measure produces
again a density $G_\#\refrho\in\dens$, given by
\begin{align}
  \label{eq:pushdens}
  G_\#\refrho = \frac{\refrho}{\det \df G}\circ G^{-1}.
\end{align}

\subsection{Gradient flow in the Wasserstein metric}
Below, some basic facts about the Wasserstein metric
and the formulation of~\eqref{eq:NFPsystem} as gradient flow in that metric are briefly reviewed.
For more detailed information, we refer the reader to the monographs of Ambrosio et~al.~\cite{book:AGS} and Villani~\cite{book:Villani}.

One of the many equivalent ways to define the \emph{$L^2$-Wasserstein distance} between $\rho_0,\rho_1\in\dens$
is as follows:
\begin{align}
  \label{eq:W2}
  \wass(\rho_0,\rho_1) :=
  \inf\myset{\int_{\R^d} \|T(x)-x\|^2\rho_0(x)\dd x}{T:{\R^d}\to{\R^d}\ \text{measurable},\,T_\#\rho_0=\rho_1}^{\frac12}.
\end{align}
The infimum above is in fact a minimum,
and the --- essentially unique --- optimal map $T^*$ is characterized by Brenier's criterion; see, e.g., Villani~\cite[Section 2.1]{book:Villani}.
A trivial but essential observation is that
if $\refrho\in\dens$ is a reference density with support $K\subset{\R^d}$,
and $\rho_0=(G_0)_\#\refrho$ with a measurable $G_0 \colon K\to{\R^d}$,
then~\eqref{eq:W2} can be re-written as follows:
\begin{align}
  \label{eq:W3}
  \wass(\rho_0,\rho_1) =
  \inf\myset{\int_K \|G(\omega)-G_0(\omega)\|^2\refrho(\omega)\dd\omega}
  {G \colon K\to{\R^d}\ \text{measurable},\,G_\#\refrho=\rho_1}^{\frac12},
\end{align}
and the essentially unique minimizer $G^*$ in~\eqref{eq:W3}
is related to the optimal map $T^*$ in~\eqref{eq:W2} via $G^*=T^*\circ G_0$.

$\wass$ is a metric on $\dens$;
convergence in $\wass$ is equivalent to weak-$\star$ convergence in $L^1({\R^d})$
and convergence of the second moment.
Since $P$ and hence also $h$ are of super-linear growth at infinity,
each sublevel set $\aent$ is weak-$\star$ closed
and thus complete with respect to $\wass$.

As already mentioned above, solutions $\rho$ to~\eqref{eq:NFPsystem}
constitute a gradient flow for the functional $\aent$ from
\eqref{eq:aent} in the metric space $(\dens;\wass)$. In fact, the
flow is even $\lambda$-contractive as a semi-group, thanks to the
$\lambda$-uniform displacement convexity of $\aent$
(see McCann~\cite{art:McCann}, or Daneri and Savar\'e~\cite{art:DS}), which is a strengthened form of
$\lambda$-uniform convexity along geodesics.
The $\lambda$-contractivity of the flow implies various properties
(see Ambrosio et~al.~\cite[Section 11.2]{book:AGS})
like global existence, uniqueness and regularity
of the flow, monotonicity of $\aent$ and its sub-differential,
uniform exponential estimates on the convergence (if $\lambda<0$)
or divergence (if $\lambda\ge0$) of trajectories, quantified
exponential rates for the approach to equilibrium (if $\lambda<0$)
and the like.

An important further consequence is that the unique flow can be obtained as the limit for $\tau \searrow0$
of the time-discrete \emph{minimizing movement scheme} (see Ambrosio et~al.~\cite{book:AGS} and Jordan, Kinderlehrer and Otto~\cite{art:JKO}):
\begin{equation}
  \label{eq:mmrho}
  \rho_\tau^{n} := \argmin_{\rho \in \dens}\aent_\tau(\rho;\rho_\tau^{n-1}),
  \quad \aent_\tau(\rho,\hat\rho):=\frac{1}{2\tau}\wass(\rho,\hat\rho)^{2} + \aent(\rho).
\end{equation}
This time discretization is well-adapted to approximate $\lambda$-contractive gradient flows.
All of the properties of mentioned above
are already reflected on the level of these time-discrete solutions.

\subsection{Gradient flow in $L^2$}
Equation~\eqref{eq:GG} is the gradient flow of $\ent$ on the space $L^2(K\to{\R^d};\refrho)$
of square integrable (with respect to $\refrho$) maps $G \colon K\to{\R^d}$
(see Evans et~al.~\cite{art:EGS} or Jordan et~al.~\cite{art:JMO}).
However, the variational structure behind this gradient flow is much weaker than above:
most notably, $\ent$ is only poly-convex, but \emph{not $\lambda$-uniformly convex}.
Therefore, the abstract machinery for $\lambda$-contractive gradient flows in Ambrosio et~al.~\cite{book:AGS} does not apply here.
Clearly, by equivalence of~\eqref{eq:NFPsystem} and~\eqref{eq:GG} at least for sufficiently smooth solutions,
certain properties of the primary gradient flow are necessarily inherited by this secondary flow,
but for instance $\lambda$-contractivity of the flow in the $L^2$-norm seems to fail.

Nevertheless, it can be proven (see Ambrosio, Lisini and Savar\'e~\cite{art:ALS}) that the gradient
flow is globally well-defined, and it can again be approximated by
the minimizing movement scheme:
\begin{align}
  \label{eq:mmG}
  G_\tau^n:=\argmin_{G\in L^2(K\to{\R^d};\refrho)}\ent_\tau\big(G;G_\tau^{n-1}\big),
  \quad
  \ent_\tau(G;\hat G)=
  \frac1{2\tau}\int_K\|G-\hat G\|^2\,\dd\refrho + \ent(G|\refrho).
\end{align}
In fact, there is an equivalence between~\eqref{eq:mmG} and
\eqref{eq:mmrho}: simply substitute $(G_\tau^{n-1})_\#\refrho$ for
$\rho_\tau^{n-1}$ and $G_\#\refrho$ for $\rho$ in
\eqref{eq:mmrho}; notice that any $\rho\in\dens$ can be written as
$G_\#\refrho$ with a suitable (highly non-unique) choice of $G\in
L^2(K\to{\R^d};\refrho)$. This equivalence was already
exploited in Carrillo et~al.~\cite{art:CRW,art:CaMo09}.
Thanks to the equality~\eqref{eq:W3},
the minimization with respect to $\rho=G_\#\refrho$ can be relaxed
to a minimization with respect to $G$. Consequently, if
$(G_\tau^0)_\#\refrho=\rho_\tau^0$, then
$(G_\tau^n)_\#\refrho=\rho_\tau^n$ at all discrete times
$n=1,2,\ldots$. However, while the functional
$\aent_\tau(\cdot;\rho_\tau^{n-1})$ in~\eqref{eq:mmrho} is
$(\lambda+\tau^{-1})$-uniformly convex in $\rho$ along geodesics
in $\wass$, the functional $\ent_\tau(\cdot;G_\tau^{n-1})$ in
\eqref{eq:mmG} has apparently no useful convexity properties in
$G$ on $L^2(K\to{\R^d};\refrho)$.

\section{Definition of the numerical scheme}
\label{sec:SchemeDefn}
Recall the Lagrangian formulation of~\eqref{eq:NFPsystem} that has been given in Lemma~\ref{lem:Lagrange}.
For definiteness, fix a reference density $\refrho\in\dens$,
whose support $K\subset{\R^d}$ is convex and compact.

\subsection{Discretization in space}
Our spatial discretization is performed using a finite subspace of linear maps for the Lagrangian maps $G$.
More specifically:
let $\triang$ be some (finite) simplicial decomposition of $K$ 
with nodes $\omega_{1}$ to $\omega_{L}$ and $n$-simplices $\Delta_{1}$ to $\Delta_{M}$.
In the case $d=2$, which is of primary interest here, $\triang$ is a triangulation, with triangles $\Delta_m$.
The reference density $\refrho$ is approximated by a density $\refrho_\triang\in\dens$
that is piecewise constant on the simplices of $\triang$,
with respective values
\begin{align}
  \label{eq:mu}
  \refrho_\triang^m := \frac{\mu_\triang^m}{|\Delta_m|}
  \quad \text{for the simplex masses}\quad
  \mu_\triang^m:=\int_{\Delta_m}\refrho(\omega)\dd\omega.
\end{align}
The finite dimensional ansatz space $\ansatz$ is now defined as the set of maps $G\colon K\to{\R^d}$ that are
globally continous, affine on each of the simplices $\Delta_m\in\triang$, and orientation preserving.
That is, on each $\Delta_m\subset\triang$, the map $G\in\ansatz$ can be written as follows:
\begin{align}
  \label{eq:Gform}
  G(\omega) = A_{m} \omega + b_{m} \qquad \text{for all } \omega \in \Delta_{m},
\end{align}
with a suitable matrix $A_{m} \in \R^{d\times d}$ of positive determinant and a vector $b_{m} \in \R^d$.

For the calculations that follow, we shall use a more geometric way to describe the maps $G\in\ansatz$,
namely by the positions $G_\ell=G(\omega_\ell)$ of the images of each node $\omega_\ell$.
Denote by $\Gansatz\subset\R^{L\cdot d}$ the space of $L$-tuples $\vecG=(G_\ell)_{\ell=1}^L$
of points $G_\ell\in{\R^d}$ with the same simplicial combinatorics (including orientation) as the $\omega_\ell$ in $\triang$.
Clearly, any $G\in\ansatz$ is uniquely characterized by the $L$-tuple $\vecG$ of its values,
and moreover, any $\vecG\in\Gansatz$ defines a $G\in\ansatz$.

More explicitly, fix a $\Delta_m\in\triang$,
with nodes labelled $\omega_{m,0}$ to $\omega_{m,d}$ in some orientation preserving order,
and respective image points $G_{m,0}$ to $G_{m,d}$.
With the standard $d$-simplex given by
\[
\stdtri^d:= \myset{\xi=(\xi_1,\ldots,\xi_d)\in\Rnn^d}{\sum_{j=1}^d\xi_j\le1},
\]
introduce the linear interpolation maps $r_m \colon \stdtri^d\to K$ and $q_m \colon \stdtri^d\to{\R^d}$ by
\begin{align*}
  r_m(\xi) &= \omega_{m,\zs} + \sum_{j=1}^d(\omega_{m,j}-\omega_{m,\zs})\xi_j, \\
  q_m(\xi) &= G_{m,\zs} + \sum_{j=1}^d(G_{m,j}-G_{m,\zs})\xi_j.
\end{align*}
Then the affine map~\eqref{eq:Gform} equals to $q_m\circ r_m^{-1}$.
In particular, we obtain that
\begin{align}
  \label{eq:det}
  \det A_m = \frac{\det\df q_m}{\det\df r_m} = \frac{\det Q_\triang^m[G]}{2|\Delta_m|}
  \quad\text{where}\quad
  Q_\triang^m[G]:=\big(G_{m,1}-G_{m,\zs}\big|\cdots\big|G_{m,d}-G_{m,\zs}\big).
\end{align}
For later reference, we give a more explicit representation
for the transformed entropy $\ent$ for $G\in\ansatz$,
and for the $L^2$-distance between two maps $G,\hat G\in\ansatz$.
Substitution of the special form~\eqref{eq:Gform} into~\eqref{eq:ent} produces
\begin{align}
  \label{eq:dent}
  \ent(G|\refrho_\triang)
  =\sum_{\Delta_m\in\triang}\mu_\triang^m \big[\entdens_\triang^m(G)+\potdens_\triang^m(G)\big]
\end{align}
with the internal energy (recall the definition of $\widetilde h$ from~\eqref{eq:ent})
\begin{align*}
  \entdens_\triang^m(G)
  :=\widetilde{h} \left( \frac{\det A_m}{\refrho_\triang^m} \right)
  = \widetilde h\left(\frac{\det Q_\triang^m[G]}{2\mu_\triang^m}\right)
\end{align*}
and the potential energy
\begin{align*}
  \potdens_\triang^m(G)
  = \aint_{\Delta_m}V(A_m\omega+b_m)\dd \omega
  = \aint_{\stdtri}V\big(r_m(\omega)\big)\dd\omega.
\end{align*}
For the $L^2$-difference of $G$ and $G^*$,
we have
\begin{align}
  \label{eq:ddist}
  \|G-G^*\|_{L^2(K;\refrho_\triang)}^2
  =\int_K\|G-G^*\|^2\refrho_\triang\dd \omega
  = \sum_{\Delta_m\in\triang}\mu_\triang^m \distdens_\triang^m(G,G^*).
\end{align}
Using Lemma~\ref{lem:triint}, we obtain on each simplex $\Delta_m$:
\begin{align}
  \nonumber
  \distdens_\triang^m(G,G^*)
  &:=\aint_{\Delta_m}\|G(\omega)-G^*(\omega)\|^2\dd\omega \\
  \nonumber
  &= \aint_{\stdtri}\|r_m(\omega)-r_m^*(\omega)\|^2\dd\omega \\
  \label{eq:distdens}
  &= \frac2{(d+1)(d+2)}\sum_{0\le i\le j\le d}(G_{m,i}-G^*_{m,i})\cdot(G_{m,j}-G^*_{m,j}).
\end{align}

\subsection{Discretization in time}
Let a time step $\tau>0$ be given;
in the following, we symbolize the spatio-temporal discretization by $\disc$,
and we write $\disc\to0$ for the joint limit of $\tau\to0$ and vanishing mesh size in $\triang$.

The discretization in time is performed in accordance with~\eqref{eq:mmG}:
we modify $\ent_\tau$ from~\eqref{eq:mmG} by restriction to the ansatz space $\ansatz$.
This leads to the minimization problem
\begin{align}
  \label{eq:mini}
  G_\disc^n:=\argmin_{G\in\ansatz}\ent_\disc\big(G;G_\disc^{n-1}\big)
  \quad\text{where}\quad
  \ent_\disc(G;G^*) = \frac1{2\tau}\|G-G^*\|_{L^2(K;\refrho_\triang)}^2 + \ent(G|\refrho_\triang).
\end{align}
For a fixed discretization $\disc$,
the fully discrete scheme is well-posed in the sense that
for a given initial map $G_\disc^0\in\ansatz$,
an associated sequence $(G_\disc^n)_{n\ge0}$ can be determined
by successive solution of the minimization problems~\eqref{eq:mini}.
One only needs to verify:
\begin{lemma}
\label{lem:minexist}
  For each given $G^*\in\ansatz$, there exists at least one global minimizer $G\in\ansatz$
  of $\ent_\disc(\cdot;G^*)$.
\end{lemma}
\begin{rmk}
  We \emph{do not} claim uniqueness of the minimizers.
  Unfortunately,
  the minimization problem~\eqref{eq:mini} inherits the lack of convexity from~\eqref{eq:mmG},
  whereas the correspondence between~\eqref{eq:mmG} and the convex problem~\eqref{eq:mmrho}
  is lost under spatial discretization.
  A detailed discussion of $\ent_\disc$'s (non-)convexity is provided in Appendix~\ref{sct:notconvex}.
\end{rmk}
\begin{proof}[Proof of Lemma~\ref{lem:minexist}]
  We only sketch the main arguments.
  For definiteness, let us choose (just for this proof) one of the
  infinitely many equivalent norm-induced metrics
  on the $dL$-dimensional vector space $V_\triang$ of all continuous maps $G \colon K\to\R^d$
  that are piecewise affine with respect to the fixed simplicial decomposition $\triang$:
  given $G,G'\in V_\triang$ with their respective point locations $\vecG,\vecG'\in\R^{dL}$,
  i.e., $\vecG=(G_\ell)_{\ell=1}^L$ for $G_\ell=G(\omega_\ell)$,
  define the distance between these maps
  as the maximal $\R^d$-distance $\|G_\ell-G'_\ell\|$ of corresponding points $G_\ell\in\vecG$, $G'_\ell\in\vecG'$.
  Clearly, this metric makes $V_\triang$ a complete space.

  It is easily seen that the subset $\ansatz$
  --- which is singled out by requiring orientation preservation of the $G$'s ---
  is an open subset of $V_\triang$.
  It is further obvious that the map $G\mapsto\ent_\disc(G;G^*)$
  is continuous with respect to the metric.
  The claim of the lemma thus follows if we can show that the sub-level
  \[ S_c:=\myset{G\in\ansatz}{\ent_\disc(G;G^*)\le c} \quad \text{with}\quad
    c:=\ent(G^*|\refrho_\triang) \]
  is a non-empty compact subset of $V_\triang$.
  Clearly, $G^*\in S_c$, so it suffices to verify compactness.

  \emph{$S_c$ is bounded.}
  We are going to show that there is a radius $R>0$ such that no $G\in S_c$ has a distance larger than $R$ to $G^*$.
  From non-negativity of $\ent$,
  and from the representations~\eqref{eq:ddist} and~\eqref{eq:distdens},
  it follows that
  \begin{align*}
    c\ge\frac1{2\tau}\|G-G^*\|_{L^2(K;\refrho_\triang)}^2
    &\ge\frac{\underline\mu_\triang}{2\tau}\sum_{\Delta_m\in\triang}\distdens_\triang^m(G,G^*) \\
    &= \frac{\underline\mu_\triang}{(d+1)(d+2) \tau}\sum_{0\le i\le j\le d}(G_{m,i}-G^*_{m,i})\cdot(G_{m,j}-G^*_{m,j})\\
    &\ge \frac{\underline\mu_\triang}{2(d+1)(d+2) \tau}\sum_{\ell=1}^L\|G_\ell-G_\ell^*\|^2,
  \end{align*}
  where $\underline\mu_\triang=\min_{\Delta_m}\mu_\triang^m$.
  It is now easy to compute a suitable value for the radius $R$.

  \emph{$S_c$ is a closed subset of $V_\triang$.}
  It suffices to show that the limit $\overline G\in V_\triang$
  of any sequence $(G^{(k)})_{k=1}^\infty$  of maps $G^{(k)}\in S_c$ belongs to $\ansatz$.
  By definition of our metric on $V_\triang$,
  global continuity and piecewise linearity of the $G^{(k)}$ trivially pass to the limit $\overline G$.
  We still need to verify that $\overline G$ is orientation-preserving.
  Fix a simplex $\Delta_m$ and consider the corresponding matrices $A_m^{(k)}$ and $\overline A_m$ from~\eqref{eq:Gform}.
  Since the $G^{(k)}$ converge to $\overline G$ in the metric, also $A_m^{(k)}\to\overline A_m$ entry-wise.
  Now, by non-negativity of $\widetilde h$, we have for all $k$ that
  \begin{align*}
    c\ge\ent(G^{(k)}|\refrho_\triang) \ge \mu_\triang^m\widetilde h\left(\frac{\det A^{(k)}_m}{\refrho_\triang^m}\right),
  \end{align*}
  and since $\widetilde h(s)\to+\infty$ as $s\downarrow0$,
  it follows that $\det A^{(k)}_m>0$ is bounded away from zero, uniformly in $k$.
  But then also $\det\overline A_m>0$, i.e., the $m$th linear map piece of the limit $\overline G$ preserves orientation.
\end{proof}

\subsection{Fully discrete equations}
We shall now derive the Euler-Lagrange equations associated to the minimization problem~\eqref{eq:mini},
i.e., for each given $G^*:=G_\disc^{n-1}\in\ansatz$,
we calculate the variations of $\ent_\disc(G;G^*)$ with respect to the degrees of freedom of $G\in\ansatz$.
Since that function is a weighted sum over the triangles $\Delta_m\in\triang$,
it suffices to perform the calculations for one fixed triangle $\Delta_m$,
with respective nodes $\omega_{m,0}$ to $\omega_{m,d}$, in positive orientation.
The associated image points are $G_{m,0}$ to $G_{m,d}$.
Since we may choose any vertex to be labelled $\omega_{m,0}$, it will suffice to perform the calculations at one fixed image point $G_{m,0}$.
%
\begin{itemize}
\item \emph{mass term:}
  \begin{align*}
    \frac{\pd}{\pd G_{m,0}}\distdens_\triang^m(G,G^*)
    &=\frac2{(d+1)(d+2)}\frac{\pd}{\pd G_{m,0}}\sum_{0\le i\le j\le d}(G_{m,i}-G^*_{m,i})\cdot(G_{m,j}-G^*_{m,j}) \\
    &=\frac2{(d+1)(d+2)}\left(2(G_{m,0}-G^*_{m,0})+\sum_{j=1}^d(G_{m,j}-G^*_{m,j})\right)
  \end{align*}
\item \emph{internal energy:}
  observing that --- recall~\eqref{eq:Ph} ---
  \begin{align}
    \label{eq:h2P}
    \widetilde h'(s) = \frac{\dn}{\dd s}\big[sh(s^{-1})\big] = h(s^{-1})-s^{-1}h'(s^{-1}) = -P(s^{-1}),
  \end{align}
  we obtain
  \begin{align*}
    \frac{\pd}{\pd G_{m,0}}\entdens_\triang^m(G)
    = \frac{\pd}{\pd G_{m,0}} \widetilde h\left(\frac{\det Q_\triang^m[G]}{2\mu_\triang^m}\right)
    = \frac1{2\mu_\triang^m}P\left(\frac {2\mu_\triang^m}{\det Q_\triang^m[G]}\right)\nml_\triang^m[G],
  \end{align*}
  where
  \begin{align}
    \label{eq:Pi}
    \nml_\triang^m[G]
    := -  \frac{\pd}{\pd G_{m,0}} \det Q_\triang^m[G]
    = (\det Q_\triang^m[G])\, (Q_\triang^m[G])^{-T}\sum_{j=1}^d\ee_j
  \end{align}
  is the uniquely determined vector in $\R^d$ that is
  orthogonal to the $(d-1)$-simplex with corners $G_{m,1}$ to $G_{m,d}$ (pointing \emph{away} from $G_{m,0}$)
  and whose length equals the $(d-1)$-volume of that simplex.
\item \emph{potential energy:}
  \begin{align*}
    \frac{\pd}{\pd G_{m,0}}\potdens_\triang^m(G)
    =\frac{\pd}{\pd G_{m,0}}\aint_{\stdtri} V\big(r_m(\xi)\big)\dd\xi
    = \aint_{\stdtri} \nabla V\big(r_m(\xi)\big)\,(1-\xi_1-\cdots-\xi_d)\dd\xi.
  \end{align*}
\end{itemize}
Now let $\omega_\ell$ be a fixed vertex of $\triang$.
Summing over all simplices $\Delta_m$ that have $\omega_\ell$ as a vertex,
and choosing vertex labels in accordance with above, i.e.,
such that $\omega_{m,0}=\omega_\ell$ in $\Delta_m$,
produces the following Euler-Lagrange equation:
\begin{align}
  \label{eq:EL}
  \begin{split}
    0 &= \sum_{\omega_\ell\in\Delta_m}\mu_\triang^m\bigg[
    \frac1{(d+1)(d+2)\tau}\bigg(2(G_{m,0}-G^*_{m,0})+\sum_{j=1}^d(G_{m,j}-G^*_{m,j})\bigg) \\
    & \qquad +\frac1{2\mu_\triang^m}P\left(\frac {2\mu_\triang^m}{\det Q_\triang^m[G]}\right)\nml_\triang^m[G]
    +  \aint_{\stdtri} \nabla V\big(r_m(\xi)\big)\,(1-\xi_1-\cdots-\xi_d)\dd\xi
    \bigg].
  \end{split}
\end{align}

\subsection{Approximation of the initial condition}
For the approximation $\rho^0_\disc=(G^0_\disc)_\#\refrho_\triang$ of the initial datum $\rho^0=G^0_\#\refrho$,
we require:
\begin{itemize}
\item $\rho^0_\disc$ converges to $\rho^0$ narrowly;
\item $\aent(\rho^0_\disc)$ is $\disc$-uniformly bounded, i.e.,
  \begin{align}
    \label{eq:maxaent}
    \overline\aent:=\sup\aent(\rho_\disc^0) < \infty.
  \end{align}
\end{itemize}
In our numerical experiments, we always choose $\refrho:=\rho^0$,
in which case $G^0 \colon K\to\R^d$ can be taken as the identity on $K$,
and we choose accordingly $G^0_\disc$ as the identity as well.
Hence $\rho^0_\disc=\refrho_\triang$, which converges to $\rho^0=\refrho$ even strongly in $L^1(K)$.
Moreover, since $h$ is convex, it easily follows from Jensen's inequality that
\[ \int_{\Delta_m} h\big(\refrho(x)\big)\dd x \ge |\Delta_m|h(\refrho_\triang^m), \]
and therefore,
\[ \aent(\rho_\disc^0) \le \aent(\rho^0). \]
In more general situations, in which $G^0$ is not the identity,
a sequence of approximations $G^0_\disc$ of $G^0$ is needed.
Pointwise convergence $G^0_\disc\to G^0$ is more than sufficient to guarantee
narrow convergence of $\rho_\disc^0$ to $\rho^0$,
but the uniform bound~\eqref{eq:maxaent} might require a well-adapted approximation,
especially for non-smooth $G^0$'s.

\section{Limit trajectory}
\label{sec:LimitTraj}
In this section, we assume that a sequence of vanishing
discretizations $\disc\to0$ is given, and we study the respective
limit of the fully discrete solutions $(G_\disc^n)_{n\ge0}$ that
are produced by the inductive minimization procedure
\eqref{eq:mini}. For the analysis of that limit trajectory, it is
more natural to work with the induced densities and velocities,
\begin{align*}
  \rho_\disc^n:=(G_\disc^n)_\#\refrho,
  \quad
  \velo_\disc^n:=\frac{\id-G_\disc^{n-1}\circ(G_\disc^n)^{-1}}\tau,
\end{align*}
instead of the Lagrangian maps $G_\disc^n$ themselves.
Note that $\velo_\disc^n$ is only well-defined on the support of $\rho_\disc^n$
--- that is, on the image of $G_\disc^n$ ---
and can be assigned arbitrary values outside.
Let us introduce the piecewise constant in time interpolations
$\intrho_\disc \colon [0,T]\times{\R^d}\to\Rnn$,
and $\intvelo_\disc \colon [0,T]\times{\R^d}\to\R^d$
as usual,
\begin{align*}
  \intrho_\disc(t) = \rho_\disc^n, \quad \intvelo_\disc(t) = \velo_\disc^n
  \quad \text{with $n$ such that $t\in((n-1)\tau,n\tau]$}.
\end{align*}
Note that $\intrho(t,\cdot)\in\dens$ and $\intvelo_\disc(t,\cdot)\in L^2({\R^d}\to\R^d;\intrho_\disc(t,\cdot))$
at each $t\ge0$.

\subsection{Energy estimates}
We start by proving the classical energy estimates on minimizing movements for our fully discrete scheme.
\begin{lemma}
  \label{lem:ee}
  For each discretization $\disc$ and
  for any indices $\overline n>\underline n\ge 0$,
  one has the a priori estimate
  \begin{align}
    \label{eq:esum}
    \aent(\rho_\disc^{\overline n})
    +\frac\tau2\sum_{n=\underline n+1}^{\overline n}\left(\frac{\wass(\rho_\disc^n,\rho_\disc^{n-1})}{\tau}\right)^2
    \le \aent(\rho^{\underline n}).
  \end{align}
  Consequently:
  \begin{enumerate}
  \item $\ent$ is monotonically decreasing, i.e., $\aent(\intrho_\disc(t))\le\aent(\intrho_\disc(s))$ for all $t\ge s\ge0$;
  \item $\intrho_\disc$ is H\"older-$\frac12$-continuous in $\wass$, up to an error $\tau$,
    \begin{align}
      \label{eq:wholder}
      \wass\big(\intrho_\disc(t),\intrho_\disc(s)\big) \le \sqrt{2\aent(\rho_\disc^0)}\big(|t-s|^{\frac12}+\tau^{\frac12}\big)
      \quad \text{for all $t\ge s\ge0$}.
    \end{align}
  \item $\intvelo_\disc$ is square integrable with respect to $\intrho_\disc$,
    \begin{align}
      \label{eq:veloest}
      \int_0^T\int_{\R^d}\|\intvelo_\disc\|^2\intrho_\disc\dd x\dd t \le 2\aent(\rho_\disc^0).
    \end{align}
  \end{enumerate}
\end{lemma}
\begin{proof}
  By the definition of $G_\disc^n$ as a minimizer,
  we know that $\ent_\disc(G_\disc^n;G_\disc^{n-1})\le\ent_\disc(G;G_\disc^{n-1})$ for any $G\in\ansatz$,
  and in particular for the choice $G:=G_\disc^{n-1}$,
  which yields:
  \begin{align}
    \label{eq:basicee}
    \frac1{2\tau}\int_K\|G_\disc^n-G_\disc^{n-1}\|^2\refrho_\triang\dd\omega +\ent(G_\disc^n|\refrho_\triang)
    \le \ent(G_\disc^{n-1}|\refrho_\triang).
  \end{align}
  Summing these inequalies for $n=\underline n+1,\ldots,\overline n$,
  recalling that $\aent(\rho_\disc^n)=\ent(G_\disc^n|\refrho_\triang)$ by~\eqref{eq:ent}
  and that $\wass(\rho_\disc^n,\rho_\disc^{n-1})^2\le\int_K|G_\disc^n-G_\disc^{n-1}|^2\refrho\dd\omega$ by~\eqref{eq:W3},
  produces~\eqref{eq:esum}.

  Monotonicity of $\aent$ in time is obvious.

  To prove~\eqref{eq:wholder}, choose $\underline n\le\overline n$ such that
  $s\in((\underline n-1)\tau,\underline n\tau]$ and $t\in((\overline n-1)\tau,\overline n\tau]$.
  Notice that $\tau(\overline n-\underline n)\le t-s+\tau$.
  If $\underline n=\overline n$, the claim~\eqref{eq:wholder} is obviously true;
  let $\underline n<\overline n$ in the following.
  Combining the triangle inequality for the metric $\wass$,
  estimate~\eqref{eq:esum} above
  and H\"older's inequality for sums,
  we arrive at
  \begin{align*}
    \wass\big(\intrho_\disc(t),\intrho_\disc(s)\big)
    = \wass(\rho_\disc^{\overline n},\rho_\disc^{\underline n})
    & \le \sum_{n=\underline n+1}^{\overline n}\wass(\rho_\disc^n,\rho_\disc^{n-1}) \\
    & \le \left[\sum_{n=\underline n+1}^{\overline n}\tau\right]^{\frac12}
    \left[\sum_{n=\underline n+1}^{\overline n}\frac{\wass(\rho_\disc^n,\rho_\disc^{n-1})^2}\tau\right]^{\frac12} \\
    & = \big[\tau(\overline n-\underline n) \big]^{\frac12}
    \left[\tau\sum_{n=\underline n+1}^{\overline n}\left(\frac{\wass(\rho_\disc^n,\rho_\disc^{n-1})}{\tau}\right)^2\right]^{\frac12} \\
    & \le [t-s+\tau]^{\frac12} \big[2\big(\aent(\rho_\disc^{\underline n}) - \aent(\rho_\disc^{\overline n})\big)\big]^{\frac12}
    \le \big[|t-s|^{\frac12}+\tau^{\frac12}\big]\aent(\rho_\disc^0)^{\frac12}.
  \end{align*}

  Finally, changing variables using $x=G_\disc^n(\omega)$ in~\eqref{eq:basicee} yields
  \begin{align*}
    \frac\tau2\int_{\R^d}\|\velo_\disc^n\|^2\rho_\disc^n\dd x + \ent(G_\disc^n) \le \ent(G_\disc^{n-1}),
  \end{align*}
  and summing these inequalities from $n=1$ to $n=N_\tau$ yields~\eqref{eq:veloest}.
\end{proof}

\subsection{Compactness of the trajectories and weak formulation}
Our main result on the weak limit of $\intrho_\disc$ is the following.
\begin{thm}
  \label{thm:trajectory}
  Along a suitable sequence $\disc\to0$,
  the curves $\intrho_\disc \colon \Rnn\to\dens$ convergence pointwise in time,
  i.e., $\intrho_\disc(t)\to\rho_*(t)$ narrowly for each $t>0$,
  towards a H\"older-$\frac12$-continuous limit trajectory $\rho_* \colon \Rnn\to\dens$.

  Moreover, the discrete velocities $\intvelo_\disc$
  possess a limit $\velo_*\in L^2(\Rnn\times{\R^d};\rho_*)$
  such that $\intvelo_\disc\intrho_\disc\weakstarto\velo_*\rho_*$ in $L^1(\Rnn\times{\R^d})$,
  and the continuity equation
  \begin{align}
    \label{eq:continuity}
    \partial_t\rho_* + \Div(\rho_*\velo_*) = 0
  \end{align}
  holds in the sense of distributions.
\end{thm}
\begin{rmk}
  The H\"older continuity of $\rho_*$ implies that $\rho_*$ satisfies the initial condition~\eqref{eq:NFPic}
  in the sense that $\rho_*(t)\to\rho^0$ narrowly as $t\downarrow0$.
\end{rmk}
\begin{proof}[Proof of Theorem~\ref{thm:trajectory}]
  We closely follow an argument that is part of the general convergence proof
  for the minimizing movement scheme as given in Ambrosio et~al.~\cite[Section 11.1.3]{book:AGS}.
  Below, convergence is shown for some arbitrary but fixed time horizon $T>0$;
  a standard diagonal argument implies convergence at arbitrary times.

  First observe that by estimate~\eqref{eq:wholder} --- applied with $0=s\le t\le T$ ---
  it follows that $\wass(\intrho_\disc(t),\rho_\disc^0)$ is bounded,
  uniformly in $t\in[0,T]$ and in $\disc$.
  Since further $\rho^0_\disc$ converges narrowly to $\rho^0$ by our hypotheses on the initial approximation,
  we conclude that all densities $\intrho_\disc(t)$ belong to a sequentially compact subset
  for the narrow convergence.
  The second observation is that the term on the right hand side of~\eqref{eq:wholder}
  simplifies to $(2\overline\aent)^\frac12|t-s|^\frac12$ in the limit $\disc\to0$.
  A straightforward application of the ``refined version'' of
  the Ascoli-Arzel\`{a} theorem (Proposition~3.3.1 in Ambrosio et~al.~\cite{book:AGS})
  yields the first part of the claim,
  namely the pointwise narrow convergence of $\intrho_\disc$
  towards a H\"older continuous limit curve $\rho_*$.

  It remains to pass to the limit with the velocity $\intvelo_\disc$.
  Towards that end, we define a probability measure $\intplan_\disc\in\prb(Z_T)$
  on the set $Z_T:=[0,T]\times\R^d\times\R^d$ as follows:
  \begin{align*}
    \int_{Z_T}\varphi(t,x,v)\dd\intplan_\disc(t,x,v)
    = \int_0^T\int_{\R^d} \varphi\big(t,x,\intvelo_\disc(t,x)\big)\,\intrho_\disc(t,x)\dd x\frac{\dn t}{T},
  \end{align*}
  for every bounded and continuous function $\varphi\in C^0_b(Z_T)$.
  For brevity, let $\intM_\disc\in\prb([0,T]\times\R^d)$ be the $(t,x)$-marginals of $\intplan_\disc$,
  that have respective Lebesgue densities $\frac{\rho_\disc(t,x)}T$ on $[0,T]\times\R^d$.
  Thanks to the result from the first part of the proof,
  $\intM_\disc$ converges narrowly to a limit $M_*$, which has density $\frac{\rho_*(t,x)}{T}$.
  On the other hand, the estimate~\eqref{eq:veloest} implies that
  \begin{align*}
    \int_{Z_T} |v|^2\dd\intplan_\disc(t,x,v)
    = \int_{[0,T]\times\R^d}|\intvelo_\disc(t,x)|^2\dd\intM_\disc(t,x)
    \le 2\overline\aent.
  \end{align*}
  We are thus in the position to apply Theorem~5.4.4 in Ambrosio et~al.~\cite{book:AGS},
  which yields the narrow convergence of $\intplan_\disc$ towards a limit $\plan_*$.
  Clearly, the $(t,x)$-marginal of $\plan_*$ is $M_*$.
  Accordingly, we introduce the disintegration $\plan_{(t,x)}$ of $\plan_*$ with respect to $M_*$,
  which is well-defined $M_*$-a.e..
  Below, it will turn out that $\plan_*$'s $v$-barycenter,
  \begin{align}
    \label{eq:velolimit}
    \velo_*(t,x) := \int_{\R^d} v \dd\gamma_{(t,x)}(v),
  \end{align}
  is the sought-for weak limit of $\intvelo_\disc$.
  The convergence $\intvelo_\disc\intrho_\disc\weakstarto\velo_*\rho_*$
  and the inheritance of the uniform $L^2$-bound~\eqref{eq:veloest} to the limit $\velo_*$
  are further direct consequences of Theorem~5.4.4 in Ambrosio et~al.~\cite{book:AGS}.

  The key step to establish the continuity equation for the just-defined $\velo_*$
  is to evaluate the limit as $\disc\to0$ of
  \begin{align*}
    J_\disc[\phi] := \frac1\tau\left[
    \int_0^T\int_{\R^d}\phi(t,x)\intrho_\disc(t,x)\dd x\dd t
    - \int_0^T\int_{\R^d}\phi(t,x)\intrho_\disc(t-\tau,x)\dd x\dd t
    \right]
  \end{align*}
  for any given test function $\phi\in C^\infty_c((0,T)\times\R^d)$
  in two different ways.
  First, we change variables $t\mapsto t+\tau$ in the second integral, which gives
  \begin{align*}
    J_\disc[\phi]
    = \int_0^T\int_{\R^d} \frac{\phi(t,x)-\phi(t+\tau,x)}\tau\intrho_\disc(t,x)\dd x\dd t
    \stackrel{\disc\to0}{\longrightarrow}
      -\int_0^T\int_{\R^d} \partial_t\phi(t,x)\,\rho_*(t,x)\dd x\dd t.
  \end{align*}
  For the second evaluation, we write
  \begin{align*}
    \rho_\disc^{n-1}
    = \big(G_\disc^{n-1}\circ(G_\disc^n)^{-1}\big)_\#\rho_\disc^n
    = \big(\id-\tau\velo_\disc^n\big)_\#\rho_\disc^n,
  \end{align*}
  and substitute accordingly $x\mapsto x-\tau\intvelo_\disc(t,x)$ in the second integral,
  leading to
  \begin{align*}
    J_\disc[\phi]
    &= \int_0^T\int_{\R^d} \frac{\phi(t,x)-\phi\big(t,x-\tau\intvelo_\disc(t,x)\big)}\tau\intrho_\disc(t,x)\dd x\dd t \\
    &= \int_0^T\int_{\R^d} \nabla\phi(t,x)\cdot\intvelo_\disc(t,x)\intrho_\disc(t,x)\dd x\dd t + \mathfrak{e}_\disc[\phi]\\
    &= \int_{Z_T} \nabla\phi(t,x)\cdot v\dd\intplan_\disc(t,x,v) + \mathfrak{e}_\disc[\phi] \\
    &\stackrel{\disc\to0}{\longrightarrow}
      \int_{Z_T} \nabla\phi(t,x)\cdot v\dd\plan_*(t,x,v) \\
    &\qquad= \int_{[0,T]\times\R^d}\nabla\phi(t,x)\cdot\left[\int_{\R^d}v\dd\plan_{(t,x)}(v)\right]\dd M_*(t,x) \\
    &\qquad = \int_0^T\int_{\R^d}\nabla\phi(t,x)\cdot\velo_*(t,x)\rho_*(t,x)\dd x\dd t.
  \end{align*}
  The error term $\mathfrak{e}_\disc[\phi]$ above is controlled
  via Taylor expansion of $\phi$ and by using~\eqref{eq:veloest},
  \begin{align*}
    \big|\mathfrak{e}_\disc[\phi]\big|
    \le \int_0^T\int_{\R^d}\frac\tau2\|\phi\|_{C^2}\big\|\intvelo_\disc(t,x)\big\|^2\intrho_\disc(t,x)\dd x\dd t
    \le \overline\aent\|\phi\|_{C^2}T\;\tau.
  \end{align*}
  Equality of the limits for both evaluations of $J_\disc[\phi]$ for arbitrary test functions $\phi$
  shows the continuity equation~\eqref{eq:continuity}.
\end{proof}

Unfortunately, the convergence provided by Theorem~\ref{thm:trajectory}
is generally not sufficient to conclude that $\rho_*$ is a weak solution to~\eqref{eq:NFPsystem},
since we are not able to identify $\velo_*$ as $\velo[\rho_*]$ from~\eqref{eq:velo}.
The problem is two-fold:
first, weak-$\star$ convergence of $\intrho_\disc$ is insufficient to pass to the limit inside the nonlinear function $P$.
Second, even if we would know that, for instance, $P(\intrho_\disc)\weakstarto P(\rho_*)$,
we would still need a $\disc$-independent a priori control
on the regularity (e.g., maximal diameter of triangles) of the meshes generated by the $G_\disc^n$
to justify the passage to limit in the weak formulation below.

The main difficulty in the weak formulation that we derive now is
that we can only use ``test functions'' that are piecewise affine
with respect to the changing meshes generated by the $G_\disc^n$.
For definiteness, we introduce the space
\begin{align*}
  \test(\triang):=\myset{\Gamma:K\to\R^d}{\text{$\Gamma$ is globally continuous, and is piecewise affine w.r.t. $\Delta_m$}}.
\end{align*}
\begin{lemma}
  Assume $S \colon {\R^d}\to\R^d$ is such that $S\circ G_\disc^n\in\test(\triang)$.
  Then:
  \begin{equation}
    \label{eq:WeakFormS}
    \int_{{\R^d}} P(\rho_\disc^{n}) \, \Div S \dd x - \int_{{\R^d}} \Grad V \cdot S\, \rho_\disc^{n} \dd x
    = \int_{{\R^d}} S\cdot\velo_\disc^n \rho_\disc^{n} \dd x.
  \end{equation}
\end{lemma}
\begin{proof}
  For all sufficiently small $\eps>0$, let $G_\eps = (\id+S)\circ G_\disc^n$.
  By definition of $G_\disc^n$ as a minimizer,
  we have that $\ent_\disc(G_\eps;G_\disc^{n-1})\ge\ent_\disc(G_\disc^n;G_\disc^{n-1})$.
  This implies that
  \begin{align}
    \label{eq:weakintegral}
    \begin{split}
      0&\le \frac1{\eps}\int_K\bigg(\frac1{2\tau}\big[\|G_\eps-G_\disc^{n-1}\|^2-\|G_\disc^n-G_\disc^{n-1}\|^2\big] \\
      & \qquad + \left[\widetilde h\left(\frac{\det\df G_\eps}{\refrho_\triang}\right)
        -\widetilde h\left(\frac{\det\df G_\disc^n}{\refrho_\triang}\right)\right]
      + \big[V\circ G_\eps-V\big]\bigg) \refrho_\triang\dd\omega.
    \end{split}
  \end{align}
  We discuss limits of the three terms under the integral for $\eps\searrow0$.
  For the metric term:
  \begin{align*}
    \frac1{2\tau\eps}\big[\|G_\eps-G_\disc^{n-1}\|^2-\|G_\disc^n-G_\disc^{n-1}\|^2\big]
    &= \frac{G_\disc^n-G_\disc^{n-1}}\tau\cdot\frac{G_\eps-G_\disc^n}\eps + \frac1{2\tau\eps}\|G_\eps-G_\disc^n\|^2 \\
    &= \left[\left(\frac{\id-T_\disc^n}\tau\right)\cdot S\right]\circ G_\disc^n + \frac{\eps}{2\tau}\|S\|^2\circ G_\disc^n,
  \end{align*}
  and since $S$ is bounded, the last term vanishes uniformly on $K$ for $\eps\searrow0$.
  For the internal energy, since $\df G_\eps=\df(\id+\eps S)\circ G_\disc^n\cdot\df G_\disc^n$,
  and recalling~\eqref{eq:h2P},
  \begin{align*}
    \frac1\eps\left[\widetilde h\left(\frac{\det\df G_\eps}{\refrho_\triang}\right)
      -\widetilde h\left(\frac{\det\df G_\disc^n}{\refrho_\triang}\right)\right]
    & = \frac1\eps\left[\widetilde h\left(\frac{\det\df G_\disc^n}{\refrho_\triang}\det(\idm+\eps\df S)\circ G_\disc^n\right)
      -\widetilde h\left(\frac{\det\df G_\disc^n}{\refrho_\triang}\right)\right] \\
    & \stackrel{\eps\searrow0}{\longrightarrow} \frac{\det\df G_\disc^n}{\refrho_\triang}\widetilde h'\left(\frac{\det\df G_\disc^n}{\refrho_\triang}\right)
    \left(\lim_{\eps\searrow0}\frac{\det(\idm+\eps\df S)}{\eps}\right)\circ G_\disc^n \\
    & = -\frac{\det\df G_\disc^n}{\refrho_\triang} P\left(\frac{\refrho_\triang}{\det\df G_\disc^n}\right)\tr[\df S]\circ G_\disc^n \\
    &= - \frac{\det\df G_\disc^n}{\refrho_\triang}\big[P(\rho^n)\,\Div S\big]\circ G_\disc^n.
  \end{align*}
  Since the piecewise constant function $\det\df G_\disc^n$ has a positive lower bound,
  the convergence as $\eps\searrow0$ is uniform on $K$.
  Finally, for the potential energy,
  \begin{align*}
    \frac1\eps\big[V\circ(\id+\eps S)\circ G_\disc^n-V\circ G_\disc^n\big]
    \stackrel{\eps\searrow0}{\longrightarrow}
    \big[\Grad V\cdot S\big]\circ G_\disc^n.
  \end{align*}
  Again, the convergence is uniform on $K$.
  Passing to the limit in the integral~\eqref{eq:weakintegral} yields
  \begin{align*}
    0 &\le \int_K \left[\left(\frac{\id-T_\disc^n}\tau\right)\cdot S\right]\circ G_\disc^n\refrho_\triang\dd\omega \\
    & \qquad - \int_K \big[P(\rho^n)\,\Div S\big]\circ G_\disc^n \det\df G_\disc^n\dd\omega
    + \int_K \big[\Grad V\cdot S\big]\circ G_\disc^n\refrho_\triang\dd\omega.
  \end{align*}
  The same inequality is true with $-S$ in place of $S$,
  hence this inequality is actually an equality.
  Since $\rho_\disc^n=(G_\disc^n)_\#\refrho_\triang$,
  a change of variables $x=S_\disc^n(\omega)$ produces~\eqref{eq:WeakFormS}.
\end{proof}
\begin{cor}
  \label{cor:limiteq}
  In addition to the hypotheses of Theorem~\ref{thm:trajectory},
  assume that
  \begin{enumerate}
  \item $P(\intrho_\disc)\weakstarto p_*$ in $L^1([0,T]\times\Omega)$;
  \item each $G_\disc^n$ is injective;
  \item as $\disc\to0$, all simplices in the images of $\triang$ under $G_\disc^n$
    have non-degenerate interior angles and tend to zero in diameter,
    uniformly w.r.t.\ $n$.
  \end{enumerate}
  Then $\rho_*$ satisfies the PDE
  \begin{align}
    \label{eq:limiteq}
    \partial_t\rho_* = \Delta p_* + \Div(\rho_*\Grad V)
  \end{align}
  in the sense of distributions.
\end{cor}
\begin{proof}
  Let a smooth test function $\zeta\in C^\infty_c(\R^d\to\R^d)$ be given.
  For each $\disc$ and each $n$,
  a $\zeta_\disc^n \colon \R^d\to\R^d$ with $\zeta_\disc^n\circ G_\disc^n\in\test(\triang)$
  can be constructed in such a way that
  \begin{align}
    \label{eq:togoodtobetrue}
    \zeta_\disc^n\to\zeta, \quad \Div\zeta_\disc^n\to\Div\zeta
  \end{align}
  uniformly on $\R^d$, and uniformly in $n$ as $\disc\to0$.
  This follows from our hypotheses on the $\disc$-uniform regularity of the Lagrangian meshes:
  inside the image of $G_\disc^n$, one can simply choose $\zeta_\disc^n$ as the affine interpolation
  of the values of $\zeta$ at the points $G_\disc^n(\omega_\ell)$.
  Outside, one can take an arbitrary approximation of $\zeta$ that is compatible
  with the piecewise-affine approximation on the boundary of $G_\disc^n$'s image;
  one may even choose $\zeta_\disc^n\equiv\zeta$ at sufficient distance to that boundary.
  The uniform convergences~\eqref{eq:togoodtobetrue} then follow
  by standard finite element analysis.

  Further, let $\eta\in C^\infty_c(0,T)$ be given.
  For each $t\in((n-1)\tau,n\tau]$,
  substitute $S(t,x):=\eta(t)\zeta_\disc^n(x)$ into~\eqref{eq:WeakFormS}.
  Integration of these equalities with respect to $t\in(0,T)$ yields
  \begin{align*}
    \int_0^T\int_{\R^d} P(\intrho_\disc)\Div S\dd x\dd t
    - \int_0^T\int_{\R^d} \Grad V\cdot S\dd x\dd t
    = \int_0^T\int_{\R^d} S\cdot\intvelo_\disc\intrho_\disc\dd x\dd t.
  \end{align*}
  We pass to the limit $\disc\to0$ in these integrals.
  For the first, we use that $P(\intrho)\weakstarto p_*$ by hypothesis,
  for the last, we use Theorem~\ref{thm:trajectory} above.
  Since any test function $S\in C^\infty_c((0,T)\times\Omega)$ can be approximated in $C^1$
  by linear combinations of products $\eta(t)\zeta(x)$ as above,
  we thus obtain the weak formulation of
  \begin{align*}
    \rho_*\velo_* = \Grad p_* + \rho_*\Grad V.
  \end{align*}
  In combination with the continuity equation~\eqref{eq:continuity},
  we arrive at~\eqref{eq:limiteq}.
\end{proof}
\begin{rmk}
  In principle, our discretization can also be applied to
  the \emph{linear} Fokker-Planck equation with $P(r)=r$ and $h(r)=r\log r$.
  In that case, one automatically has $P(\intrho)\weakstarto p_*\equiv P(\rho_*)$
  thanks to Theorem~\ref{thm:trajectory}.
  Corollary~\ref{cor:limiteq} above then provides an \emph{a posteriori} criterion for convergence:
  if the Lagrangian mesh does not deform too wildly under the dynamics as the discretization is refined,
  then the discrete solutions converge to the genuine solution.
\end{rmk}

\section{Consistency in 2D}
\label{sec:Consistency2D}
In this section, we prove consistency of our discretization in the following sense.
Under certain conditions on the spatial discretization $\triang$,
any smooth and positive solution $\rho$ to the initial value problem~\eqref{eq:NFPsystem}
projects to a discrete solution that satisfies the Euler-Lagrange equations up to a controlled error.
We restrict ourselves to $d=2$ dimensions.

\subsection{Smooth Lagrangian evolution}
First, we derive an alternative form of the velocity field $\velo$ from~\eqref{eq:velo} in terms of $G$.
\begin{lemma}
  For $\rho=G_\#\refrho$ with a smooth diffemorphism $G \colon K\to{\R^d}$,
  we have
  \begin{align}
    \velo[\rho]\circ G=\bigvelo[G]
    := P'\left(\frac{\refrho}{\det\df G}\right)\, (\df G)^{-T}\left(\tr_{12}\big[(\df G)^{-1}\df ^2G\big]^T-\frac{\nabla\refrho}{\refrho}\right) - \Grad V \circ G.
  \end{align}
  Consequently, the Lagrangian map $G$ --- relative to the reference density $\refrho$ ---
  for a smooth solution $\rho$ to~\eqref{eq:NFPsystem} satisfies
  \begin{align}
    \label{eq:Geq}
    \partial_t G = \bigvelo[G].
  \end{align}
\end{lemma}
\begin{proof}
  On the one hand,
  \begin{align*}
    \df \big[h'(\rho)\circ G\big] = \big[\df h'(\rho)\big]\circ G\,\df G,
  \end{align*}
  and on the other hand, by definition of the push forward,
  \begin{align*}
    \df \big[h'(\rho)\circ G\big]
    &= \df h'\left(\frac{\refrho}{\det \df G}\right)  \\
    &= h''\left(\frac{\refrho}{\det \df G}\right)\,\left(\frac{\refrho}{\det \df G}\right)\,\left(\frac{\df \refrho}{\refrho}-\tr_{12}\big[(\df G)^{-1}\df ^2G \big]\right) \\
    &= \big[\rho h''(\rho)\big]\circ G \,\left(\frac{\df \refrho}{\refrho}-\tr_{12}\big[(\df G)^{-1}\df ^2G \big]\right).
  \end{align*}
  Hence
  \begin{align*}
    \Grad h'(\rho) \circ G
    = \big[\rho h''(\rho)\big]\circ G \,(\df G)^{-T}\left(\frac{\Grad\refrho}{\refrho}-\tr_{12}\big[(\df G)^{-1}\df ^2G \big]^T\right).
  \end{align*}
  Observing that~\eqref{eq:Ph} implies that $rh''(r)=P'(r)$, we conclude~\eqref{eq:Geq} directly from~\eqref{eq:velo}.
\end{proof}

\renewcommand{\zs}{\times}

\subsection{Discrete Euler-Lagrange equations in dimension $d=2$}
In the planar case $d=2$,
the Euler-Lagrange equation~\eqref{eq:EL} above can be rewritten in a more convenient way.

In the following, fix some vertex $\omega_\zs$ of the triangulation, which
is indicent to precisely six triangles.
For convenience, we assume that these are labelled $\Delta_0$ to $\Delta_5$ in counter-clockwise order.
Similarly, the six neighboring vertices are labeled $\omega_0$ to $\omega_5$ in counter-clockwise order,
so that $\Delta_k$ has vertices $\omega_{k}$ and $\omega_{k+1}$,
where we set $\omega_6:=\omega_0$.

Using these conventions and recalling Lemma~\ref{lem:JAJ},
the expression for the vector $\nml$ in~\eqref{eq:Pi} simplifies to
\begin{align*}
  \nml_\triang^k = - \jm (G_{k+1}-G_{k}),
  \quad\text{where}\quad
  \jm = \begin{pmatrix} 0 & -1 \\ 1 & 0 \end{pmatrix}.
\end{align*}
Summing the Euler-Lagrange equation~\eqref{eq:EL} over $\Delta_0$ to $\Delta_5$,
we obtain
\begin{align}
  \label{eq:dEL}
  \momentum_\zs = \impulse_\zs,
\end{align}
where the momentum term $\momentum_\zs$ and the impulse $\impulse_\zs$, respectively,
are given by
\begin{align}
  \momentum_\zs &= \frac1{12}\sum_{k=0}^5\mu_\triang^k
  \left[2\left(\frac{G_\zs-G^*_\zs}\tau\right)+\left(\frac{G_k-G^*_k}\tau\right)+\left(\frac{G_{k+1}-G^*_{k+1}}\tau\right)\right]  \\
  \impulse_\zs &= \sum_{k=0}^5 \mu_\triang^{k}\bigg[
  \frac1{2 \mu_\triang^{k}}P\left(\frac {2\mu_\triang^k}{\det(G_k-G_\zs|G_{k+1}-G_\zs)}\right)\jm(G_{k+1}-G_k) \\
  &\qquad -\aint_{\stdtri}\Grad V\big((1-\xi_1-\xi_2)G_\zs + \xi_1 G_{k}+ \xi_2 G_{k+1}\big)\,(1-\xi_1-\xi_2)\dd\xi
  \bigg].
\end{align}
We shall now prove our main result on consistency.
The setup is the following:
a sequence of triangulations $\triang_\eps$ on $K$, parametrized by $\eps>0$,
and a sequence of time steps $\tau_\eps=\bigO(\eps)$ are given.
We assume that there is an $\eps$-independent region $K'\subset K$
on which the $\triang_\eps$ are \emph{almost hexagonal} in the following sense:
each node $\omega_\zs\in K'$ of $\triang_\eps$ has precisely six neighbors
--- labelled $\omega_0$ to $\omega_5$ in counter-clockwise order ---
and there exists a rotation $R\in\mathrm{SO}(2)$ such that
\begin{align}
  \label{eq:xik}
  R(\omega_k-\omega_\zs) = \eps\sigma_k + \bigO(\eps^2)
  \quad\text{with}\quad
  \sigma_k = \begin{pmatrix} \cos\frac\pi3k \\ \sin \frac\pi3k \end{pmatrix}
\end{align}
for $k=0,1,\ldots,5$.

Now, let $G \colon [0,T]\times K\to{\R^d}$ be a given smooth solution
to the Lagrangian evolution equation~\eqref{eq:Geq},
and fix a time $t\in(0,T)$.
For all sufficiently small $\eps>0$, we define maps $G_\eps,G_\eps^*\in\ansatzeps$
by linear interpolation of the values of $G(t;\cdot)$ and $G(t-\tau;\cdot)$, respectively,
on $\triang_\eps$.
That is, $G_\eps(\omega_\ell)=G(t;\omega_\ell)$ and $G^*_\eps(\omega_\ell)=G(t-\tau;\omega_\ell)$,
at all nodes $\omega_\ell$ in $\triang_\eps$.
Theorem~\ref{thm:consist} below states that the pair $G_\eps,G_\eps^*$
is an approximate solution to the discrete Euler-Lagrange equations~\eqref{eq:dEL}
at all nodes $\omega_\zs$ of the respective triangulation $\triang_\eps$ that lie in $K'$.

The hexagonality hypothesis on the $\triang_\eps$ is strong,
but some very strong restriction of $\ansatzeps$'s geometry is apparently necessary.
See Remark~\ref{rmk:weird} following the proof for further discussion.
\begin{thm}
  \label{thm:consist}
  Under the hypotheses and with the notations introduced above,
  the Euler-Lagrange equation~\eqref{eq:dEL} admits the following asymptotic expansion:
  \begin{subequations}
  \begin{align}
    \label{eq:momentum}
    \momentum_\zs &= \frac{\sqrt3}2\eps^2\,\refrho(\omega_\zs)\partial_tG(t;\omega_\zs)+\bigO(\eps^3), \\
    \label{eq:impulse}
    \impulse_\zs &= \frac{\sqrt3}2\eps^2\,\refrho(\omega_\zs)\bigvelo[G](t;\omega_\zs)+\bigO(\eps^3),
  \end{align}
  \end{subequations}
  as $\eps\to0$,
  uniformly at the nodes $\omega_\zs\in K'$ of the respective $\triang_\eps$.
\end{thm}
\begin{rmk}
  Up to an error $\bigO(\eps^3)$,
  the geometric pre-factor $\frac{\sqrt3}2\eps^2$ equals to
  one third of the total area of the hexagon with vertices $\omega_0$ to $\omega_5$,
  and is thus equal to the integral of the piecewise affine hat function with peak at $\omega_\zs$.
\end{rmk}
\begin{proof}[Proof of Theorem~\ref{thm:consist}]
  Throughout the proof, let $\eps>0$ be fixed;
  we shall omit the $\eps$-index for $\triang_\eps$ and $\tau_\eps$.
  First, we fix a node $\omega_\zs$ of $\triang\cap K'$.
  Thanks to the equivariance of both~\eqref{eq:Geq} and~\eqref{eq:dEL} under rigid motions of the domain,
  we may assume that $R$ in~\eqref{eq:xik} is the identity, and that $\omega_\zs=0$.

  We collect some relations that are helpful for the calculations that follow.
  Trivially,
  \begin{align}
    \label{eq:sigma0}
    \sum_{k=0}^5\sigma_k=0, \quad \sum_{k=0}^5\omega_k=\bigO(\eps^2).
  \end{align}
  Moreover, we have that
  \begin{align}
    \label{eq:sigma1}
    |\Delta_k| = \det(\omega_k|\omega_{k+1})
    = \eps^2\det(\sigma_k|\sigma_{k+1}) + \bigO(\eps^3)
    = \frac{\sqrt3}4\eps^2+\bigO(\eps^3).
  \end{align}
  On the other hand, by definition of $\mu_\triang^k$ in~\eqref{eq:mu},
  it follows that
  \begin{align}
    \label{eq:sigma2}
    \begin{split}
      \mu_\triang^k = |\Delta_k|\aint_{\Delta_k}\refrho\dd\omega
      &= \frac12\det(\omega_{k}|\omega_{k+1})
      \,\left[\refrho\left(\frac{\omega_k+\omega_{k+1}}3\right)+\bigO(\eps)\right] \\
      &= \frac12\det(\omega_{k}|\omega_{k+1})
      \left[\refrho_\zs+\eps\Grad\refrho_\zs\cdot \frac{\sigma_k+\sigma_{k+1}}3+\bigO(\eps^2)\right].
    \end{split}
  \end{align}
  Combining~\eqref{eq:sigma1} and~\eqref{eq:sigma2} yields
  \begin{align}
    \label{eq:sigma3}
    \mu_\triang^k = \eps^2\left(\frac{\sqrt{3}}4\refrho_\zs+\bigO(\eps)\right).
  \end{align}
  In accordance with the definition of $G_\eps$ and $G_\eps^*$ from $G$ detailed above,
  let $G_\zs:=G(t,\omega_\zs)$ and $G^*_\zs=G(t-\tau,\omega_\zs)$,
  and define $G_k$, $G_k^*$ for $k=0,\ldots,5$ in the analogous way.
  Further, we introduce $\df G_\zs=\df G(t,\omega_\zs)$, $\df^2G_\zs=\df^2G(t,\omega_\zs)$, $\partial_tG_\zs=\partial_tG(t,\omega_\zs)$.

  To perform an expansion in the \emph{momentum term}, first observe that
  \begin{align*}
    G(t-\tau;\omega_k) = G(t;\omega_k) - \tau\partial_t G(t;\omega_k) + \bigO(\tau^2),
  \end{align*}
  for each $k=0,1,\ldots,5$,
  and so, using that $\tau=\bigO(\eps)$ by hypothesis,
  \begin{align*}
    \frac{G_k-G_k^*}\tau = \partial_tG(t;\omega_k) + \bigO(\tau)
    = \partial_tG_\zs + \bigO(\eps) + \bigO(\tau) = \partial_tG_\zs + \bigO(\eps).
  \end{align*}
  Using~\eqref{eq:sigma3} and then~\eqref{eq:sigma0} yields
  \begin{align*}
    \momentum_\zs
    &=\frac1{12\tau}\sum_{k=0}^5 \eps^2\left(\frac{\sqrt{3}}4\refrho_\zs+\bigO(\eps)\right)\big[4\partial_tG_\zs+\bigO(\eps)\big] \\
    &=\frac{\sqrt3}2\eps^2\,\refrho_\zs\partial_tG_\zs + \bigO(\eps^3).
  \end{align*}
  This is~\eqref{eq:momentum}.

  For the \emph{impulse term}, we start with a Taylor expansion to second order in space:
  \begin{align*}
    G_k = G_\zs + \df G_\zs\omega_k + \frac12 \df ^2G_\zs:[\omega_k]^2 + \bigO(\eps^3).
  \end{align*}
  We combine this with the observation that $(\omega_k|\omega_{k+1})^{-1}=\bigO(\eps^{-1})$
  to obtain:
  \begin{align*}
    &\frac{\mu_\triang^k}{\det(G_k-G_\zs|G_{k+1}-G_\zs)} \\
    &= \frac{\det(\omega_k|\omega_{k+1})}{\det \df G_\zs}
    \frac{\refrho_\zs+\eps\Grad\refrho_\zs\cdot \frac{\sigma_k+\sigma_{k+1}}3+\bigO(\eps^2)}
    {\det\big[(\omega_k|\omega_{k+1})+\frac12(\df G_\zs)^{-1}\big(\df ^2G_\zs:[\omega_k]^2\big|\df ^2G_\zs:[\omega_{k+1}]^2\big)+\bigO(\eps^3)\big]} \\
    &= \frac{\refrho_\zs}{\det\df G_\zs}
    \frac{1+\displaystyle{\eps\frac{\Grad\refrho_\zs}{\refrho_\zs}\cdot \frac{\sigma_k+\sigma_{k+1}}3}+\bigO(\eps^2)}
    {\det\big[\idm+\frac12(\df G_\zs)^{-1}\big(\df ^2G_\zs:[\omega_k]^2\big|\df ^2G_\zs:[\omega_{k-1}]^2\big)
      \,(\omega_k|\omega_{k+1})^{-1}+\bigO(\eps^2) \big]} \\
    &=\frac{\refrho_\zs}{\det\df G_\zs}\left(1+\eps\left\{\chi_k-\frac12\vartheta_k\right\}+\bigO(\eps^2)\right),
  \end{align*}
  where
  \begin{align*}
    \chi_k &= \frac{\Grad\refrho_\zs}{\refrho_\zs}\cdot \frac{\sigma_k+\sigma_{k+1}}3, \\
    \vartheta_k &= \tr\left[\big((\df G_\zs)^{-1}\df ^2G_\zs:[\sigma_k]^2\big|(\df G_\zs)^{-1}\df ^2G_\zs:[\sigma_{k+1}]^2\big)\,(\sigma_k|\sigma_{k+1})^{-1}\right].
  \end{align*}
  Plugging this in leads to
  \begin{align*}
    &\sum_{k=0}^5\left\{\frac12 P\left(\frac{\refrho_\zs}{\det \df G_\zs}\right)
      + \frac\eps2P'\left(\frac{\refrho_\zs}{\det \df G_\zs}\right)\left\{\chi_k-\frac12\vartheta_k\right\}
      + \bigO(\eps^2) \right\}
    \jm \df G_\zs(\omega_{k+1}-\omega_{k}) \\
    &=\frac12 P\left(\frac{\refrho_0}{\det \df G_\zs}\right) \jm \df G_\zs\left(\sum_{k=0}^5 (\omega_{k+1}-\omega_{k})\right) \\
    &\qquad + \frac{\eps^2}4P'\left(\frac{\refrho_\zs}{\det \df G_\zs}\right)
    \jm \df G_\zs\jm^T\left(\sum_{k=0}^5 \left\{2\chi_k-\vartheta_k\right\} \jm(\sigma_{k+1}-\sigma_{k})\right) + \bigO(\eps^3) \\
    & = 0 + \frac{\sqrt3}2\eps^2P'\left(\frac{\refrho_\zs}{\det \df G_\zs}\right)\,
    (\df G_\zs)^{-T}\left\{\tr_{12}\big[(\df G_\zs)^{-1}\df ^2G_\zs\big]^T-\frac{\Grad\rho_\zs}{\rho_\zs}\right\}
    + \bigO(\eps^3),
  \end{align*}
  where we have use the auxiliary algebraic results
  from Lemma~\ref{lem:JAJ}, Lemma~\ref{lem:circle}, and Lemma~\ref{lem:traces}.

  For the remaining part of the impulse term, a very rough approximation is sufficient:
  \begin{align*}
    \Grad V(g) = \Grad V(G_\zs) + \bigO(\eps)
  \end{align*}
  holds for any $g$ that is a convex combination of $G_\zs,G_0,\ldots,G_5$,
  where the implicit constant is controlled in terms of the supremum of $\df^2V$ and $\df G$ on $K'$.
  With that, we simply have, using again~\eqref{eq:sigma3}:
  \begin{align*}
    &\sum_{k=0}^5\mu_\triang^k\aint_{\stdtri}\Grad V\big((1-\xi_1-\xi_2)G_\zs + \xi_1 G_{k}+ \xi_2 G_{k+1}\big)\,(1-\xi_1-\xi_2)\dd\xi \\
    & = 6\eps^2\left(\frac{\sqrt{3}}4\refrho_\zs+\bigO(\eps)\right)\,\big(\Grad V(G_\zs) + \bigO(\eps)\big)
    = \frac{\sqrt3}2\eps^2\,\refrho_\zs\Grad V(G_\zs)+\bigO(\eps^3).
  \end{align*}
  Together, this yields~\eqref{eq:impulse}.
\end{proof}
\begin{rmk}
  \label{rmk:weird}
  The hypotheses of Theorem~\eqref{thm:consist} require that the $\triang_\eps$ are almost hexagonal on $K'$.
  This seems like a technical hypothesis that simplifies calculations,
  but apparently, \emph{some} strong symmetry property of the $\triang_\eps$ is necessary for the validity of the result.

  To illustrate the failure of consistency --- at least in the specific form considered here ---
  assume that $V\equiv0$ and $\refrho\equiv1$,
  and consider a sequence of triangulations $\triang_\eps$ for which there is a node $\omega_\zs$
  such that~\eqref{eq:xik} holds with the $\sigma_k$ being replaced
  by a different six-tuple of vectors $\sigma'_k$. 
  Repeating the steps of the proof above, it is easily seen that
  $\momentum_\zs=a\eps^2\,\pd_tG(t;\omega_\zs)+\bigO(\eps^3)$,
  with an $\eps$-independent constant $a>0$ in place of $\sqrt3/2$,
  and that
  \[ \impulse_\zs=-\frac{\eps^2}4 P'\left(\frac1{\det\df G_\zs}\right)\,(\df G_\zs)^{-T}
  \sum_{k=0}^5\vartheta'_k\jm(\sigma_{k+1}'-\sigma_k')+\bigO(\eps^3), \]
  with
  \begin{align*}
    \vartheta_k' = \tr\left[\big((\df G_\zs)^{-1}\df ^2G_\zs:[\sigma_k']^2\big|(\df G_\zs)^{-1}\df ^2G_\zs:[\sigma_{k+1}']^2\big)\,(\sigma_k'|\sigma_{k+1}')^{-1}\right].
  \end{align*}
  If a result of the form~\eqref{eq:impulse} --- with $\sqrt3/2$ replaced by $a$ ---
  was true, then this implies in particular that
  \begin{align}
    \label{eq:howinhellcanthatfail}
    \sum_{k=0}^5\vartheta'_k\jm(\sigma_{k+1}'-\sigma_k') = a'\tr_{12}\big[(\df G_\zs)^{-1}\df^2G_\zs\big]
  \end{align}
  holds with some constant $a'>0$ for
  arbitrary matrices $\df G_\zs\in\R^{2\times2}$ of positive determinant
  and tensors $\df^2 G_\zs\in\R^{2\times2\times2}$ that are symmetric in the second and third component.
  A specific example for which~\eqref{eq:howinhellcanthatfail} is not true is given by
  \begin{align}
    \label{eq:newsigma}
    \sigma_0' = {1\choose0} = -\sigma_3',\quad
    \sigma_1' = {\frac12\choose\frac12} = -\sigma_4',\quad
    \sigma_2' = {0\choose1} = -\sigma_5',
  \end{align}
  in combination with $\df G_\zs=\idm$,
  and a $\df^2G_\zs$ that is zero except for two ones, at the positions $(1,2,2)$ and $(2,1,1)$.
  In Lemma~\ref{lem:algebra2}, we show that the left-hand side in~\eqref{eq:howinhellcanthatfail} equals to $1\choose1$;
  on the other hand, the right-hand side is clearly zero.

  Note that this counter-example is significant,
  insofar as the skew (in fact, degenerate) hexagon described by the $\sigma_k'$ in~\eqref{eq:newsigma}
  corresponds to a popular method for triangulation of the plane.
\end{rmk}

\section{Numerical simulations in $d=2$}
\label{sec:NumSim2D}

\subsection{Implementation}
The Euler-Lagrange equations for the $d=2$-dimensional case have been derived in~\eqref{eq:dEL}.
We perfom a small modification in the potential term in order to simplify calculations
with presumably minimal loss in accuracy:
\begin{align*}
  \mathbf{Z}_\zs[G;G^*]
  &=
    \sum_{k=0}^5\frac{\mu_\triang^k}{12}
    \left[2\left(\frac{G_\zs-G^*_\zs}\tau\right)+\left(\frac{G_k-G^*_k}\tau\right)+\left(\frac{G_{k+1}-G^*_{k+1}}\tau\right)\right]  \\
  & +
    \sum_{k=0}^5 \bigg[
    \frac12\widetilde h'\left(\frac{\det(G_k-G_\zs|G_{k+1}-G_\zs)}{2\mu_\triang^k}\right)\jm(G_{k+1}-G_k)
    + \frac{\mu_\triang^{k}}6\Grad V(G_{k+\frac12})\bigg],
\end{align*}
with the short-hand notation
\begin{align*}
  G_{k+\frac12} = \frac13(G_\zs+G_k+G_{k+1}).
\end{align*}
On the main diagonal, the Hessian amounts to
\begin{align*}
  \mathbf{H}_{\zs\zs}[G]
  &=
    \left(\sum_{k=0}^5\frac{\mu_\triang^k}{6\tau}\right)\mathds{1}_2 \\
  &+
    \sum_{k=0}^5\frac1{4\mu_\triang^k}\widetilde h''\left(\frac{\det(G_k-G_\zs|G_{k+1}-G_\zs)}{2\mu_\triang^k}\right)
    \big[\jm(G_{k+1}-G_k)\big]\big[\jm(G_{k+1}-G_k)\big]^\trap \\
  &+
    \sum_{k=0}^5 \frac{\mu_\triang^{k}}{18}\Grad^2 V(G_{k+\frac12})
\end{align*}
Off the main diagonal, the entries of the Hessian are given by
\begin{align*}
  \mathbf{H}_{\zs k}[G]
  &=
    \frac{\mu_\triang^k+\mu_\triang^{k-1}}{12\tau}\mathds{1}_2 \\
  &+
    \frac1{4\mu_\triang^k}\widetilde h''\left(\frac{\det(G_k-G_\zs|G_{k+1}-G_\zs)}{2\mu_\triang^k}\right)
    \big[\jm(G_{k+1}-G_k)\big]\big[\jm(G_{k+1}-G_\zs)\big]^\trap \\
  &-
    \frac1{4\mu_\triang^{k-1}}\widetilde h''\left(\frac{\det(G_{k-1}-G_\zs|G_{k}-G_\zs)}{2\mu_\triang^{k-1}}\right)
    \big[\jm(G_{k}-G_{k-1})\big]\big[\jm(G_{k-1}-G_\zs)\big]^\trap \\
  &+
    \frac{\mu_\triang^{k}}{18}\Grad^2 V(G_{k+\frac12})
    + \frac{\mu_\triang^{k-1}}{18}\Grad^2 V(G_{k-\frac12}).
\end{align*}

The scheme consists of an inner (Newton)
and an outer (time stepping) iteration. We start from a given initial
density $\rho_0$ and define the solution at the next time step
inductively by applying Newton's method in the inner iteration.
To this end we initialise $G^{(0)}:=G^n$ with $G^n$, the solution at the $n$th time step,
and define inductively
\begin{align*}
  G^{(s+1)} := G^{(s)} + \delta G^{(s+1)} ,
\end{align*}
where the update $\delta G^{(s+1)}$ is the solution to the linear system
\begin{align*}
   \mathbf{H}[G^{(s)}] \delta G^{(s+1)}
  = -\mathbf{Z}[G^{(s)};G^n] .
\end{align*}
The effort of each inner iteration step is essentially determined
by the effort to invert the sparse matrix $\mathbf{H}[G^{(s)}]$.
As soon as the norm of $\delta G^{(s+1)}$ drops below a given stopping threshold,
define $G^{n+1}:=G^{(s+1)}$ as approximate solution in the $n+1$st time step.

In all experiments the stopping criterion in the Newton iteration is set to $10^{-9}$.

\subsection{Numerical experiments}
In this section we present results of our numerical experiments for~\eqref{eq:NFPsystem}
with a cubic porous-medium nonlinearity $P(r)=r^3$
and different choices for the external potential $V$,
\begin{align}
  \label{eq:PME3}
  \partial_t\rho = \Delta(u^3) + \Div(u\Grad V).
\end{align}

\subsubsection*{Numerical experiment 1: unconfined evolution of Barenblatt profile}
As a first example, we consider the ``free'' cubic porous medium equation,
that is~\eqref{eq:PME3} with $V\equiv0$.
It is well-known (see, e.g., Vazquez~\cite{book:Vazquez}) that in the long-time limit $t\to\infty$,
arbitrary solutions approach a self-similar one,
\begin{align}
  \label{eq:selfsim}
  \rho^*(t,x) = t^{-d\alpha}\BB_3\big(t^{-\alpha}x\big)
  \quad\text{with}\quad
  \alpha=\frac16, 
\end{align}
where $\BB_3$ is the associated Barenblatt profile
\begin{equation}
  \label{eq:BBprofile}
  \BB_3(z) = \left(C_3-\frac13\|z\|^2\right)_+^{\frac12},
\end{equation}
where $C_3=(2\pi)^{-\frac23}\approx 0.29$ is chosen to normalize $\BB_3$'s mass to unity.

In this experiment, we are only interested in the quality of the numerical approximation
for the self-similar solution~\eqref{eq:selfsim}.
To reduce numerical effort, we impose a four-fold symmetry of the  approximation:
we use the quarter circle as computational domain $K$,
and interprete the discrete function thereon as one of four symmetric pieces of the full discrete solution.
To preserve reflection symmetry over time,
homogeneous Neumann conditions are imposed on the artificial boundaries.
This is implemented by reducing the degrees of freedom of the nodes along the $x$- and $y$-axes
to tangential motion.
\begin{figure}[b]
  \includegraphics[width=0.45\textwidth]{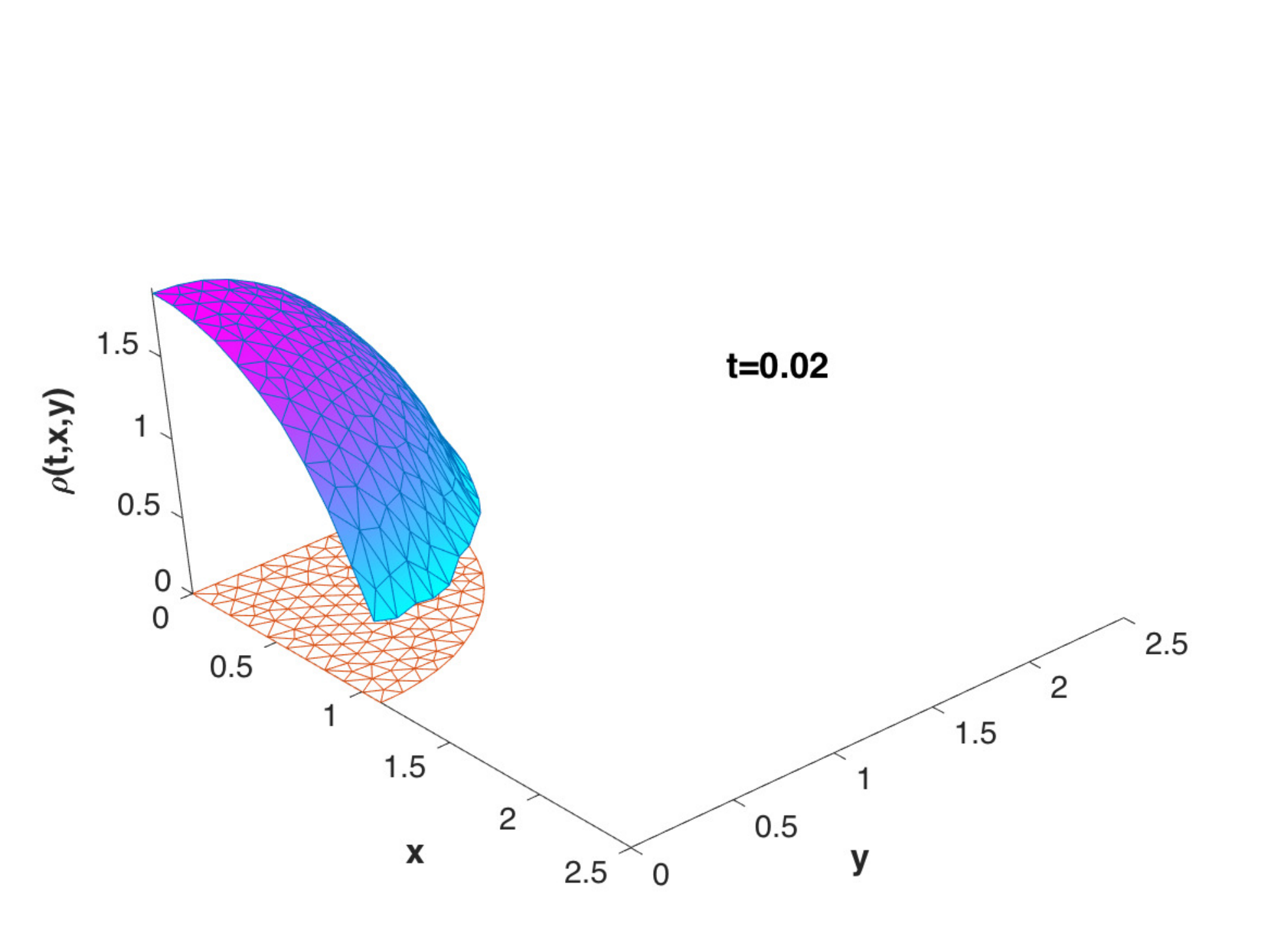}%
  \includegraphics[width=0.45\textwidth]{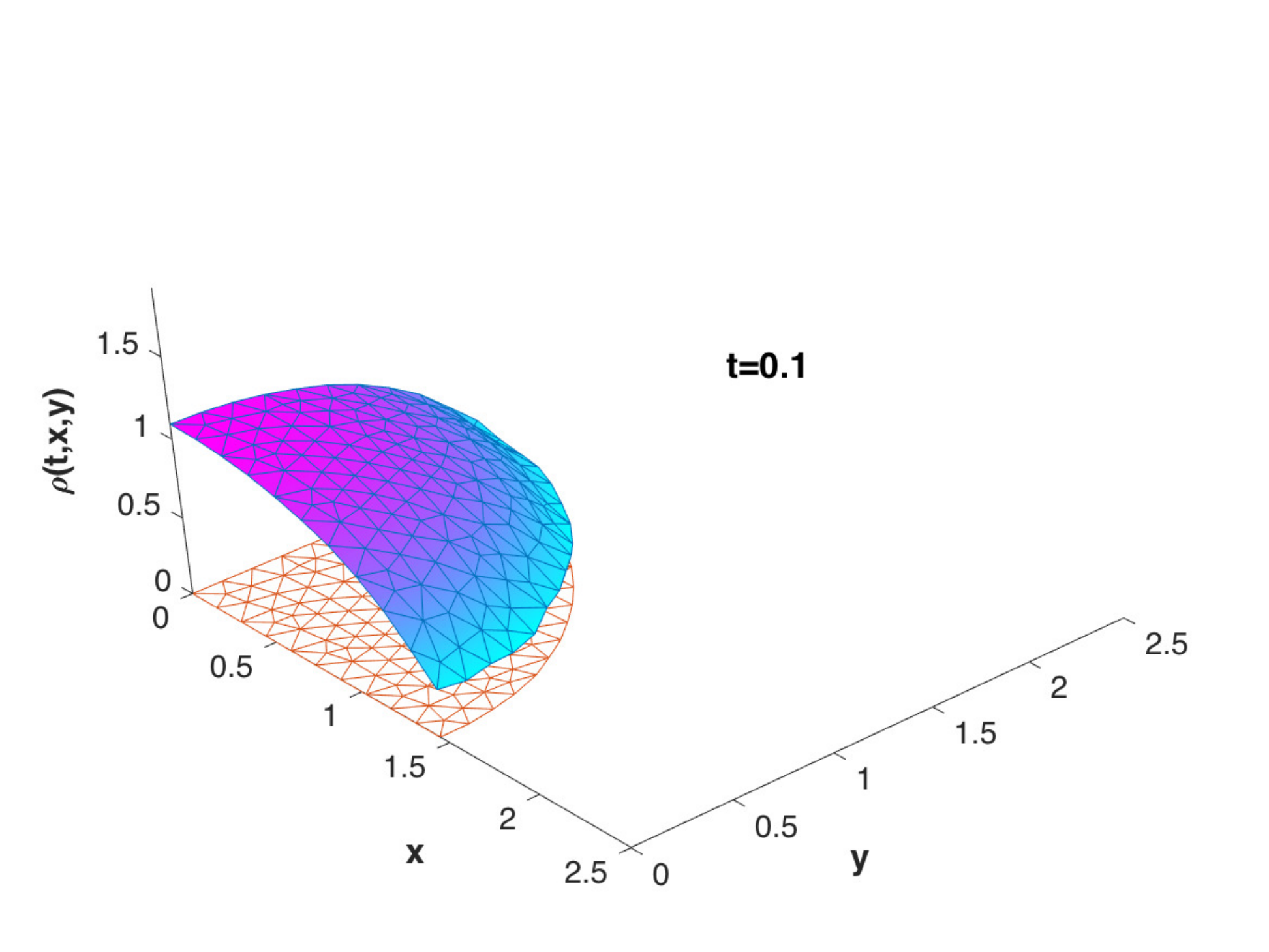}\\
  \includegraphics[width=0.45\textwidth]{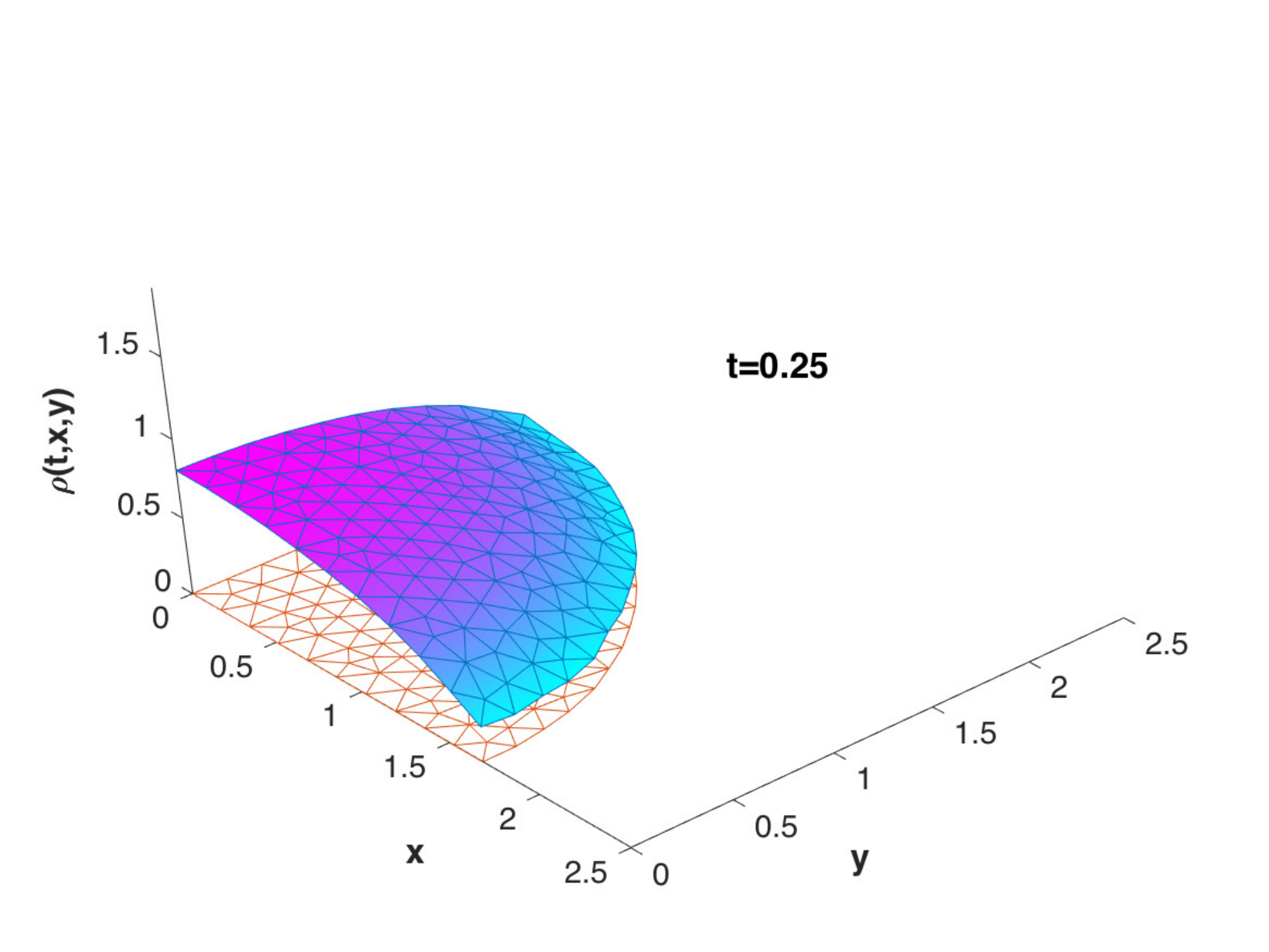}%
  \includegraphics[width=0.45\textwidth]{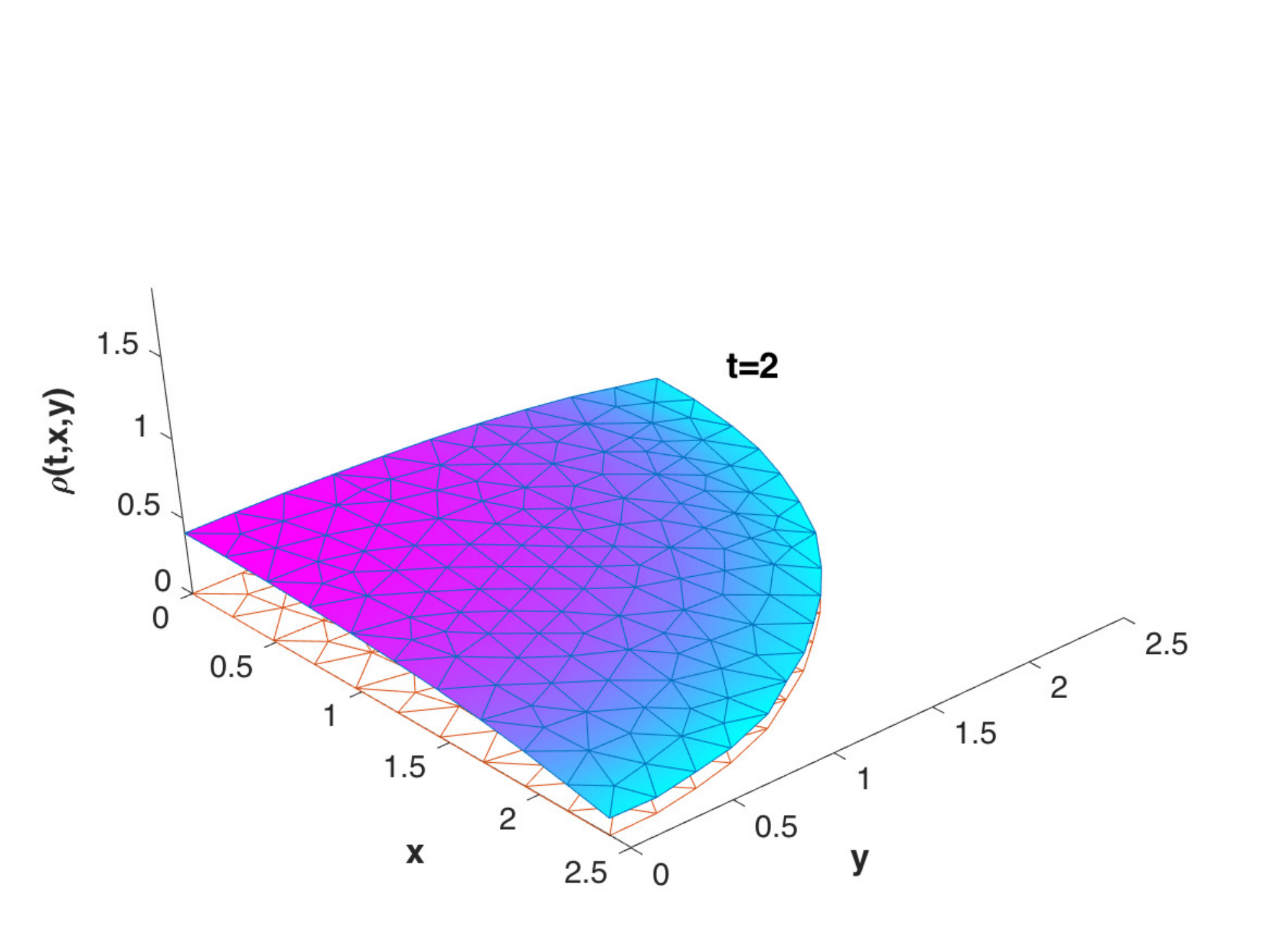}
  \caption{Numerical experiment 1: fully discrete evolution of our approximation for the self-similar solution
    to the free porous medium equation.
    Snapshots are taken at times $t=0.02$, $t=0.1$, $t=0.25$, and $t=2.0$.
  }
\label{fig:PMEevol}
\end{figure}
We initialize our simulation with a piecewise constant approximation
of the profile of $\rho^*$ from~\eqref{eq:BBprofile} at time $t=0.01$.
We choose a time step $\tau=0.001$ and the final time $T=2$.
In Figure~\ref{fig:PMEevol}, we have collected snapshots of the approximated density at different instances of time.
The Barenblatt profile of the solution is very well pertained over time.
\begin{rmk}
  It takes less than 2 minutes to complete this simulation on standard laptop
  ({\sc Matlab} code on a mid-2013 MacBook Air 11'' with 1.7 GHz Intel Core i7 processor).
\end{rmk}
\begin{figure}
  \includegraphics[width=0.65\textwidth]{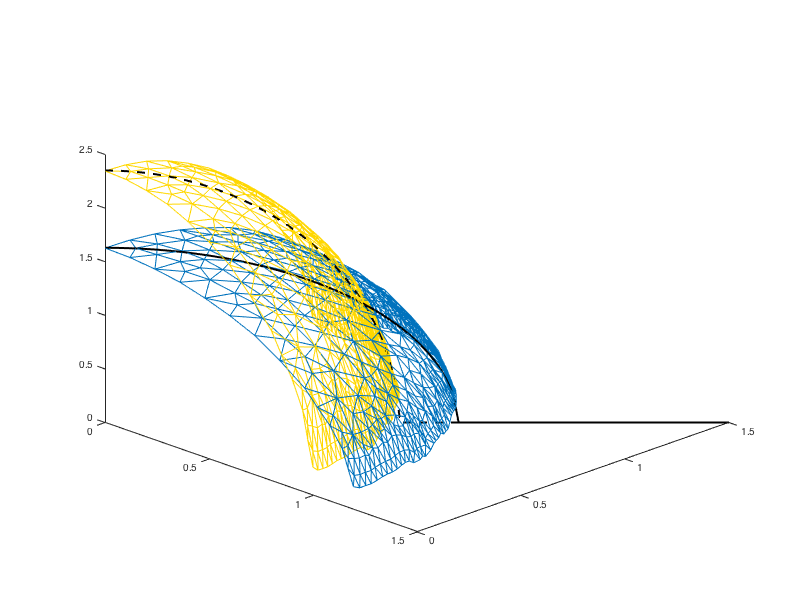}%
  \caption{Numerical experiment 1: comparison of the discrete solution (interpolated surface
    plots with triangulation) with the Barenblatt
    profile (solid and dashed black lines along the identity) at different times.}
  \label{fig:PMEenergy}
\end{figure}%
\begin{figure}
  \includegraphics[width=0.45\textwidth]{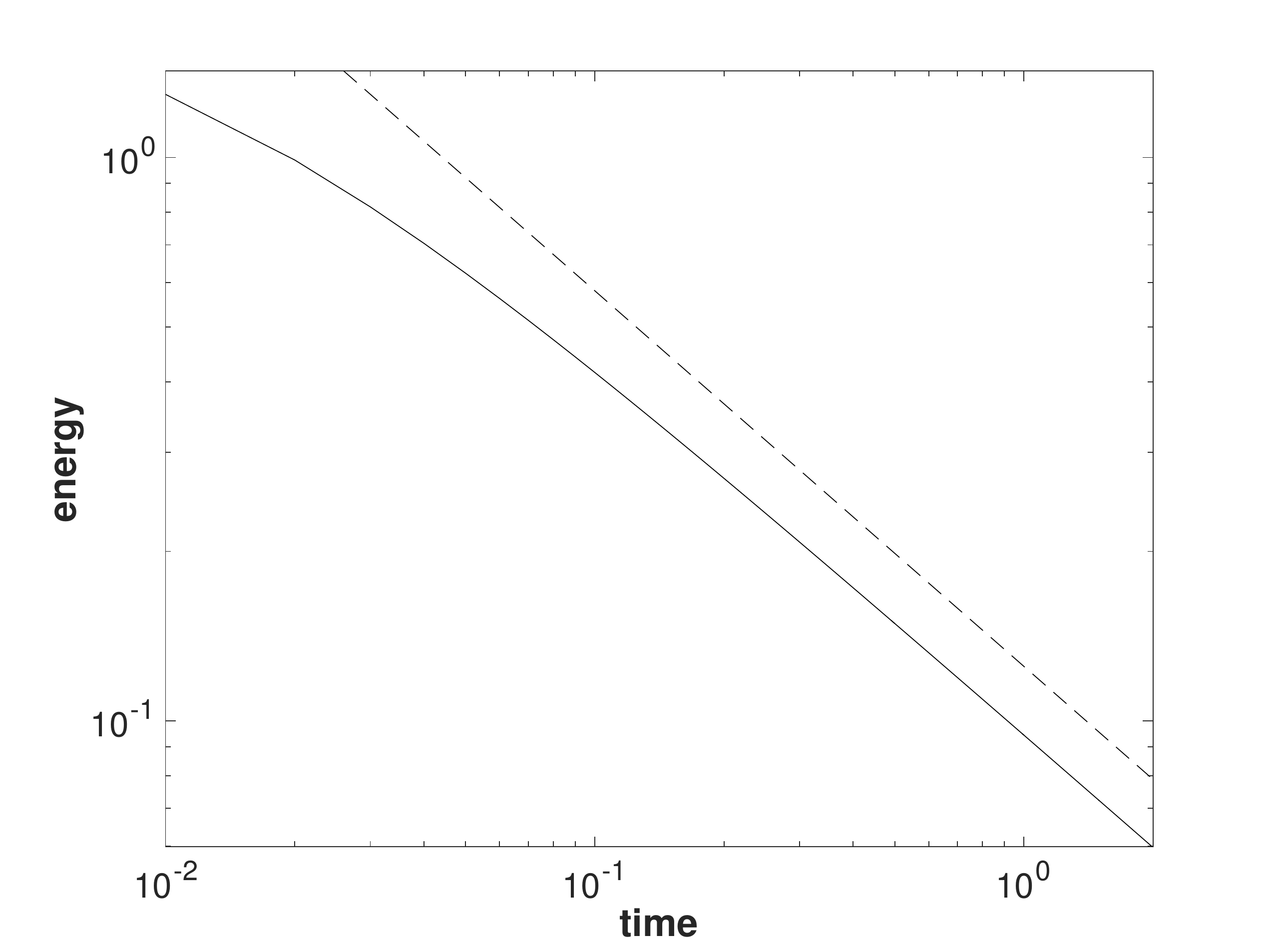}
  \includegraphics[width=0.45\textwidth]{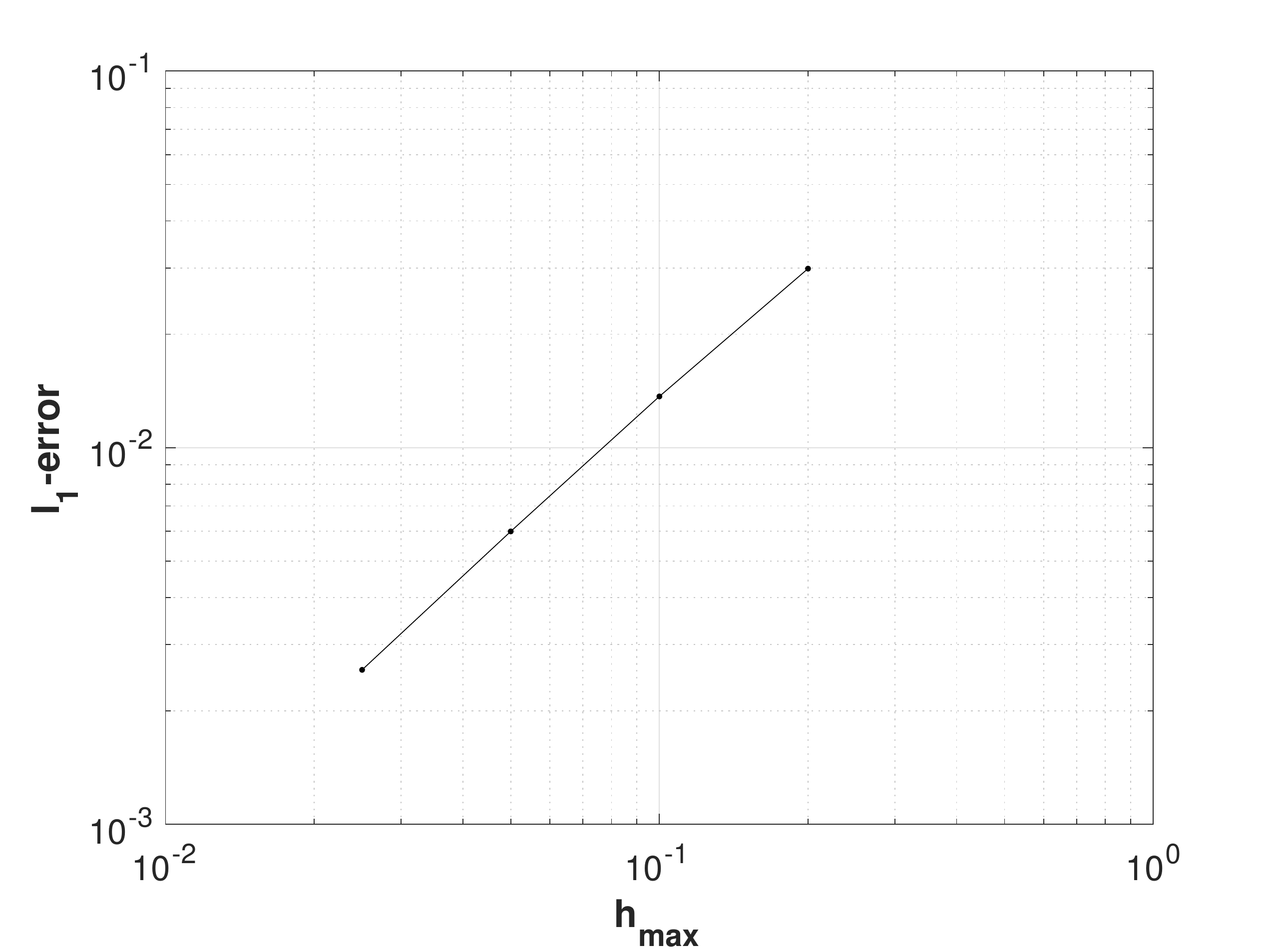}
  \caption{Numerical experiment 1: decay of the energy of the discrete solution in comparison with the
    analytical decay $t^{-2/3}$ of the Barenblatt solution
    (left). Numerical convergence for fixed ratio $\tau/h_{\rm
      max}^2=0.4$ (right).}
  \label{fig:PMEerror}
\end{figure}%
Figure~\ref{fig:PMEenergy} shows surface
    plots of the discrete solution at different times in comparison with the Barenblatt
    profile at the respective time.
By construction of the scheme, the initial mass is exactly conserved in time as the
discrete solution propagates.
The left plot in Figure~\ref{fig:PMEerror} shows the decay in the energy and gives quantitative information
about the difference of the discrete solution to the analytical Barenblatt solution.
The numerical solution shows good agreement with the analytical energy
decay rate $c=2/3$.

We also compute the $l_1$-error of the discrete solution to the
exact Barenblatt profile and observe that it remains within the order of the fineness of
the triangulation. The mass of the discrete solution is perfectly
conserved, as guaranteed by the construction of our method.

To estimate the convergence order of our method, we run several
experiments with the above initial data on different meshes. We fix
the ratio $\tau/h_{\rm max}^2=0.4$ and compute the $l_1$-error at time
$T=0.2$ on triangulations with $h_{\rm max}=0.2,\,0.1,\,0.05,\,0.025.$ We
expect the error to decay as a power of $h_{\rm max}$. The double
logarithmic plot should reveal a line with its slope indicating the
numerical convergence order. The right plot in Figure~\ref{fig:PMEerror} shows the
result, the estimated numerical convergence order which is obtained
from a least-squares fitted line through the points is equal to
$1.18$. This indicates first order convergence of the scheme with
respect to the spatial discretisation parameter $h_{\rm
      max}$.

\subsubsection*{Numerical experiment 2: Asymptotic self-similarity}
In our second example, we are still concerned with the free cubic porous medium equation,
\eqref{eq:PME3} with $V\equiv0$.
This time, we wish to give an indication that
the discrete approximation of the self-similar solution from~\eqref{eq:selfsim} from the previous experiment
might inherit the global attractivity of its continuous counterpart.
More specifically, we track the discrete evolution for the initial datum
\begin{align}
  \label{eq:exp2ic}
  \rho_0(x,y)= 3000(x^2+y^2)\exp[-5(|x|+|y|)]+0.1
\end{align}
until time $T=0.1$ and observe that it appears to approach the self-similar solution from above.
Snapshots of the simulation are collected in Figure~\ref{fig:PME4evol}.
\begin{figure}
  \includegraphics[width=0.45\textwidth]{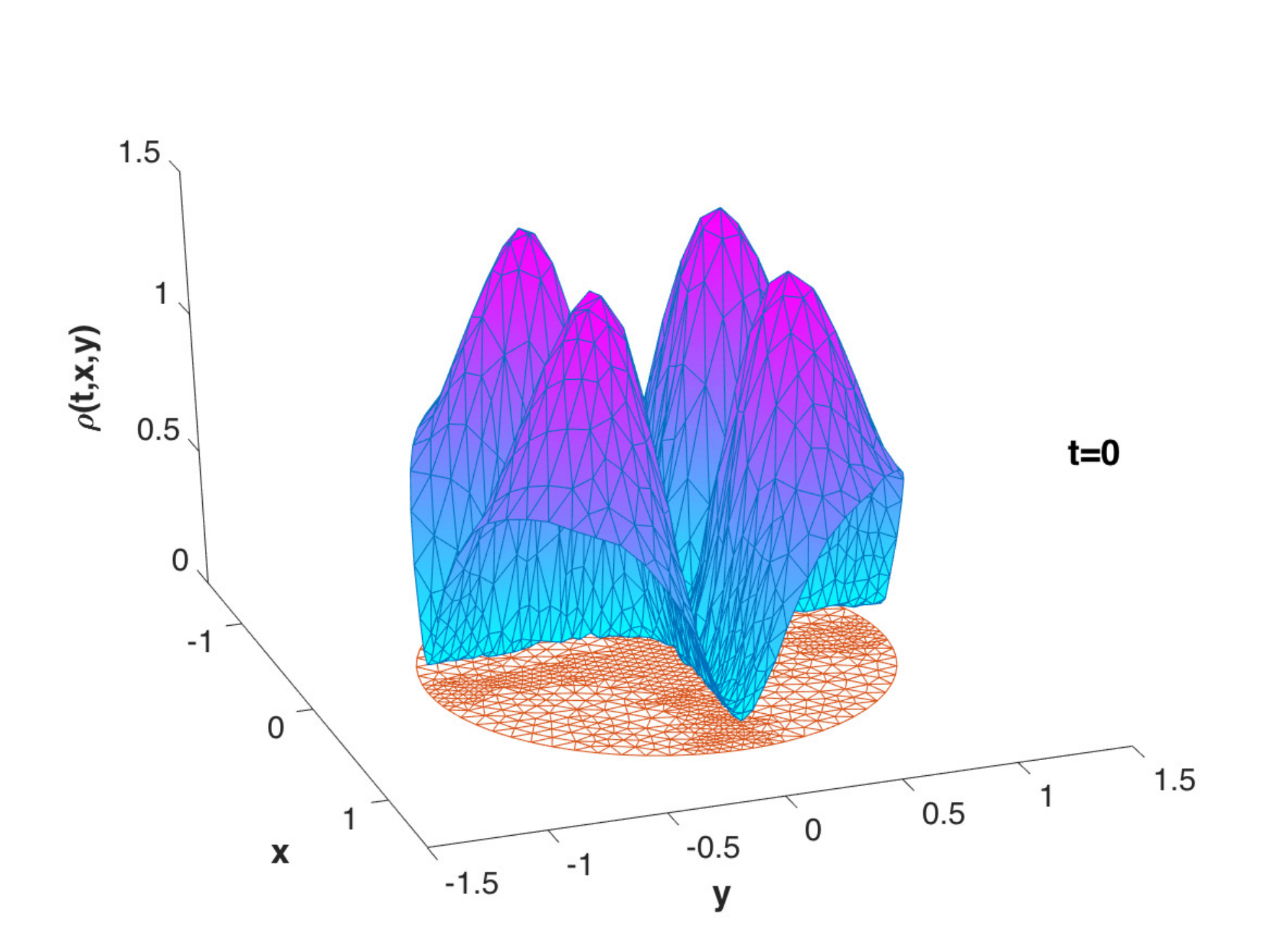}%
  \includegraphics[width=0.45\textwidth]{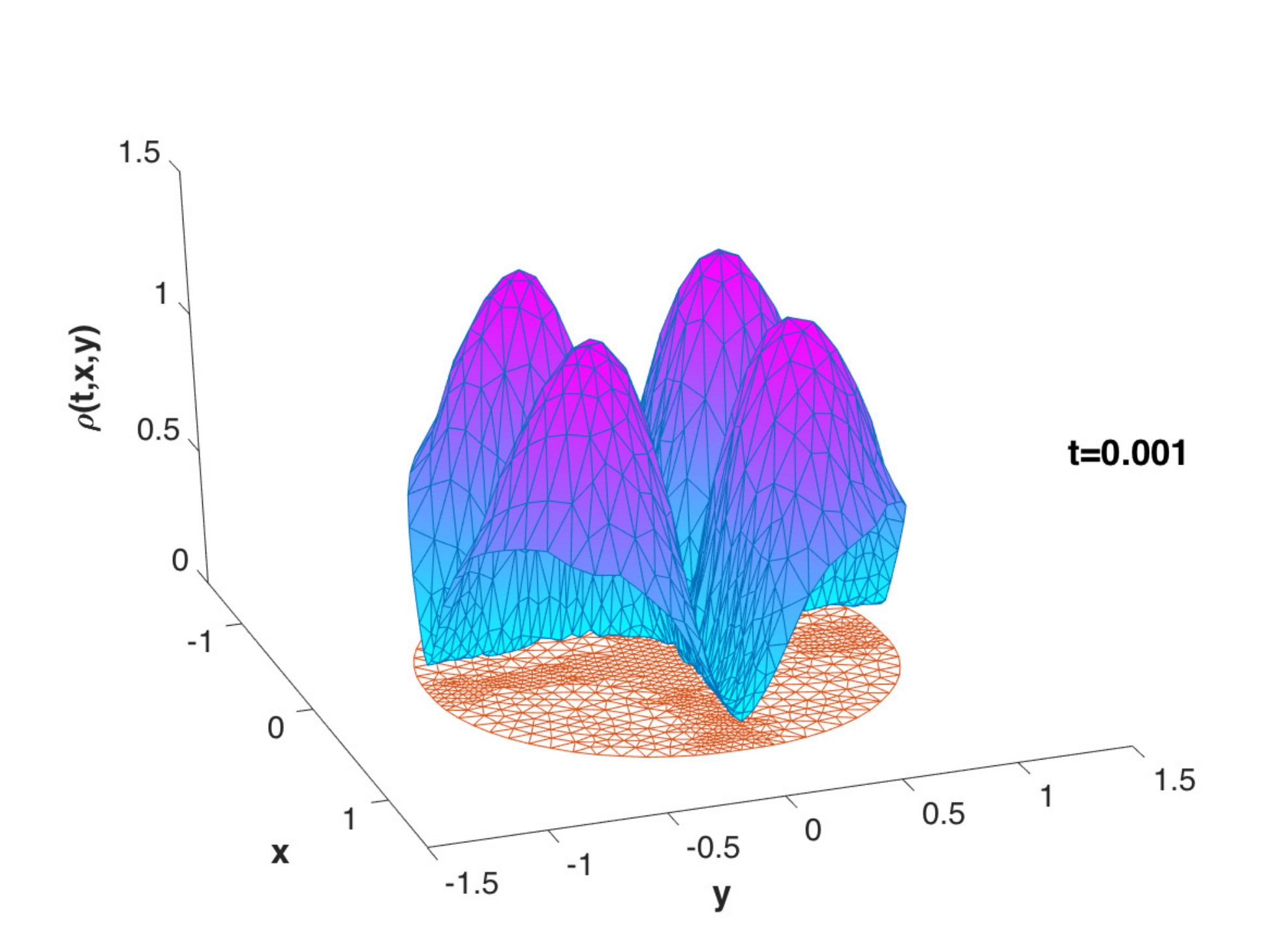}\\
  \includegraphics[width=0.45\textwidth]{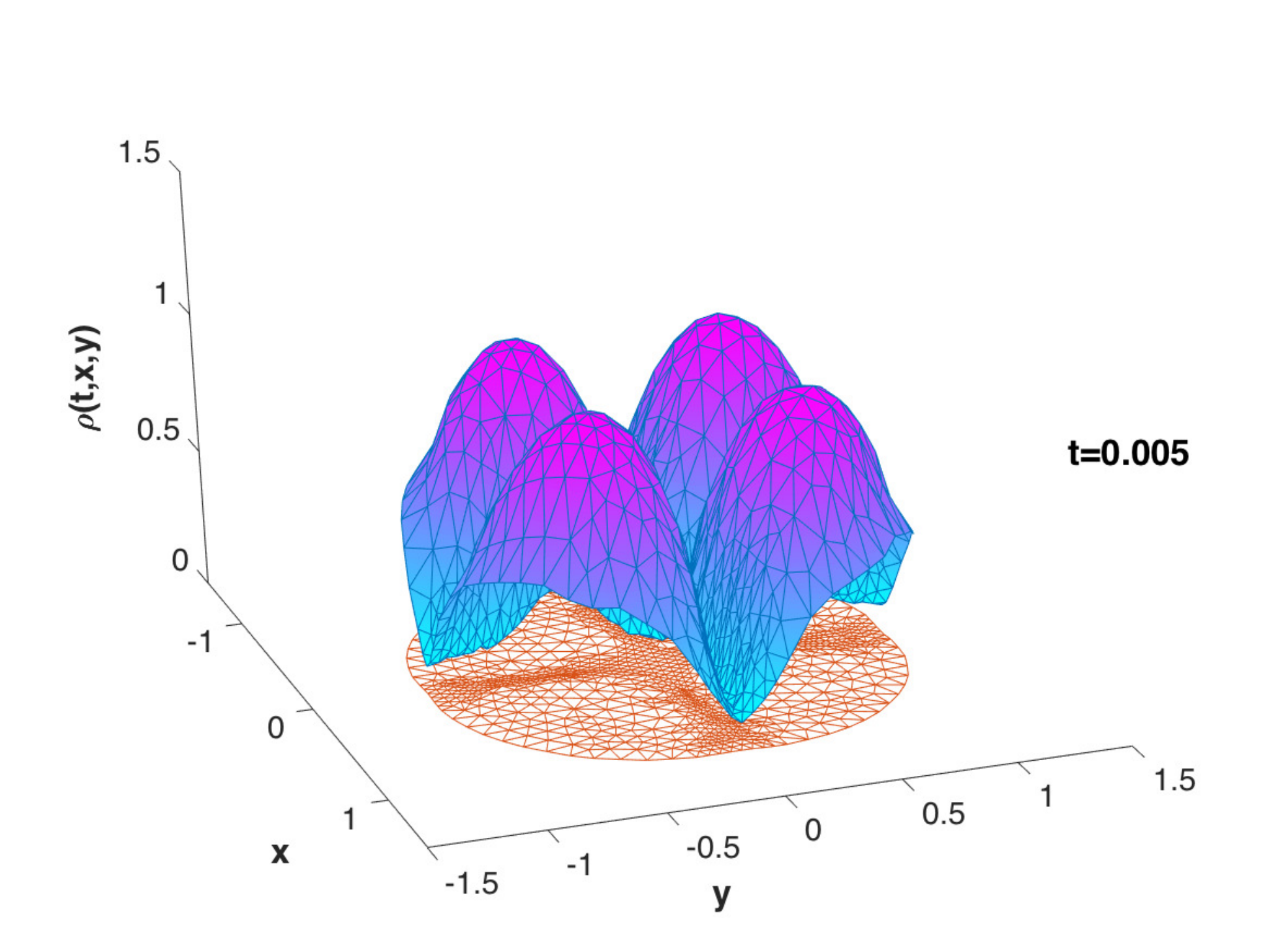}%
  \includegraphics[width=0.45\textwidth]{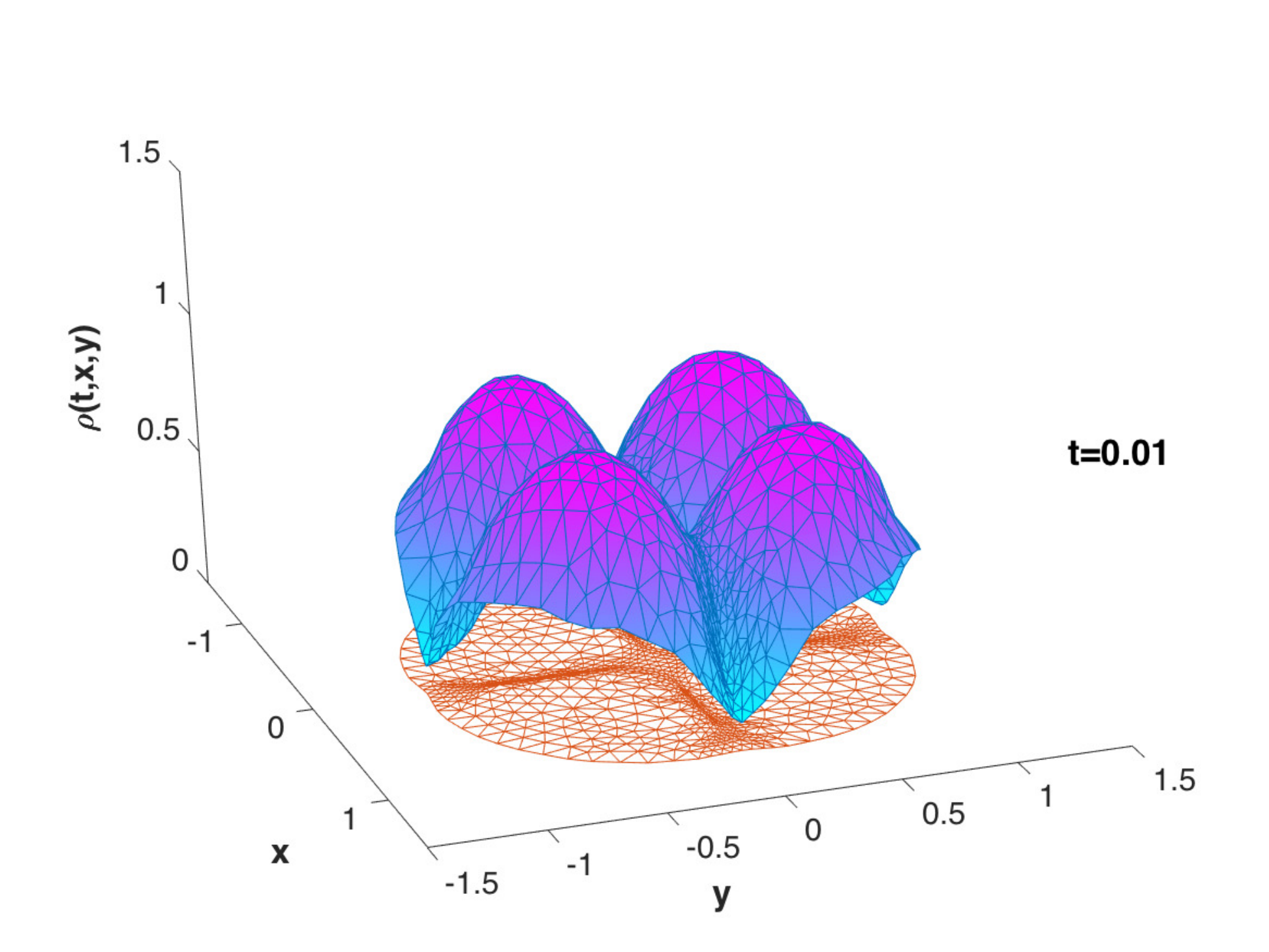}\\
  \includegraphics[width=0.45\textwidth]{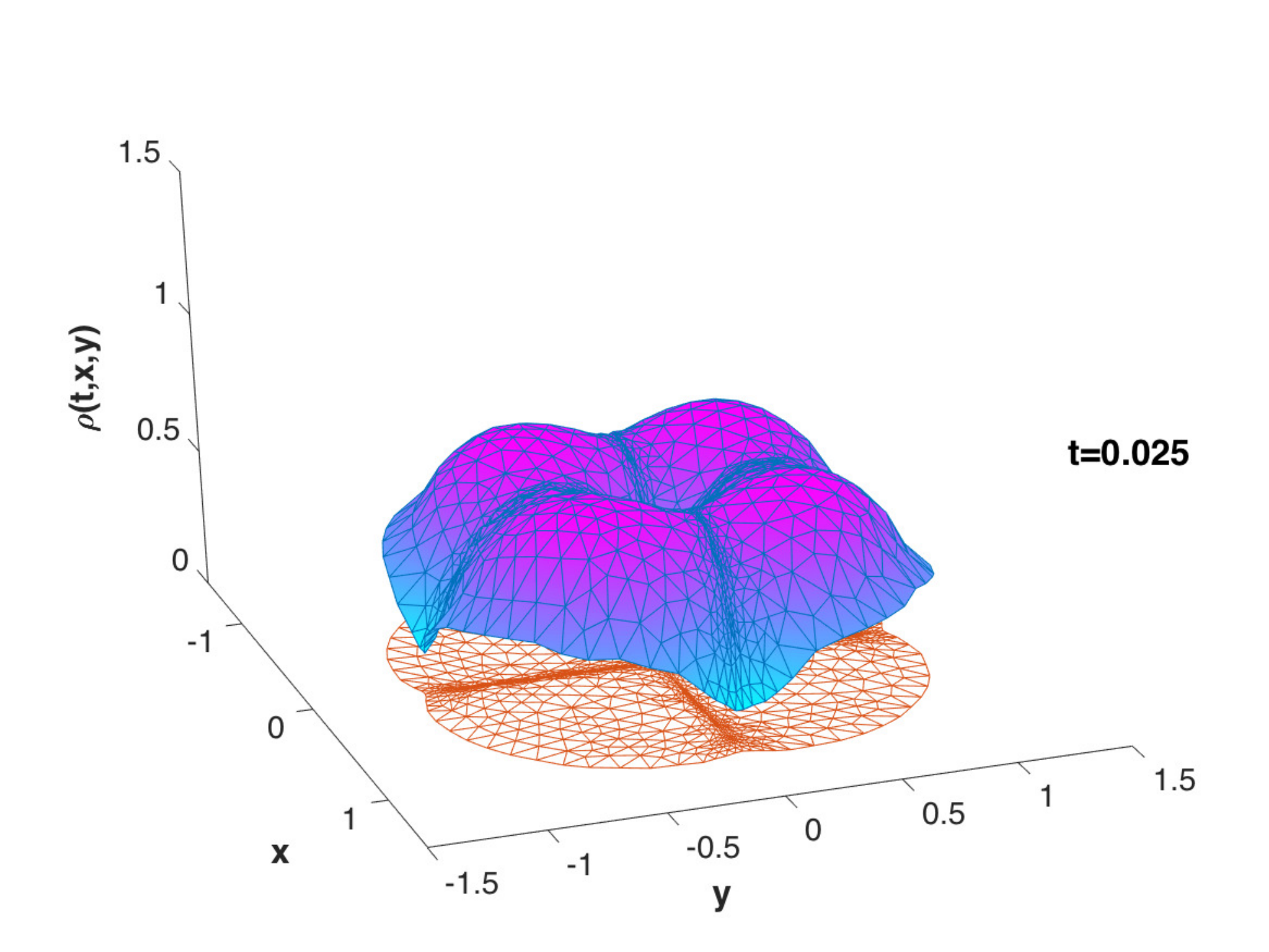}%
  \includegraphics[width=0.45\textwidth]{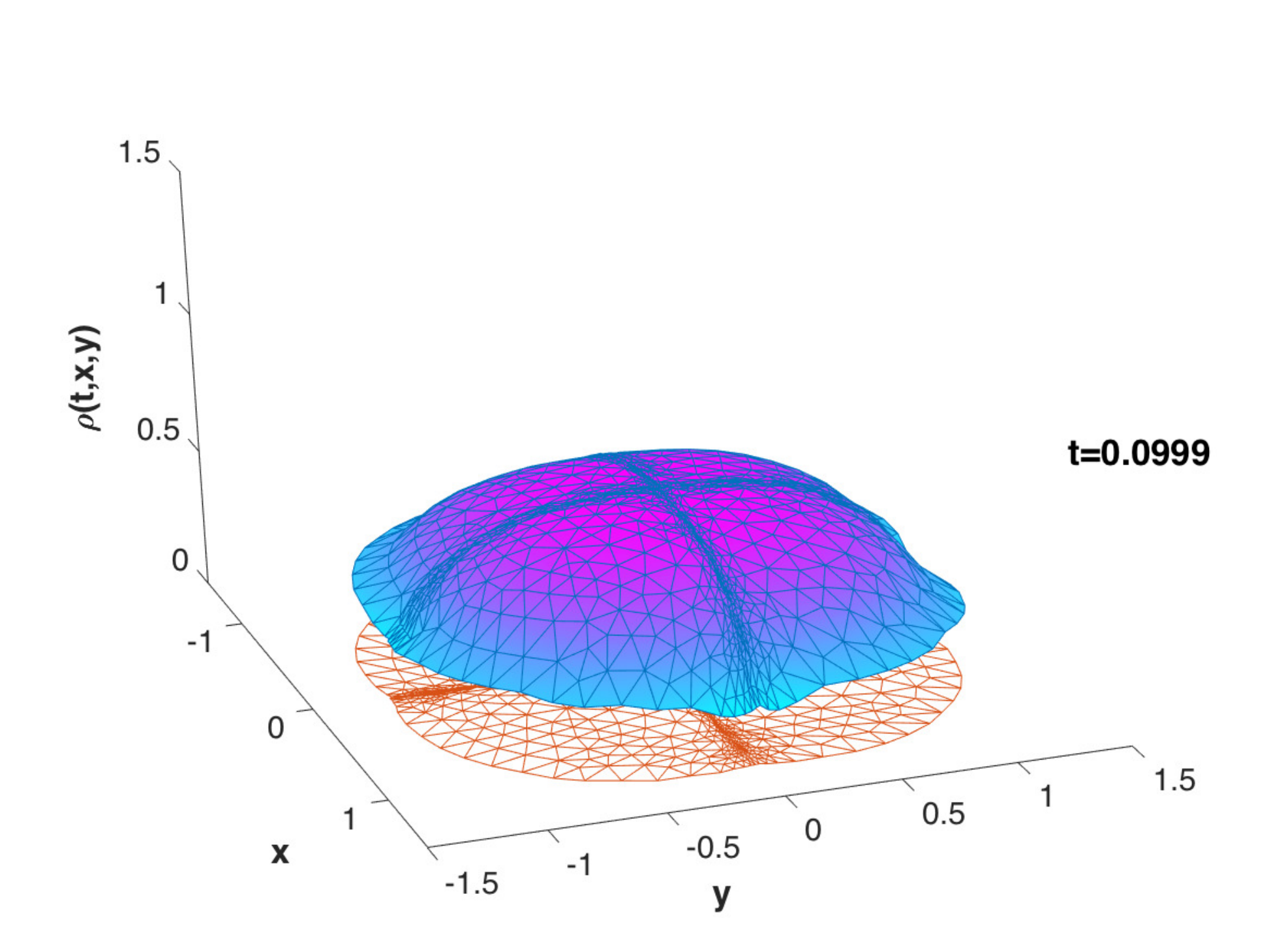}
  \caption{Numerical experiment 2: fully discrete evolution for the initial density from~\eqref{eq:exp2ic}
    under the free porous medium equation.
    Snapshots are taken at times $t=0.001$, $t=0.005$, $t=0.01$, $t=0.025$, and $t=0.1$.}
  \label{fig:PME4evol}
\end{figure}

\subsubsection*{Numerical experiment 3: two peaks merging into one under the influence of a confining potential}
In this example we consider as initial condition two peaks, connected by a thin layer of mass, given by
\begin{equation}
\label{eq:peaks}
\rho_0(x,y)=\exp[-20((x-0.35)^2+(y-0.35)^2)]+\exp[-20((x+0.35)^2+(y+0.35)^2 )]+0.001.
\end{equation}
We choose a triangulation of the square $[-1.5,1.5]^2$ and initialise
the discrete solution piecewise constant in each triangle, with a
value corresponding to~\eqref{eq:peaks},
evaluated in the centre of mass of each triangle.
We solve the porous medium equation with a confining potential, i.e.\
\eqref{eq:NFPsystem} with $P(r)=r^m$ and $V(x,y)=5(x^2+y^2)/2$.
The time step is $\tau=0.001$ and the final time is $T=0.2.$

Figure~\ref{fig:PME6evol} shows the evolution from the initial
density. As time increases the peaks smoothly merge into each
other. As the thin layer around the peaks is also subject to the
potential the triangulated domain shrinks in time. 
Even if we do not know how to prevent theoretically the intersection of the images of the discrete Lagrangian maps,
this seems not to be a problem in practice.
\begin{figure}
  \includegraphics[width=0.45\textwidth]{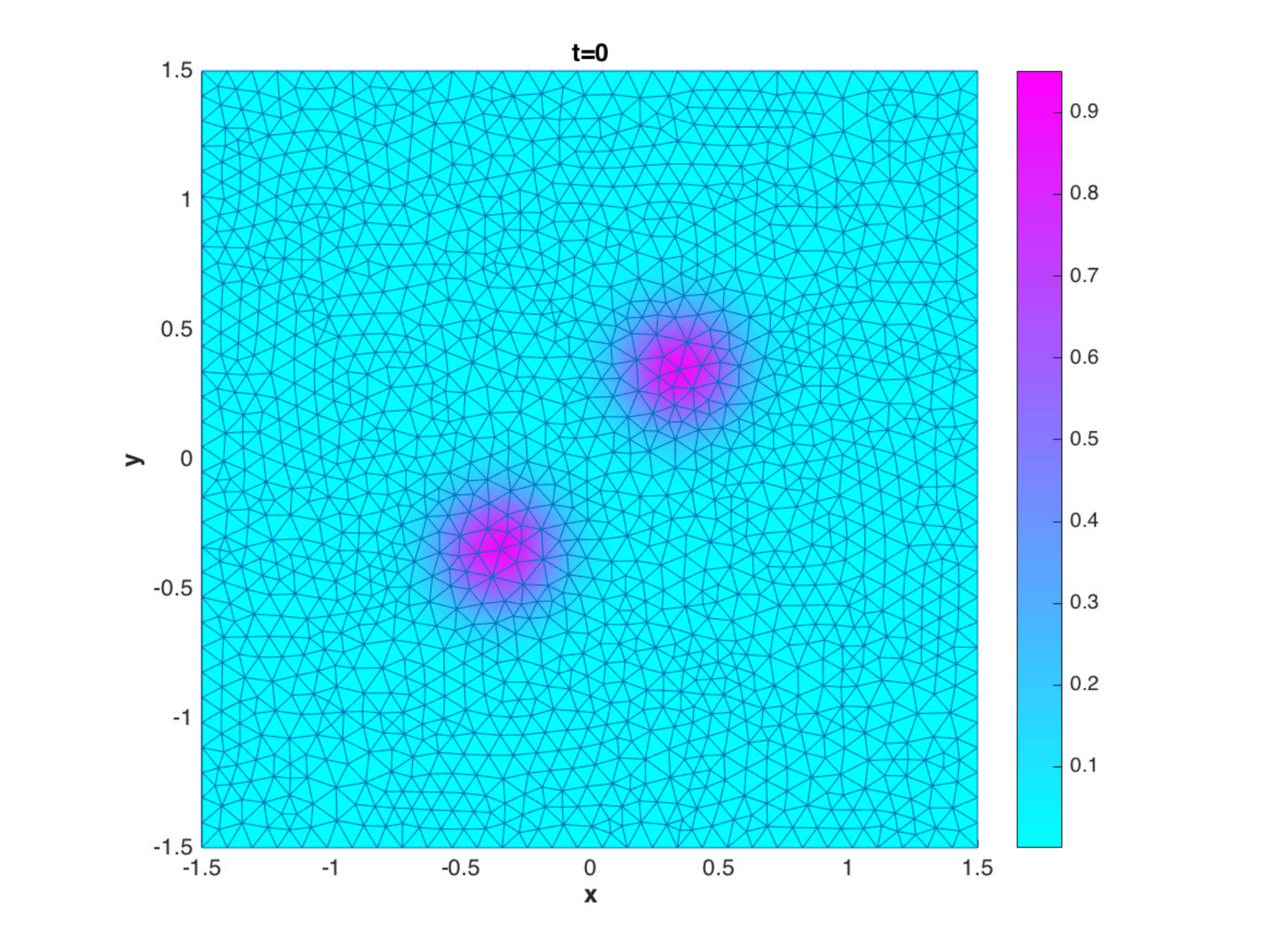}%
  \includegraphics[width=0.45\textwidth]{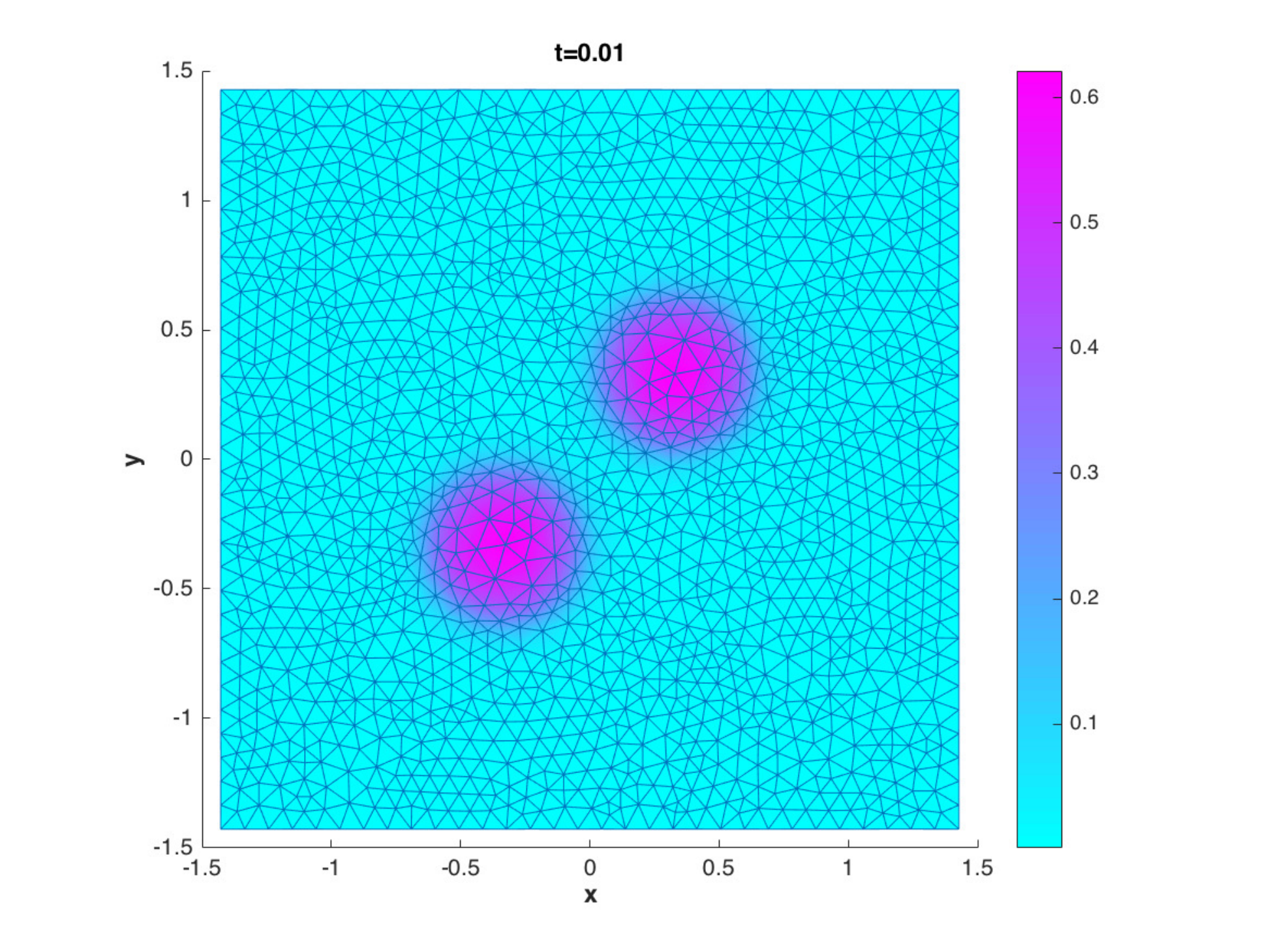}\\
  \includegraphics[width=0.45\textwidth]{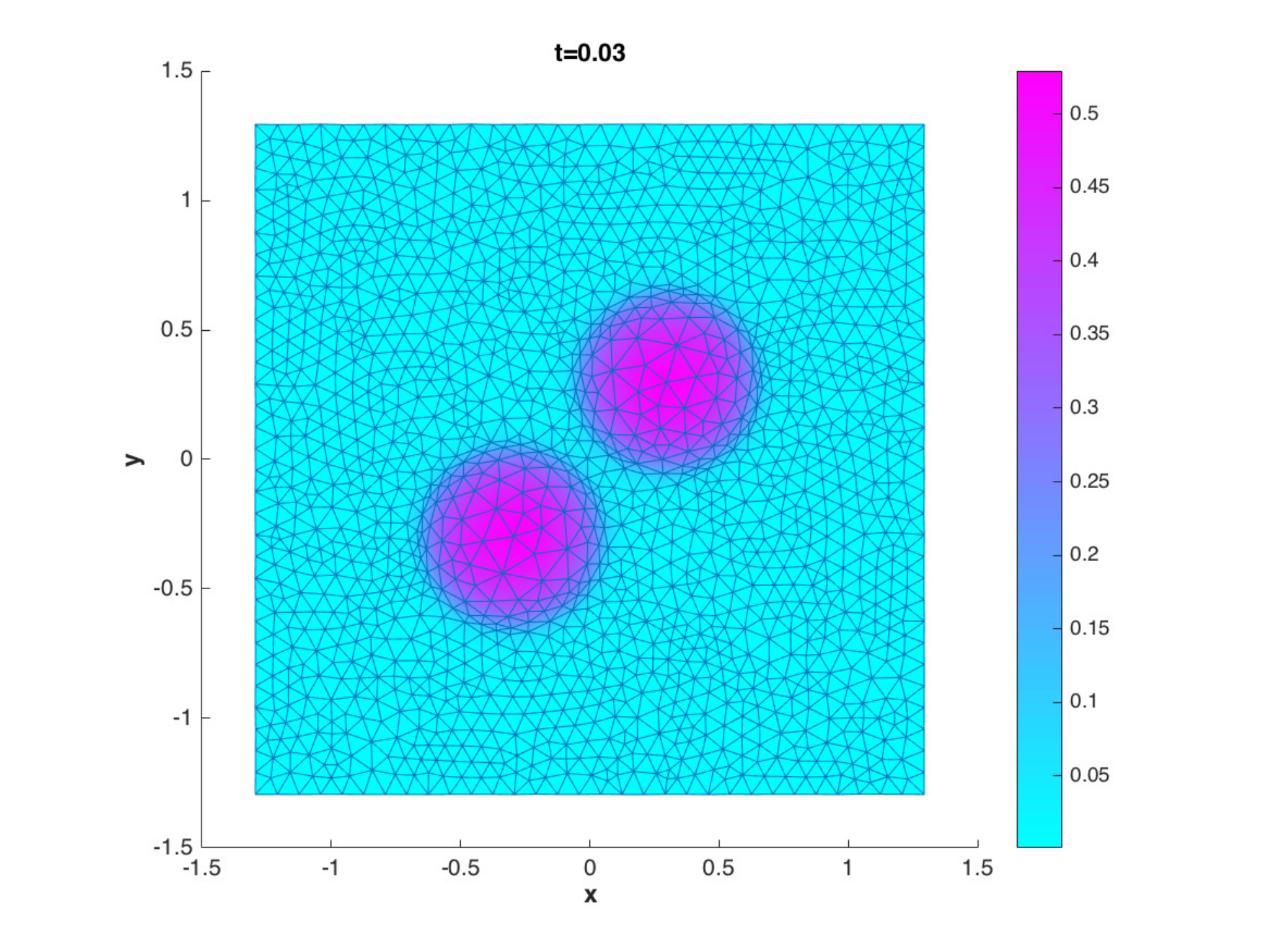}%
  \includegraphics[width=0.45\textwidth]{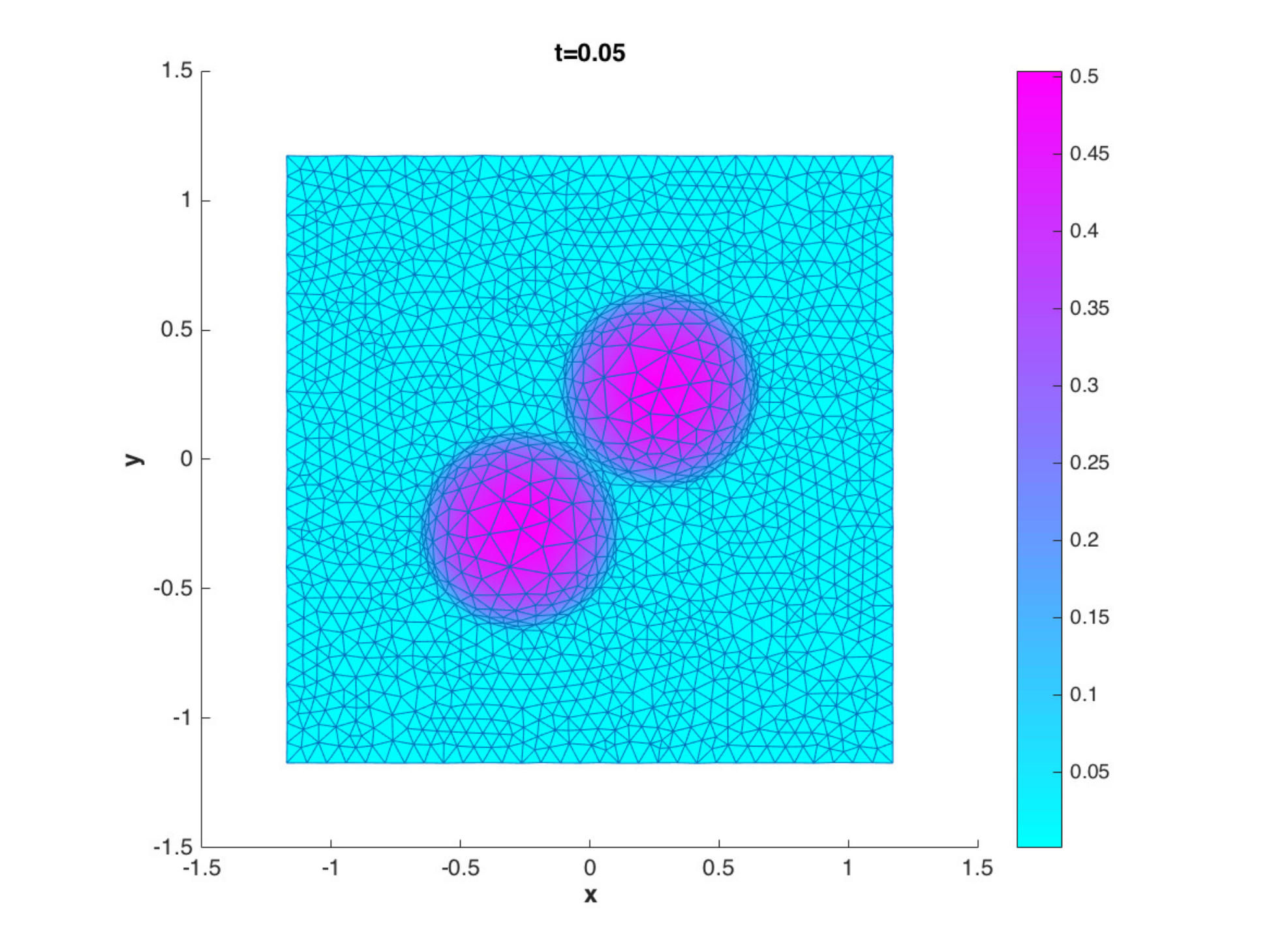}\\
  \includegraphics[width=0.45\textwidth]{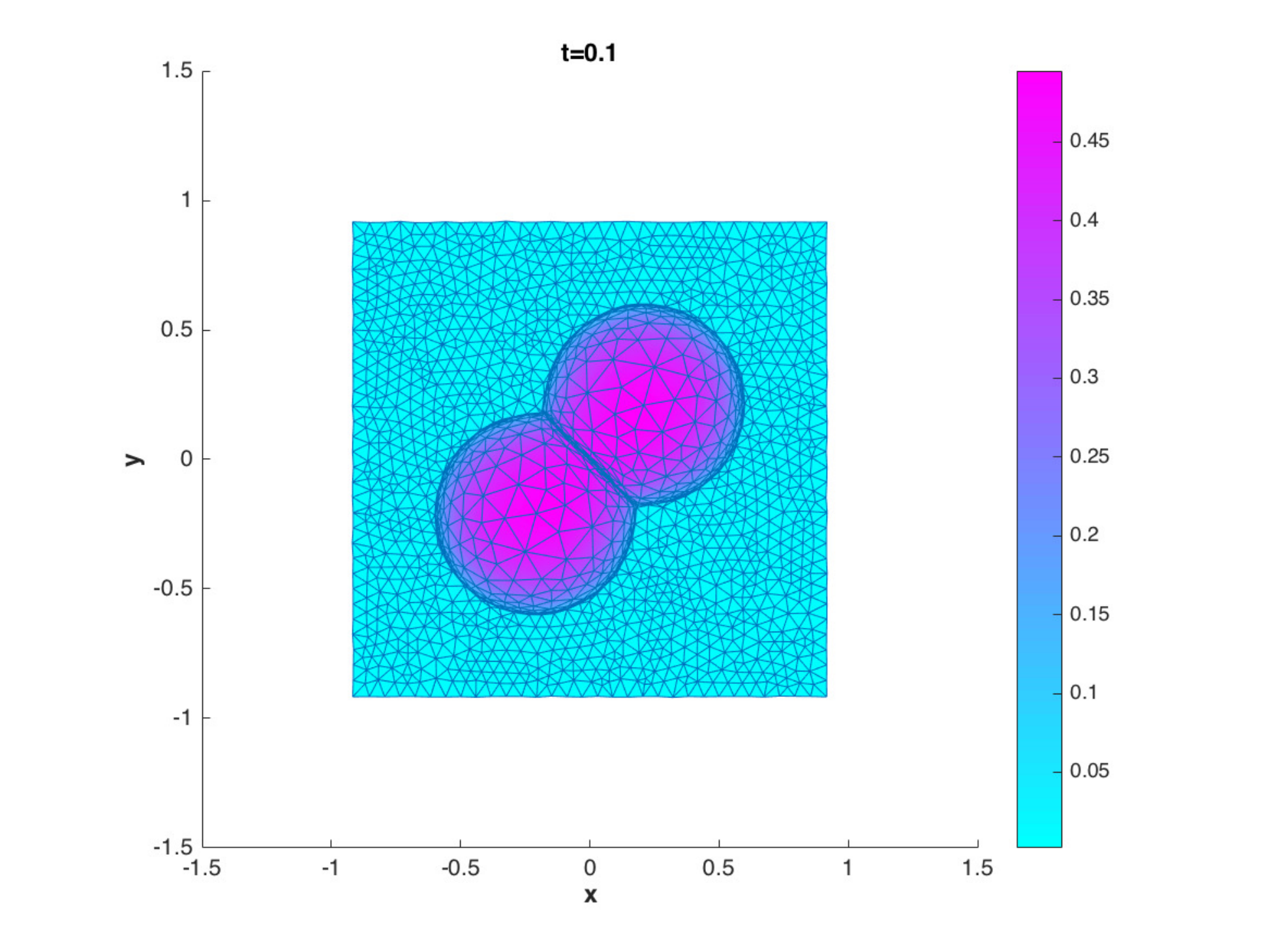}%
  \includegraphics[width=0.45\textwidth]{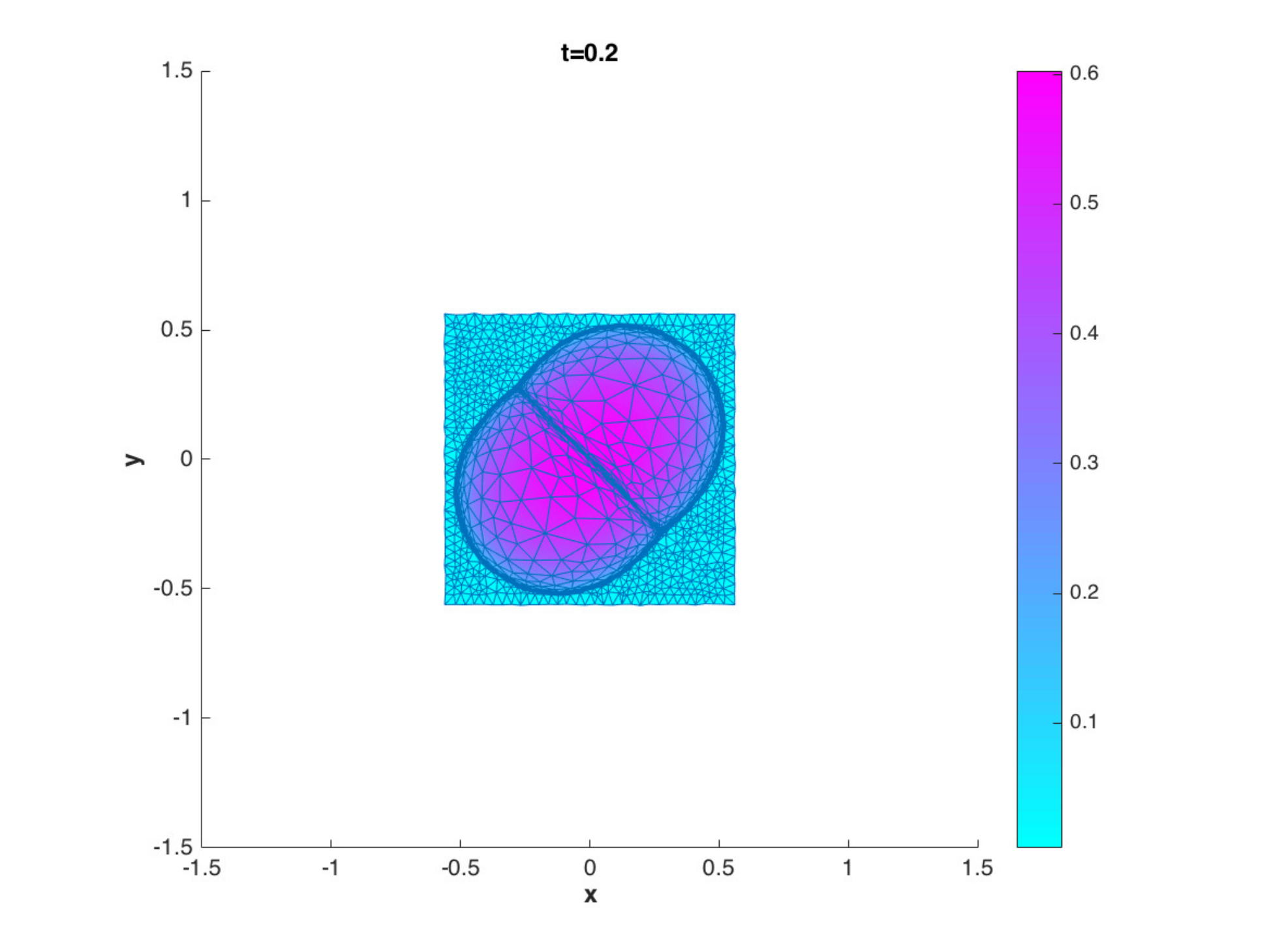}
  \caption{Numerical experiment 3: evolution of two peaks merging under the porous medium
    equation with a confining potential.}
  \label{fig:PME6evol}
\end{figure}
As time evolves, the discrete solution approaches the
steady state Barenblatt profile given by
\begin{equation}
  \label{eq:BBprofile2}
  \BB(z) = \left(C-\frac53 ||z||^2 \right)_+^{\frac12},
\end{equation}
where $C$ is chosen as the mass of the density.
The plot in Figure~\ref{fig:PME6l1dist} shows the
exponential decay of the $l_1$-distance of the discrete solution to the
    steady state Barenblatt profile \eqref{eq:BBprofile2}. We observe
    that the decay agrees very well with the analytically predicted
    decay $\exp(-5t)$ until $t=0.08$.
For larger times,
one would monitor triangle quality numerically, and re-mesh, locally
coarsening the triangulation where necessary.
\begin{figure}
  \includegraphics[width=0.45\textwidth]{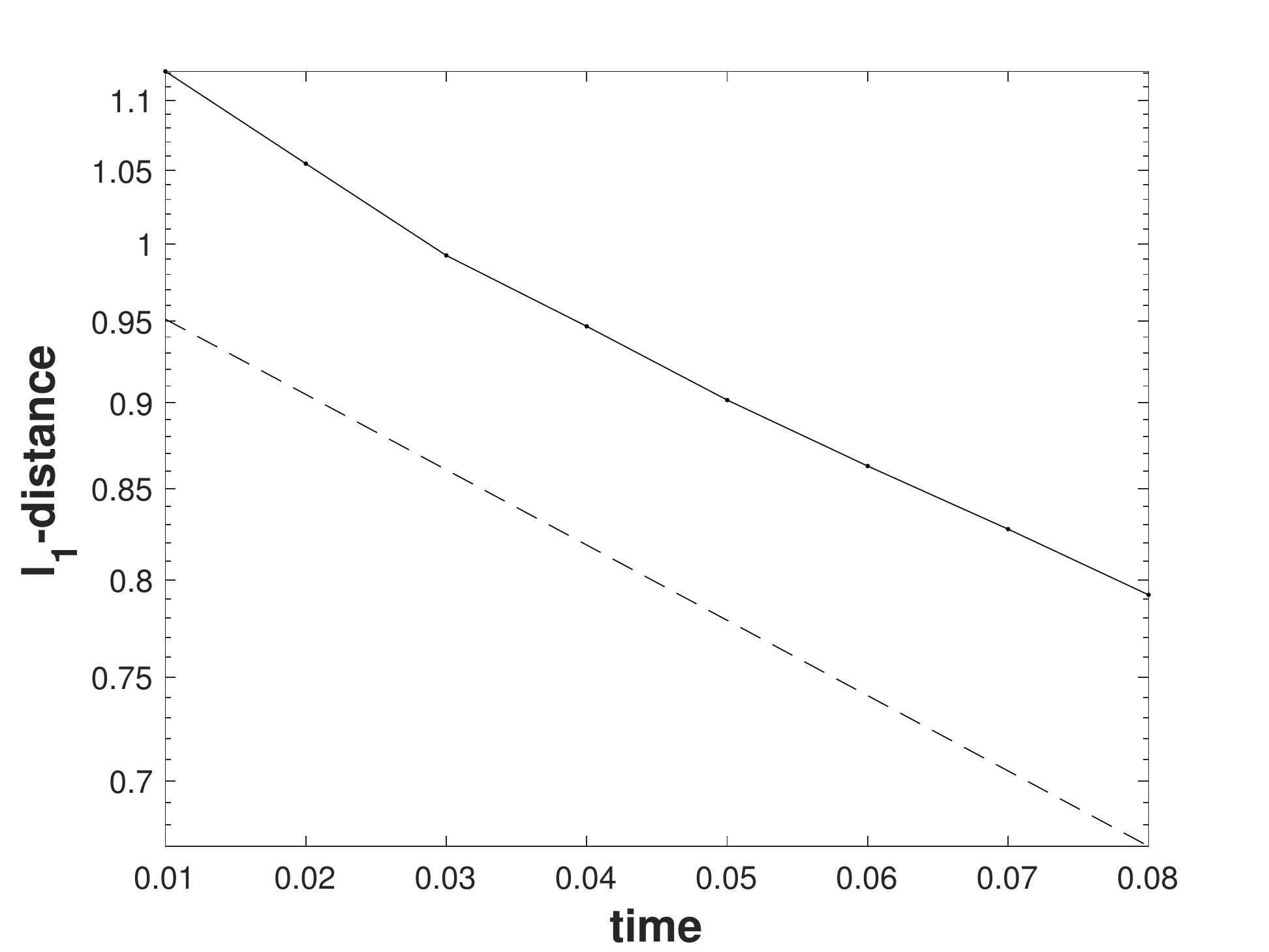}%
  \caption{Numerical experiment 3: two merging peaks: plot of the $l_1$-distance of the discrete solution to the
    steady state Barenblatt profile in comparison with the analytical
    decay $c\exp(-5t)$.}
  \label{fig:PME6l1dist}
\end{figure}

\subsubsection*{Numerical experiment 4: one peak splitting under the influence of a quartic potential}
We consider as the initial condition
\begin{equation}
\label{eq:bump}
\rho_0(x,y)=1-(x^2+y^2).
\end{equation}
We choose a triangulation of the unit circle and initialise
the discrete solution piecewise constant in each triangle, with a
value corresponding to~\eqref{eq:bump},
evaluated in the centre of mass of each triangle.
We solve the porous medium equation with a quartic potential, i.e.\
\eqref{eq:NFPsystem} with $P(r)=r^m$ and $V(x)=5(x^2+(1-y^2)^2)/2$.
The time step is $\tau=0.005$ and the final time is $T=0.02.$

Figure~\ref{fig:PME5evol} shows the evolution of the initial
density. As time increases the initial density is progressively split,
until two new maxima emerge which are connected by a thin layer.  For larger times,
when certain triangles become excessively distorted, one would monitor
triangle quality numerically, and re-mesh, locally
refining the triangulation where necessary.
\begin{figure}
  \includegraphics[width=0.45\textwidth]{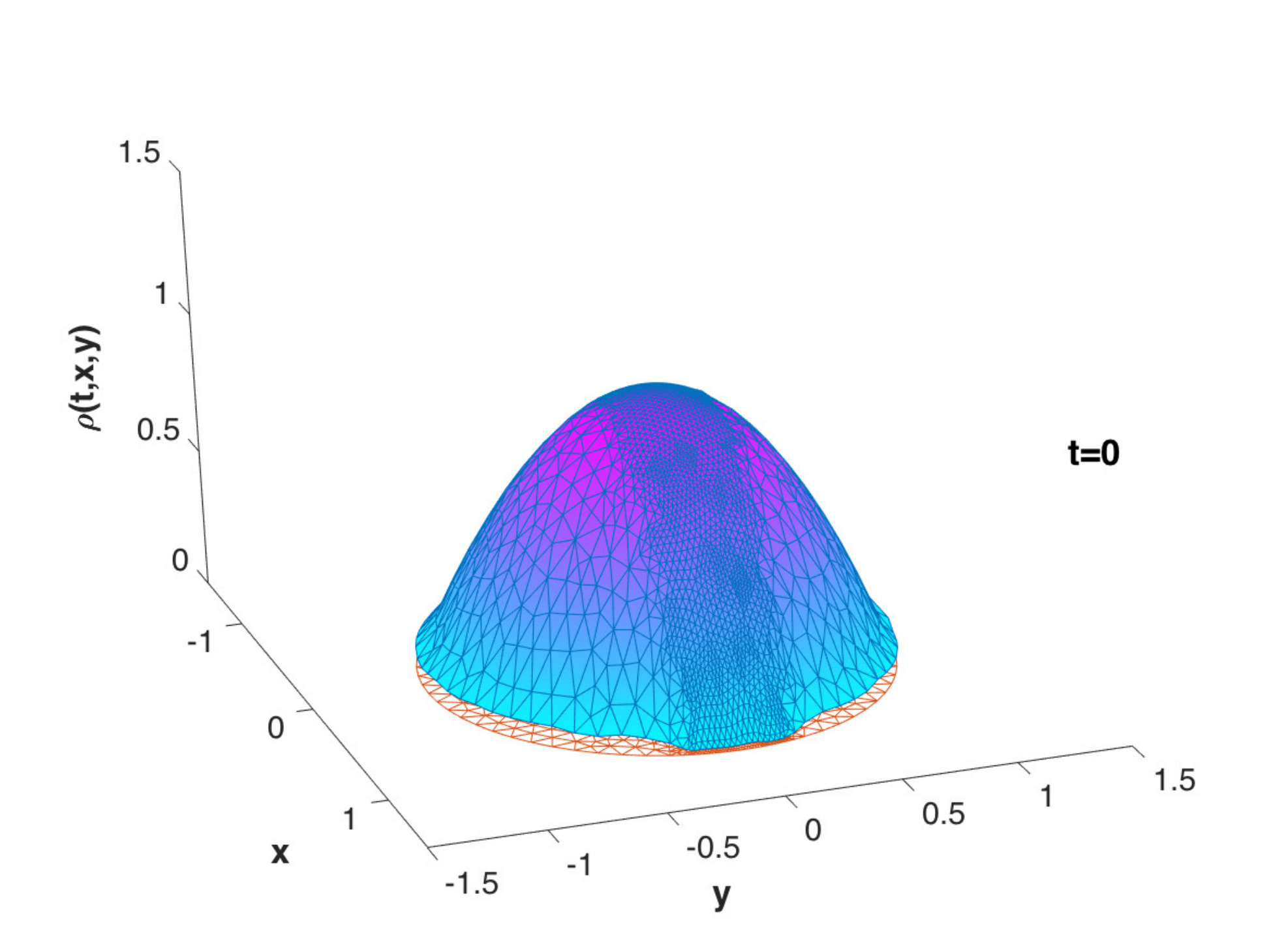}%
  \includegraphics[width=0.45\textwidth]{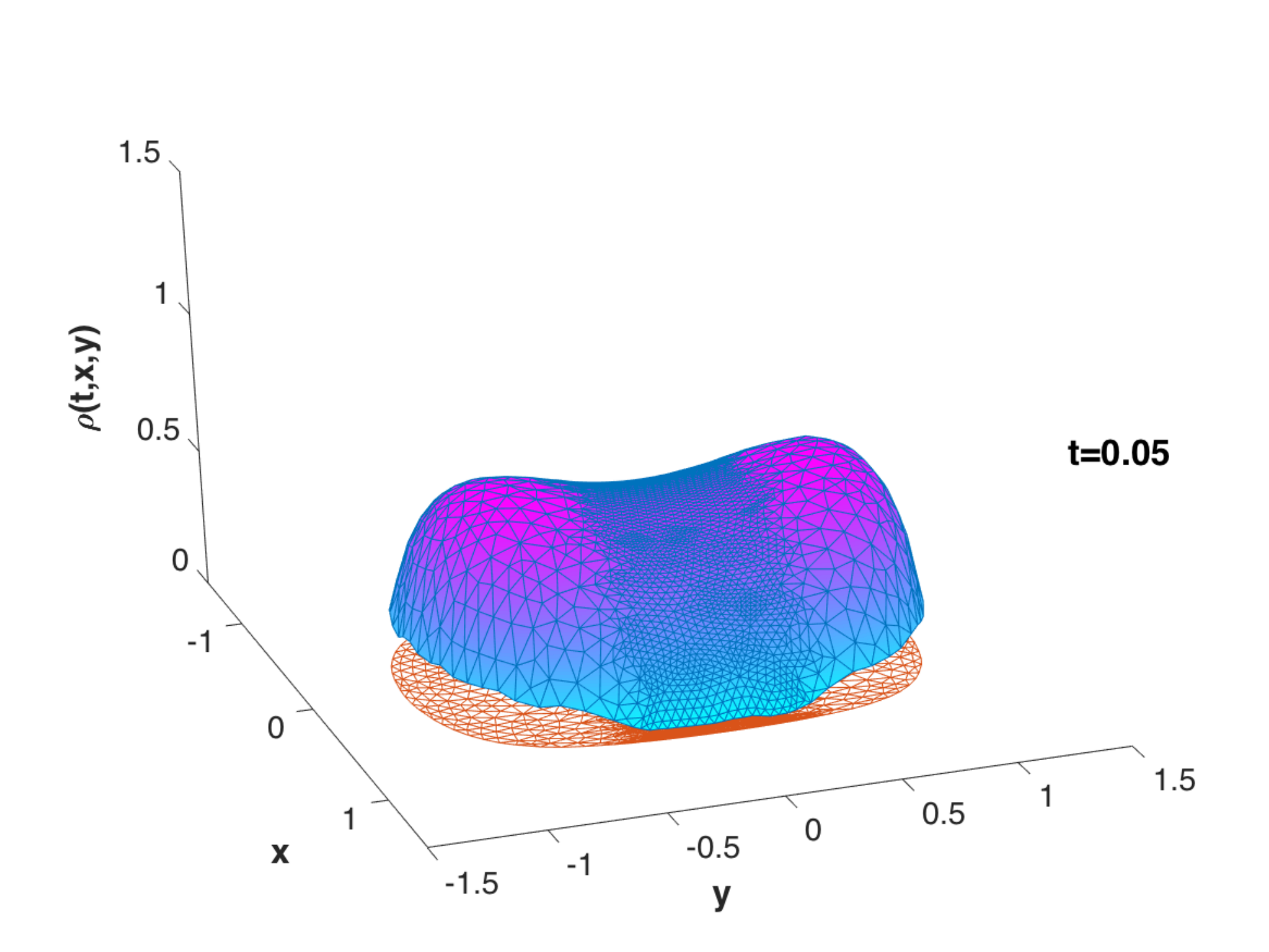}\\
  \includegraphics[width=0.45\textwidth]{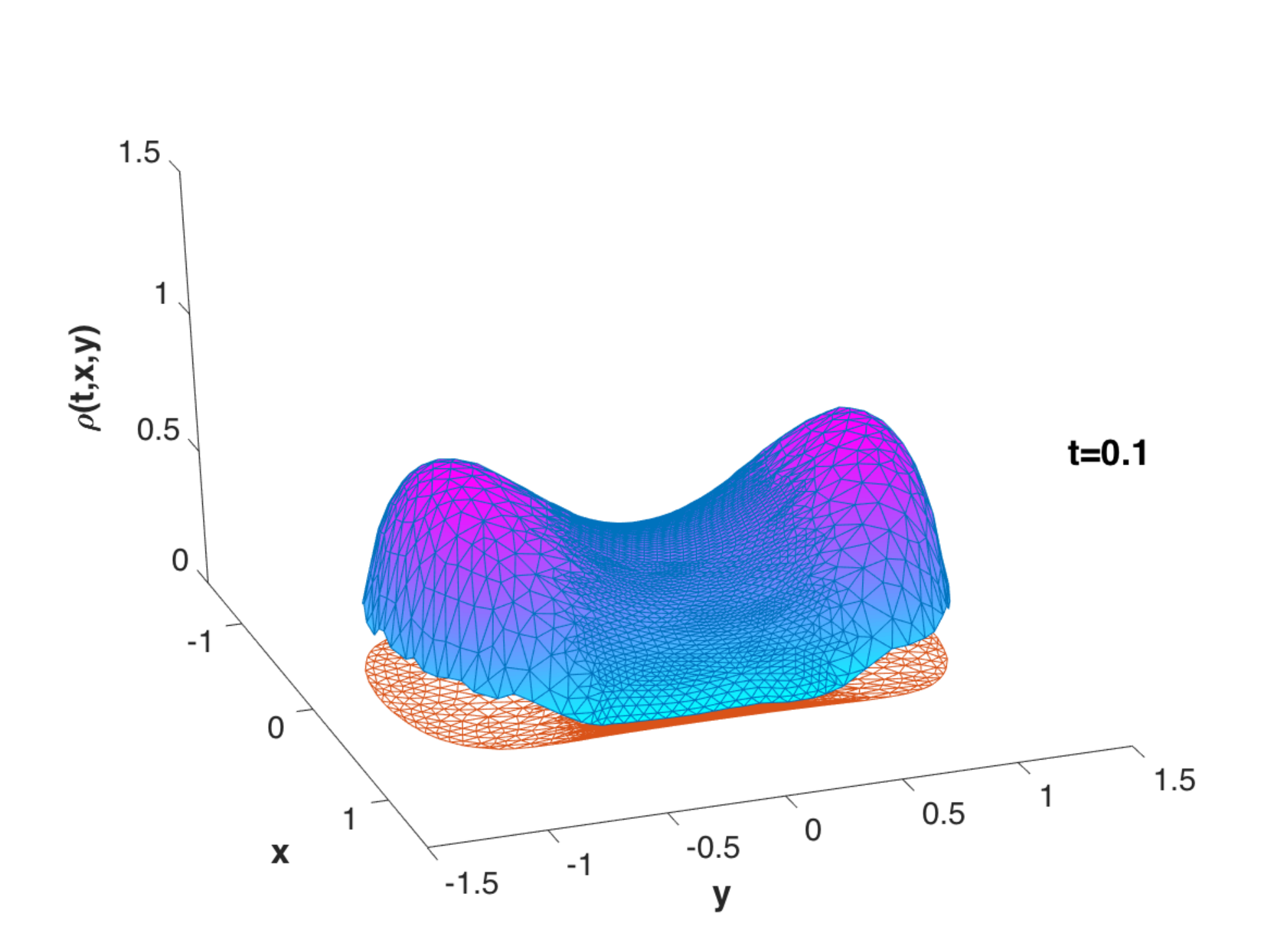}%
  \includegraphics[width=0.45\textwidth]{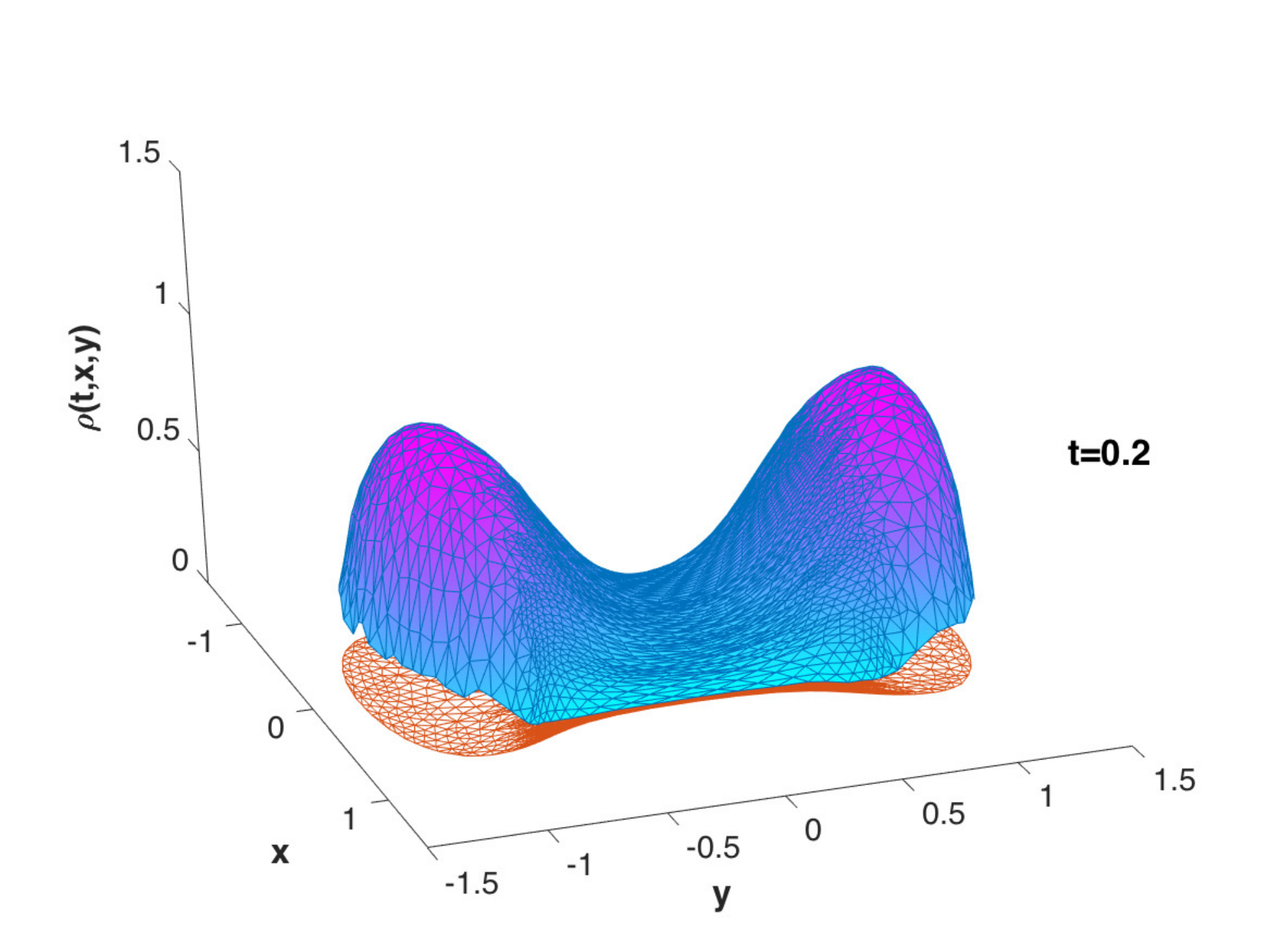}\\
  \caption{Numerical experiment 4: evolution of the initial density under the porous medium
    equation with a quartic potential.}
  \label{fig:PME5evol}
\end{figure}

\appendix

\section{Proof of the Lagrangian representation}
\label{sct:Lagrange}
\begin{proof}[Proof of Lemma~\ref{lem:Lagrange}]
  We verify that the density function given by $(G_t^{-1})_{\#}\rho_t$ on $K\subset{\R^d}$ is constant with respect to time $t$;
  the identity~\eqref{eq:push} then follows since
  \begin{align*}
    \rho_t = (G_t\circ G_t^{-1})_\#\rho_t = (G_t)_\#\big[(G_t^{-1})_{\#}\rho_t\big] = (G_t)_\#\big[(G_0^{-1})_\#\rho^0\big] = (G_t)_\#\refrho.
  \end{align*}
  Firstly, from the definition of the inverse,
  \[
  G_t^{-1}\circ G_t = \id
  \]
  for all $t$, differentiating with respect to time yields
  \[
  \df (G_t^{-1}) \circ G_t\, \pd_{t} G_t + \pd_{t} (G_t^{-1}) \circ G_t = 0,
  \]
  and so, using~\eqref{eq:Lagrange} and~\eqref{eq:velo},
  \begin{equation}
    \label{eqn:GTimeDiff}
    \pd_{t} (G_t^{-1}) = - \df (G_t^{-1}) (\pd_{t} G_t \circ G_t^{-1}) = - \df (G_t^{-1})\velo[\rho_t].
  \end{equation}
  Now, let $\varphi$ be a smooth test function, and consider
  \begin{align*}
    &\frac{\dn}{\dn t} \int \varphi \,(G_t^{-1})_{\#} \rho_t \\
    &\qquad = \frac{\dn}{\dn t} \int (\varphi \circ G_t^{-1}) \rho_t \\
    &\qquad = \int (\varphi \circ G_t^{-1}) \pd_{t} \rho_t + \int \df \varphi \circ G_t^{-1} \pd_{t} (G_t^{-1}) \rho_t \\
    &\qquad = -\int (\varphi \circ G_t^{-1}) [\Div (\rho_t v(\rho_t))] - \int (\df \varphi \circ G_t^{-1})\,\df (G_t^{-1})\,v(\rho_t) \rho_t
    && \text{by~\eqref{eq:NFPsystem} and~\eqref{eq:Lagrange} } \\
    &\qquad = \int (\df \varphi \circ G_t^{-1}) \df (G_t^{-1}) \,[v(\rho_t) - v(\rho_t)] \rho_t && \text{integrating by parts} \\
    &\qquad = 0.
  \end{align*}
  As $\varphi$ was arbitrary, $(G_t^{-1})_{\#}\rho_t$ is constant with respect to time.
\end{proof}

\section{Technical lemmas}
\begin{lemma}
  \label{lem:triint}
  Given $g_0,g_1,\ldots,g_d\in\R^d$, then
  \begin{align}
    \label{eq:triint}
    \aint_{\stdtri^d}\Big\|g_0+\sum_{j=1}^d\omega_j(g_j-g_0)\Big\|^2\dd\omega
    = \frac2{(d+1)(d+2)}\sum_{0\le i\le j\le d}g_i\cdot g_j.
  \end{align}
\end{lemma}
\begin{proof}
  Thanks to the symmetry of the integral with respect to the exchange of the components $\omega_j$,
  the left-hand side of~\eqref{eq:triint} equals to
  \begin{align}
    \label{eq:triint1}
    \begin{split}
      &\|g_0\|^2
      + 2\left(\aint_{\stdtri}\omega_d\dd\omega\right)\sum_{1\le j\le d}g_0\cdot(g_j-g_0) \\
      & + \left(\aint_{\stdtri}\omega_d^2\dd\omega\right)\sum_{1\le j\le d}\|g_j-g_0\|^2
      + 2\left(\aint_{\stdtri}\omega_{d-1}\omega_d\dd\omega\right)\sum_{1\le i<j\le d}(g_i-g_0)\cdot(g_j-g_0).
    \end{split}
  \end{align}
  We calculate the integrals, using Fubini's theorem.
  First integral:
  \begin{align*}
    \aint_{\stdtri}\omega_d\dd\omega
    & = \frac1{|\stdtri^d|}\int_0^1 \omega_d\,(1-\omega_d)^{d-1}|\stdtri^{d-1}|\dd\omega_d \\
    & = \frac{|\stdtri^{d-1}|}{|\stdtri^d|}\int_0^1 (1-z)\,z^{d-1}\dd z
    = d\left(\frac1{d}-\frac1{d+1}\right) = \frac1{d+1}.
  \end{align*}
  Second integral:
  \begin{align*}
    \aint_{\stdtri}\omega_d^2\dd\omega
    & = \frac1{|\stdtri^d|}\int_0^1 \omega_d^2\,(1-\omega_d)^{d-1}|\stdtri^{d-1}|\dd\omega_d \\
    & = \frac{|\stdtri^{d-1}|}{|\stdtri^d|}\int_0^1 (1-z)^2\,z^{d-1}\dd z
    = d\left(\frac1d-\frac2{d+1}+\frac1{d+2}\right) = \frac2{(d+1)(d+2)}.
  \end{align*}
  Third integral:
  \begin{align*}
    \aint_{\stdtri}\omega_{d-1}\omega_d\dd\omega
    & = \frac1{|\stdtri^d|}\int_0^1 \left[\int_0^{1-\omega_d}\omega_{d-1}\omega_d
      \,(1-\omega_{d-1}-\omega_d)^{d-2}|\stdtri^{d-2}|\dd\omega_{d-1}\right]\dd\omega_d \\
    & = \frac{|\stdtri^{d-2}|}{|\stdtri^d|}\int_0^1\left[\int_0^z (1-z)(z-y)\,y^{d-2}\dd y\right]\dd z \\
    & = d(d-1)\int_0^1\left[\frac1{d-1}-\frac1d\right](1-z)z^d\dd z
    = \frac1{d+1}-\frac1{d+2}=\frac1{(d+1)(d+2)}.
  \end{align*}
  Substitute this into~\eqref{eq:triint1}:
  \begin{align*}
    &\left(1-\frac2{d+1}+\frac{d^2+d}{(d+1)(d+2)}\right)\|g_0\|^2
    +\left(\frac2{d+1}-\frac{2d+2}{(d+1)(d+2)}\right)\sum_{1\le j\le d}g_0\cdot g_j \\
    & \qquad +\frac2{(d+1)(d+2)}\sum_{1\le j\le d}\|g_j\|^2
    +\frac2{(d+1)(d+2)}\sum_{1\le i<j\le d}g_i\cdot g_j\\
    &=\frac2{(d+1)(d+2)}\left(\|g_0\|^2 + \sum_{1\le j\le d}g_0\cdot g_j + \sum_{1\le j\le d}\|g_j\|^2 + \sum_{1\le i<j\le d}g_i\cdot g_j \right).
  \end{align*}
  Collecting terms yields the right-hand side of~\eqref{eq:triint}.
\end{proof}
\begin{lemma}
  \label{lem:JAJ}
  For each $A\in\R^{2\times 2}$, we have $\jm A\jm^T=(\det A)\,A^{-T}$.
\end{lemma}
\begin{proof}
  This is verified by direct calculation:
  \[
    \jm A\jm^T
    = \begin{pmatrix} 0 & -1 \\ 1 & 0 \end{pmatrix}
    \begin{pmatrix} a_{11} & a_{12} \\ a_{21} & a_{22} \end{pmatrix}
    \begin{pmatrix} 0 & 1 \\ -1 & 0 \end{pmatrix}
    = \begin{pmatrix} a_{22} & -a_{21} \\ -a_{12} & a_{11} \end{pmatrix}
    = (\det A)\, A^{-T}. \qedhere
  \]
\end{proof}
\begin{lemma}
  \label{lem:circle}
  With $\sigma_k\in\R^2$ defined as in~\eqref{eq:xik},
  we have that
  \begin{align*}
    \sum_{k=0}^5\jm(\sigma_{k}-\sigma_{k+1}) \left(\frac{\sigma_k+\sigma_{k+1}}3\right)^T
    = \sqrt3\ \idm.
  \end{align*}
\end{lemma}
\begin{proof}
  With the abbreviations $\phi_x=\frac\pi3x$ and $\psi=\frac\pi3$:
  \begin{align*}
    \sum_{k=0}^5\jm(\sigma_{k}-\sigma_{k+1}) \left(\frac{\sigma_k+\sigma_{k+1}}3\right)^T
    &= \frac13\sum_{k=0}^5
      {\sin\phi_{k+1}-\sin\phi_k\choose \cos\phi_k-\cos\phi_{k+1}}{\cos\phi_k+\cos\phi_{k+1}\choose\sin\phi_k+\sin\phi_{k+1}}^T \\
    &= \frac13\sum_{k=0}^5(2\sin\frac\psi2){\cos\phi_{k+\frac12}\choose \sin\phi_{k+\frac12}}
      \,(2\cos\frac\psi2){\cos\phi_{k+\frac12}\choose\sin\phi_{k+\frac12}}^T \\
    &= \frac{\sin\psi}3\sum_{k=0}^5
      \begin{pmatrix}
        2\cos^2\phi_{k+\frac12} & 2\cos\phi_{k+\frac12}\sin\phi_{k+\frac12} \\
        2\cos\phi_{k+\frac12}\sin\phi_{k+\frac12} & 2\sin^2\phi_{k+\frac12}
      \end{pmatrix}
    \\
    &= \frac{\sqrt3}6 \sum_{k=0}^5\left[\idm+
      \begin{pmatrix}
        \cos\phi_{2k+1} & \sin\phi_{2k+1} \\
        \sin\phi_{2k+1} & -\cos\phi_{2k+1}
      \end{pmatrix}
                          \right]
                          = \sqrt3\ \idm. \qedhere
  \end{align*}
\end{proof}
\begin{lemma}
  \label{lem:traces}
  Let the scheme $B:=(b_{pqr})_{p,q,r\in\{1,2\}}\in\R^{2\times2\times2}$ of eight numbers $b_{pqr}\in\R$
  be symmetric in the last two indices, $b_{pqr}=b_{prq}$.
  With $\sigma_k\in\R^2$ defined as in~\eqref{eq:xik},
  we have that
  \begin{align}
    \label{eq:traces}
    \sum_{k=0}^5 \tr\big[\big(\sigma_k\big|\sigma_{k+1}\big)^{-1}\big(B:[\sigma_k]^2\big|B:[\sigma_{k+1}]^2\big)\big]\,\jm(\sigma_{k}-\sigma_{k+1})
    =2\sqrt3\tr_{12}[B]^T.
  \end{align}
\end{lemma}
\begin{proof}
  In principle, this lemma can be verified by a direct calculation,
  by writing out the six terms in the sum explicitly and using trigonometric identities.
  Below, we give a slightly more conceptual proof,
  in which we use symmetry arguments to reduce the number of expressions significantly.

  For the matrix involving $B$, we obtain
  \begin{align*}
    &\big(B:[\sigma_k]^2\big|B:[\sigma_{k+1}]^2\big) \\
    &=
    \begin{pmatrix}
      b_{111}\sigma_{k,1}^2+b_{122}\sigma_{k,2}^2+2b_{112}\sigma_{k,1}\sigma_{k,2}
      & b_{111}\sigma_{k+1,1}^2+b_{122}\sigma_{k+1,2}^2+2b_{112}\sigma_{k+1,1}\sigma_{k+1,2} \\
      b_{211}\sigma_{k,1}^2+b_{222}\sigma_{k,2}^2+2b_{212}\sigma_{k,1}\sigma_{k,2}
      & b_{211}\sigma_{k+1,1}^2+b_{222}\sigma_{k+1,2}^2+2b_{212}\sigma_{k+1,1}\sigma_{k+1,2}
    \end{pmatrix},
  \end{align*}
  while clearly
  \begin{align*}
    \big(\sigma_k\big|\sigma_{k+1}\big)^{-1}
    = \frac{2}{\sqrt3}
    \begin{pmatrix} \sigma_{k+1,2} & -\sigma_{k+1,1} \\ -\sigma_{k,2} & \sigma_{k,1} \end{pmatrix}.
  \end{align*}
  The sum of the diagonal entries of the matrix product are easily calculated,
  \begin{align*}
    T_k:=\tr\big[\big(\sigma_k\big|\sigma_{k+1}\big)^{-1}\big(B:[\sigma_k]^2\big|B:[\sigma_{k+1}]^2\big)\big]
    = \frac2{\sqrt3}\sum_{p,q,r=1}^2 b_{pqr}\gamma_{pqr,k},
  \end{align*}
  with the trigonometric expressions
  \begin{align*}
    &\gamma_{111,k}=\sigma_{k,1}^2\sigma_{k+1,2}-\sigma_{k+1,1}^2\sigma_{k,2},\quad
    \gamma_{122,k}=\sigma_{k,2}^2\sigma_{k+1,2}-\sigma_{k+1,2}^2\sigma_{k,2}, \\
    &\gamma_{112,k}=\gamma_{121,k}=\sigma_{k,1}\sigma_{k,2}\sigma_{k+1,2}-\sigma_{k+1,1}\sigma_{k+1,2}\sigma_{k,2}, \\
    &\gamma_{211,k}=\sigma_{k+1,1}^2\sigma_{k,1}-\sigma_{k,1}^2\sigma_{k+1,1},\quad
    \gamma_{222,k}=\sigma_{k+1,2}^2\sigma_{k,1}-\sigma_{k,2}^2\sigma_{k+1,1},\\
    &\gamma_{212,k}=\gamma_{221,k}=\sigma_{k+1,1}\sigma_{k+1,2}\sigma_{k,1}-\sigma_{k,1}\sigma_{k,2}\sigma_{k+1,1}.
  \end{align*}
  To key step is to calculate the sum over $k=0,1,\ldots,5$ of the products of $T_k$
  with the respective vector
  \[ \eta_k =\jm(\sigma_{k}-\sigma_{k+1})=
  \begin{pmatrix}
    \sigma_{k+1,2}-\sigma_{k,2}\\ \sigma_{k,1}-\sigma_{k+1,1}
  \end{pmatrix}
  .
  \]
  Several simplifications of this sum can be performed,
  thanks to the particular form of the $\gamma_{pqr,k}$ and elementary trigonometric identities.
  First, observe that $\sigma_{k+3}=-\sigma_k$, and hence that $\gamma_{pqr,k+3}=-\gamma_{pqr,k}$.
  Since further $\eta_{k+3}=-\eta_k$, it follows that
  \begin{align}
    \label{eq:trig1}
    \gamma_{pqr,k+3}\eta_{k+3}=\gamma_{pqr,k}\eta_k.
  \end{align}
  Second, $\eta$ can be evaluated explicitly for $k=1,2,3$:
  \begin{align}
    \label{eq:trig2}
    \eta_0 =\frac12 \begin{pmatrix} \sqrt3 \\ 1 \end{pmatrix},
    \quad
    \eta_1 = \begin{pmatrix} 0 \\ 1 \end{pmatrix},
    \quad
    \eta_2 =\frac12 \begin{pmatrix} -\sqrt3 \\ 1 \end{pmatrix}.
  \end{align}
  Third, since
  $\sigma_{0,1}=-\sigma_{3,1}$ and $\sigma_{1,1}=-\sigma_{2,1}$,
  as well as
  $\sigma_{0,2}=\sigma_{3,2}$ and $\sigma_{1,2}=\sigma_{2,2}$,
  we obtain that
  \begin{align}
    \label{eq:trig3}
    \gamma_{pqr,1} = 0 \quad\text{if $p+q+r$ is odd},
    \quad\text{and}\quad
    \gamma_{pqr,2} = (-1)^{p+q+r}\gamma_{pqr,0}.
  \end{align}
  By putting this together, we arrive at
  \begin{align*}
    \sum_{k=0}^5\gamma_{pqr,k}\eta_k
    &\stackrel{\eqref{eq:trig1}}{=}2\sum_{k=0}^2\gamma_{pqr,k}\eta_k \\
    & \stackrel{\eqref{eq:trig2}}{=}
    \begin{pmatrix}
      \sqrt3\big(\gamma_{pqr,0}-\gamma_{pqr,2}\big) \\
      \gamma_{pqr,0}+2\gamma_{pqr,1}+\gamma_{pqr,2}
    \end{pmatrix} \\
    &\stackrel{\eqref{eq:trig3}}{=}
    \begin{pmatrix}
      \sqrt3\big(1-(-1)^{p+q+r}\big)\gamma_{pqr,0}\\
      \big(1+(-1)^{p+q+r}\big)\big(\gamma_{pqr,0}+\gamma_{pqr,1}\big)
    \end{pmatrix}
    =
    \begin{pmatrix}
      2\sqrt3\,\gamma_{pqr,0}\,(1-\mathfrak{e}_{pqr})\\
      2\big(\gamma_{pqr,0}+\gamma_{pqr,1}\big)\,\mathfrak{e}_{pqr}
    \end{pmatrix}
    ,
  \end{align*}
  where $\mathfrak{e}_{pqr}=1$ if $p+q+r$ is even,
  and $\mathfrak{e}_{pqr}=0$ if $p+q+r$ is odd.
  By elementary computations,
  \begin{align*}
    \begin{array}{llll}
      \text{$p+q+r$ odd, $k=0$}:
      &\gamma_{111,0}=\frac{\sqrt3}2,
      &\gamma_{122,0}=0,
      &\gamma_{212,0}=\gamma_{221,0}=\frac{\sqrt{3}}4;\\
      \text{$p+q+r$ even, $k=0$}:
      &\gamma_{211,0}=-\frac14,
      &\gamma_{222,0}=\frac34,
      &\gamma_{112,0}=\gamma_{121,0}=0;\\
      \text{$p+q+r$ even, $k=1$}:
      &\gamma_{211,1}=\frac14,
      &\gamma_{222,1}=\frac34,
      &\gamma_{112,1}=\gamma_{121,1}=\frac34,
    \end{array}
  \end{align*}
  and so the final result is:
  \begin{align*}
    &\sum_{k=0}^5 \tr\big[\big(\sigma_k\big|\sigma_{k+1}\big)^{-1}\big(B:[\sigma_k]^2\big|B:[\sigma_{k+1}]^2\big)\big]\,\jm(\sigma_{k}-\sigma_{k+1}) \\
    &=\sum_{k=0}^5T_k\eta_k
    =\frac2{\sqrt3}\sum_{p,q,r=1}^2\left(b_{pqr}\sum_{k=0}^5\gamma_{pqr,k}\eta_k\right)
    =2\sqrt3
    \begin{pmatrix}
      b_{111}+b_{212} \\ b_{222}+b_{112}
    \end{pmatrix},
  \end{align*}
  which is~\eqref{eq:traces}.
\end{proof}
\begin{lemma}
  \label{lem:algebra2}
  With $\sigma_k'\in\R^2$ defined as in~\eqref{eq:newsigma},
  and with $B=(b_{pqr})_{p,q,r\in\{1,2\}}\in\R^{2\times2\times2}$
  such that $b_{pqr}=0$ except for $b_{122}=b_{211}=1$,
  we have that
  \begin{align}
    \label{eq:traces2}
    \sum_{k=0}^5 \tr\big[\big(\sigma_k'\big|\sigma_{k+1}'\big)^{-1}\big(B:[\sigma_k']^2\big|B:[\sigma_{k+1}']^2\big)\big]\,\jm(\sigma_{k}-\sigma_{k+1})
    =-{1\choose1}.
  \end{align}
\end{lemma}
\begin{proof}
  This is a slightly tedious, but straightforward calculation.
  First, by the choice of $B$,
  \begin{align*}
    \beta_k:=\big(B:[\sigma_k']^2\big|B:[\sigma_{k+1}']^2\big)
    = \begin{pmatrix}
      (\sigma_{k,2}')^2 & (\sigma_{k+1,2}')^2 \\ (\sigma_{k,1}')^2 & (\sigma_{k+1,1}')^2
    \end{pmatrix},
  \end{align*}
  and so, by definition of the $\sigma_k'$ in~\eqref{eq:newsigma},
  \begin{align*}
    \beta_0 =\beta_3
    = \begin{pmatrix} 0 & \frac14 \\  1 & \frac14 \end{pmatrix},
                                          \quad
                                          \beta_0=\beta_3
                                          = \begin{pmatrix} \frac14 & 1 \\ \frac14 & 0 \end{pmatrix},
                                                                                     \quad
                                                                                     \beta_0=\beta_3
                                                                                     = \begin{pmatrix} 1 & 0 \\ 0 & 1 \end{pmatrix}.
  \end{align*}
  For the inverse matrices $S_k:=\big(\sigma_k'\big|\sigma_{k+1}'\big)^{-1}$, we obtain
  \begin{align*}
    S_0 = \begin{pmatrix}  1&-1 \\ 0 & 2 \end{pmatrix} = -S_3,
                                       \quad
                                   S_1 = \begin{pmatrix}  2&0 \\ -1&1 \end{pmatrix} = -S_4,
                                                                     S_2 = \begin{pmatrix}  0&1 \\ -1&0 \end{pmatrix} = -S_5.
  \end{align*}
  For the traces $T_k:=\tr\big[S_k\beta_k\big]$, we thus obtain the values:
  \begin{align*}
    T_0=T_1=-\frac12,\quad T_3=T_4=\frac12, \quad T_2=T_5=0.
  \end{align*}
  In conclusion,
  \begin{align*}
    \sum_{k=0}^5 T_k\,\jm(\sigma_{k}-\sigma_{k+1})
    =  \jm\left[-\frac12 (\sigma_0-\sigma_2) +\frac12 (\sigma_3-\sigma_5) \right]
    = \jm{-1\choose1}
    = -{1\choose1},
  \end{align*}
  which is~\eqref{eq:traces2}.
\end{proof}

\section{Lack of convexity}
\label{sct:notconvex}
Below, we discuss why the minimization problem~\eqref{eq:mini} is not convex.
More precisely, we show that $G\mapsto\ent_\disc(G;\hat G)$ is not convex
as a function of $G$ on the affine ansatz space $\ansatz$.
Since $\ent_\disc(G;\hat G)$ is a convex combination
of the expressions $\entdens_m\big((A_m|b_m);(\hat A_m|\hat b_m)\big)$,
it clearly suffices to discuss the convexity of the latter.

We consider a curve $s\mapsto (A_m+s\alpha_m|b_m+s\beta_m)$
and evaluate the second derivatives of the components of the functional at $s=0$.
First,
\begin{align*}
  \I:=\frac{\dn^{2}}{\dn s^{2}}\bigg|_{s=0} &
  \left(\frac1{2\tau}\aint_{\Delta_m} \big|(A_{m}-\hat A_m + s \alpha_{m})\omega + (b_{m}-\hat b_m) + s \beta_{m}\big|^2\dd\omega\right) \\
  &= \frac1\tau\aint_{\Delta_m}|\alpha_m\omega+\beta_m|^2\dd\omega.
\end{align*}
Second,
\begin{align*}
  \II:=\frac{\dn^{2}}{\dn s^{2}}\bigg|_{s=0} & \aint_{\Delta_m} V\big((A_{m} + s \alpha_{m})\omega+ (b_{m} + s \beta_{m})\big)\dd\omega \\
  &= \aint_{\Delta_{m}} (\alpha_{m} \omega + \beta_{m})^{T}\cdot\Grad^{2} V(A_{m} \omega + b_{m})\cdot(\alpha_{m} \omega + \beta_{m}) \dd \omega .
\end{align*}
If we assume that $\Grad^{2} V \geq \lambda \idm$,
then we obtain for the sum of these two contributions that
\begin{align*}
  \I+\II\ge \left(\frac1\tau+\lambda\right)\aint_{\Delta_m}|\alpha_m\omega+\beta_m|^2\dd\omega.
\end{align*}
For the remaining term, however, we obtain
--- using the abbreviations $\widetilde g(s)=s\widetilde h'(s)$ and $\widetilde f(s)=s\widetilde g'(s)$ ---
that
\begin{align*}
  \frac{\dn^2}{\dn^2 s}\bigg|_{s=0}& \widetilde h\left(\frac{\det(A_{m} + s \alpha_{m})}{\refrho_m}\right) \\
  &=\frac{\dn}{\dn s}\bigg|_{s=0} \left\{
    \widetilde g\left(\frac{\det(A_{m} + s \alpha_{m})}{\refrho_m}\right)
    \,\tr\big[(A_{m} + s \alpha_{m})^{-1} \alpha_{m}\big]
  \right\} \\
  &=\widetilde f\left(\frac{\det A_{m}}{\refrho_m}\right)\,\big(\tr\big[A_m^{-1}\alpha_m\big]\big)^2
  - \widetilde g\left(\frac{\det A_{m}}{\refrho_m}\right)\,\tr\big[\big( A_m^{-1}\alpha_m \big)^2\big] .
\end{align*}
Now observe that $\widetilde f(s)=P'(1/s)-sP(1/s)$ is a non-negative,
and $\widetilde g(s) = -sP(1/s)$ is a non-positive function.
Thus, from the two terms in the final sum,
the first one is generally non-negative whereas the second one is of indefinite sign.
Choosing
\[
\alpha_m:=A_m \begin{pmatrix} 0 & 1 \\ 1 & 0 \end{pmatrix},
\quad\text{such that}\quad
\big(\tr\big[A_m^{-1}\alpha_m\big]\big)^2=0, \,
\tr\big[\big( A_m^{-1}\alpha_m \big)^2\big]=2,
\]
the sum is obviously negative.

\addcontentsline{toc}{section}{Bibliography}

\bibliographystyle{abbrv}
\bibliography{CDMM_lagrangian}

\begin{thebibliography}{10}

\bibitem{book:AGS}
L.~Ambrosio, N.~Gigli, and G.~Savar{\'e}.
\newblock {\em Gradient flows in metric spaces and in the space of probability
  measures}.
\newblock Lectures in Mathematics ETH Z\"urich. Birkh\"auser Verlag, Basel,
  second edition, 2008.

\bibitem{art:ALS}
L.~Ambrosio, S.~Lisini, and G.~Savar{\'e}.
\newblock Stability of flows associated to gradient vector fields and
  convergence of iterated transport maps.
\newblock {\em manuscripta mathematica}, 121(1):1--50, 2006.

\bibitem{art:BCMO}
J.-D. Benamou, G.~Carlier, Q.~M{\'e}rigot, and E.~Oudet.
\newblock Discretization of functionals involving the monge--amp{\`e}re
  operator.
\newblock {\em Numerische Mathematik}, pages 1--26, 2015.

\bibitem{Filbet}
M.~Bessemoulin-Chatard and F.~Filbet.
\newblock A finite volume scheme for nonlinear degenerate parabolic equations.
\newblock {\em SIAM J. Sci. Comput.}, 34(5):B559--B583, 2012.

\bibitem{Blanchet}
A.~Blanchet, V.~Calvez, and J.~A. Carrillo.
\newblock Convergence of the mass-transport steepest descent scheme for the
  sub-critical {P}atlak-{K}eller-{S}egel model.
\newblock {\em SIAM J. Numer. Anal.}, 46(2):691--721, 2008.

\bibitem{art:BCC}
A.~Blanchet, V.~Calvez, and J.~A. Carrillo.
\newblock Convergence of the mass-transport steepest descent scheme for the
  subcritical {P}atlak-{K}eller-{S}egel model.
\newblock {\em SIAM J. Numer. Anal.}, 46(2):691--721, 2008.

\bibitem{art:Budd}
C.~J. Budd, G.~J. Collins, W.~Z. Huang, and R.~D. Russell.
\newblock Self-similar numerical solutions of the porous-medium equation using
  moving mesh methods.
\newblock {\em R. Soc. Lond. Philos. Trans. Ser. A Math. Phys. Eng. Sci.},
  357(1754):1047--1077, 1999.

\bibitem{CG}
V.~Calvez and T.~O. Gallou{\"e}t.
\newblock Particle approximation of the one dimensional {K}eller-{S}egel
  equation, stability and rigidity of the blow-up.
\newblock {\em Discrete Contin. Dyn. Syst. Ser. A}, 36(3):1175--1208, 2015.

\bibitem{CCH2}
J.~A. Carrillo, A.~Chertock, and Y.~Huang.
\newblock A finite-volume method for nonlinear nonlocal equations with a
  gradient flow structure.
\newblock {\em Communications in Computational Physics}, 17(01):233--258, 2015.

\bibitem{CHPW}
J.~A. Carrillo, Y.~Huang, F.~S. Patacchini, and G.~Wolansky.
\newblock Numerical study of a particle method for gradient flows.
\newblock {\em Kinet. Relat. Models}, 10(3):613--641, 2017.

\bibitem{art:CaLi10}
J.~A. Carrillo and S.~Lisini.
\newblock On the asymptotic behavior of the gradient flow of a polyconvex
  functional.
\newblock In {\em Nonlinear partial differential equations and hyperbolic wave
  phenomena}, volume 526 of {\em Contemp. Math.}, pages 37--51. Amer. Math.
  Soc., Providence, RI, 2010.

\bibitem{art:CaMo09}
J.~A. Carrillo and J.~S. Moll.
\newblock Numerical simulation of diffusive and aggregation phenomena in
  nonlinear continuity equations by evolving diffeomorphisms.
\newblock {\em SIAM J. Sci. Comput.}, 31(6):4305--4329, 2009/10.

\bibitem{CPSW}
J.~A. Carrillo, F.~S. Patacchini, P.~Sternberg, and G.~Wolansky.
\newblock Convergence of a particle method for diffusive gradient flows in one
  dimension.
\newblock {\em SIAM J. Math. Anal.}, 48(6):3708--3741, 2016.

\bibitem{art:CRW}
J.~A. Carrillo, H.~Ranetbauer, and M.-T. Wolfram.
\newblock Numerical simulation of nonlinear continuity equations by evolving
  diffeomorphisms.
\newblock {\em J. Comput. Phys.}, 327:186--202, 2016.

\bibitem{art:DS}
S.~Daneri and G.~Savar{\'e}.
\newblock Eulerian calculus for the displacement convexity in the {W}asserstein
  distance.
\newblock {\em SIAM J. Math. Anal.}, 40(3):1104--1122, 2008.

\bibitem{DM}
P.~Degond and F.-J. Mustieles.
\newblock A deterministic approximation of diffusion equations using particles.
\newblock {\em SIAM J. Sci. Statist. Comput.}, 11(2):293--310, 1990.

\bibitem{during2010gradient}
B.~D{\"u}ring, D.~Matthes, and J.~P. Mili{\v{s}}ic.
\newblock A gradient flow scheme for nonlinear fourth order equations.
\newblock {\em Discrete Contin. Dyn. Syst. Ser. B}, 14(3):935--959, 2010.

\bibitem{art:EGS}
L.~Evans, O.~Savin, and W.~Gangbo.
\newblock Diffeomorphisms and nonlinear heat flows.
\newblock {\em SIAM journal on mathematical analysis}, 37(3):737--751, 2005.

\bibitem{art:GT}
L.~Gosse and G.~Toscani.
\newblock Identification of asymptotic decay to self-similarity for
  one-dimensional filtration equations.
\newblock {\em SIAM J. Numer. Anal.}, 43(6):2590--2606 (electronic), 2006.

\bibitem{GT}
L.~Gosse and G.~Toscani.
\newblock Lagrangian numerical approximations to one-dimensional
  convolution-diffusion equations.
\newblock {\em SIAM J. Sci. Comput.}, 28(4):1203--1227 (electronic), 2006.

\bibitem{art:JKO}
R.~Jordan, D.~Kinderlehrer, and F.~Otto.
\newblock The variational formulation of the {F}okker-{P}lanck equation.
\newblock {\em SIAM J. Math. Anal.}, 29(1):1--17, 1998.

\bibitem{art:JMO}
O.~Junge, D.~Matthes, and H.~Osberger.
\newblock A fully discrete variational scheme for solving nonlinear
  fokker-planck equations in higher space dimensions.
\newblock {\em arXiv preprint arXiv:1509.07721}, 2015.

\bibitem{LM}
P.-L. Lions and S.~Mas-Gallic.
\newblock Une m\'ethode particulaire d\'eterministe pour des \'equations
  diffusives non lin\'eaires.
\newblock {\em C. R. Acad. Sci. Paris S\'er. I Math.}, 332(4):369--376, 2001.

\bibitem{MGallic}
S.~Mas-Gallic.
\newblock The diffusion velocity method: a deterministic way of moving the
  nodes for solving diffusion equations.
\newblock {\em Transport Theory and Statistical Physics}, 31(4-6):595--605,
  2002.

\bibitem{art:MO1}
D.~Matthes and H.~Osberger.
\newblock Convergence of a variational {L}agrangian scheme for a nonlinear
  drift diffusion equation.
\newblock {\em ESAIM Math. Model. Numer. Anal.}, 48(3):697--726, 2014.

\bibitem{art:McCann}
R.~J. McCann.
\newblock A convexity principle for interacting gases.
\newblock {\em Adv. Math.}, 128(1):153--179, 1997.

\bibitem{MO3}
H.~Osberger and D.~Matthes.
\newblock Convergence of a variational {L}agrangian scheme for a nonlinear
  drift diffusion equation.
\newblock {\em ESAIM Math. Model. Numer. Anal.}, 48(3):697--726, 2014.

\bibitem{MO1}
H.~Osberger and D.~Matthes.
\newblock Convergence of a fully discrete variational scheme for a thin-film
  equation.
\newblock {\em Accepted at Radon Ser. Comput. Appl. Math.}, 2015.

\bibitem{MO2}
H.~Osberger and D.~Matthes.
\newblock A convergent {L}agrangian discretization for a nonlinear fourth order
  equation.
\newblock {\em Found. Comput. Math.}, pages 1--54, 2015.

\bibitem{art:Otto}
F.~Otto.
\newblock The geometry of dissipative evolution equations: the porous medium
  equation.
\newblock {\em Comm. Partial Differential Equations}, 26(1-2):101--174, 2001.

\bibitem{art:Peyre}
G.~Peyr{\'e}.
\newblock Entropic approximation of {W}asserstein gradient flows.
\newblock {\em SIAM J. Imaging Sci.}, 8(4):2323--2351, 2015.

\bibitem{Russo}
G.~Russo.
\newblock Deterministic diffusion of particles.
\newblock {\em Comm. Pure Appl. Math.}, 43(6):697--733, 1990.

\bibitem{book:Vazquez}
J.~L. V{\'a}zquez.
\newblock {\em The porous medium equation: mathematical theory}.
\newblock Oxford University Press, 2007.

\bibitem{book:Villani}
C.~Villani.
\newblock {\em Topics in optimal transportation}, volume~58 of {\em Graduate
  Studies in Mathematics}.
\newblock American Mathematical Society, Providence, 2003.

\bibitem{westdickenberg2010variational}
M.~Westdickenberg and J.~Wilkening.
\newblock Variational particle schemes for the porous medium equation and for
  the system of isentropic euler equations.
\newblock {\em ESAIM: Mathematical Modelling and Numerical Analysis},
  44(1):133--166, 2010.

\end{thebibliography}

\end{document}